\documentclass[11pt]{article}
 
\usepackage{amssymb}
\usepackage{amsmath}
\usepackage{amsthm}
\usepackage{geometry} 
\usepackage{hyperref}
\usepackage{parskip}
\usepackage{comment}
\usepackage{enumerate}
\usepackage[toc]{appendix}
\usepackage[scr=rsfso]{mathalfa}
\usepackage{mathtools}

\DeclarePairedDelimiter\floor{\lfloor}{\rfloor}
 
\usepackage{soul,xcolor}

\usepackage{stackengine}
\stackMath
\newcommand\tsup[2][2]{%
 \def\useanchorwidth{T}%
  \ifnum#1>1%
    \stackon[-.5pt]{\tsup[\numexpr#1-1\relax]{#2}}{\scriptscriptstyle\sim}%
  \else%
    \stackon[.5pt]{#2}{\scriptscriptstyle\sim}%
  \fi%
}

\usepackage{chemformula}
 
\usepackage{lineno}

\usepackage{enumerate}

\usepackage{bbm}

\RequirePackage{geometry}
\newtheorem{theorem}{Theorem}[section]
\newtheorem{proposition}[theorem]{Proposition}

\newtheorem{corollary}[theorem]{Corollary}
\newtheorem{definition}[theorem]{Definition}
\newtheorem{hypothesis}[theorem]{Hypothesis}
\theoremstyle{remark}
\newtheorem{remark}{Remark}

\newcommand{\augpr}[1]{\widehat{\mathbb{#1}}}

\usepackage[
    backend=biber,
    style=numeric,
    sorting=nyt,
    maxbibnames=10,
  ]{biblatex}
  \renewbibmacro{in:}{%
  \ifentrytype{article}
    {}
    {\bibstring{in}%
     \printunit{\intitlepunct}}}
\usepackage{authblk}

\allowdisplaybreaks
\begin{document}

\title{Signature McKean--Vlasov stochastic differential equations}

\author{Fred Espen Benth\thanks{Department of Data Science and Analytics, BI Norwegian Business School, N-0484 Oslo, Norway \& Department of Mathematics, University of Oslo, PO Box 1053 Blindern, N-0316 Oslo, Norway, \href{mailto:fred.e.benth@bi.no}{fred.e.benth@bi.no}} , Salvador Ortiz-Latorre\thanks{Department of Mathematics, University of Oslo, PO Box 1053 Blindern, N-0316 Oslo, Norway, \href{mailto:salvadoo@math.uio.no}{salvadoo@math.uio.no}} , Leonardo Tarquini\thanks{Department of Mathematics, University of Oslo, PO Box 1053 Blindern, N-0316 Oslo, Norway, \href{mailto:leonardt@math.uio.no}{leonardt@math.uio.no}}}

\maketitle
\begin{abstract} 
     McKean--Vlasov-type stochastic differential equations (SDEs) are characterized by coefficients depending on both the state and the law of the solution. In this work, we focus on a class of such equations where the coefficients depend on a linear combination of the expected signature of the geometric $p$-rough path lift of its solution, with $p\in(2,3)$. After establishing the strong existence and uniqueness of a solution, we prove how such an equation can approximate a general class of path-dependent McKean--Vlasov SDEs. Finally, we consider the associated particle system and propagation of chaos is established.
\end{abstract}



\section{Introduction}\label{sec:introduction}
McKean--Vlasov-type stochastic differential equations (SDEs) have been extensively studied since the seminal works of McKean \cite{McKean1967Propagation} and Sznitman \cite{Sznitman} and find applications across a wide range of fields, including physics \cite{Cattiauxetal,GolseFrancois2016OtMF,Malrieu,arxivTarquiniUgolini1,2025_morale_tarquini_ugolini}, biology \cite{Delarueneurons,TomasevicMilica2021ANMS}, machine learning and deep learning \cite{Carmona2021CAoM,Carmona2022CAOM,MeiSong2018Amfv}, finance, optimal control, and mean-field games \cite{CarmonaDelarue1,CarmonaDelarue2,CarmonaFouqueSun}.

The importance of such equations have naturally stimulated the development of a wide range of numerical methods tailored to McKean--Vlasov dynamics. A first and most classical approach is based on interacting particle systems, where one approximates the law of the solution by the empirical measure of a large system of weakly interacting particles, typically simulated via time-discretization schemes of Euler--Maruyama type; see, for instance \cite{AntonelliFabio2002RoCo,BossyMireille1997ASPM,Meleard2006}.
More recently, several alternative strategies have been proposed. These include analytical expansions to approximate the density of the solution of the McKean--Vlasov SDE \cite{GobetEmmanuel2018Aaon}; approaches based on the Fokker–Planck equation representation of the solution of the McKean--Vlasov SDE, which make it possible to exploit numerical methods for partial differential equations \cite{BaladronJavier2012Mdap,GoddardBD2022Nbcm}; multilevel Picard iterations \cite{HutzenthalerMartin2022MPaf}; cubature methods \cite{ChaudrudeRaynalPE2015Acba}; machine-learning-based techniques \cite{AgarwalAnkush2026NaoM,HanJiequn2018Shpd}.

Our strategy is to introduce a suitable approximating equation based on a finite-dimensional parametrization in the space of probability measures on path space via path signatures. This approach also allows to approximate McKean--Vlasov SDEs with coefficients depending on the full history of the law (see, e.g, \cite{arxivTarquiniUgolini1,2025_morale_tarquini_ugolini,TomasevicMilica2021ANMS}) while preserving the key mean-field and path-dependent characteristics of the original dynamics. In contrast, most existing numerical methods focus on the Markovian case, in which coefficients depend only on the current state and law.

Inspired by the fact that linear combinations of path signatures provide universal approximations of path functionals \cite{CSTGlobalUAT}, and by their high computational efficiency \cite{chevyrev2025primersignaturemethodmachine,ReizensteinJeremyF2020A1TI}, we introduce and study a new class of McKean--Vlasov-type SDEs. The drift and diffusion coefficients of such equations depend on the expectation of a linear functional of the (time-extended) signature of the solution, and we refer to them as \emph{signature McKean--Vlasov SDEs}. Specifically, we consider a $d$-dimensional stochastic process
\begin{equation}\label{eq:sig_MKV_SDE_intro}
    X_t = \eta + \int_0^t b\left(s,X_s,\Big\langle \ell,\mathbb{E}\left[S(\augpr{X})_s\right]\Big\rangle\right)ds + \int_0^t\sigma\left(s,X_s,\Big\langle \widetilde{\ell},\mathbb{E}\left[S(\augpr{X})_s\right]\Big\rangle\right)dW_s,
\end{equation}
where $t\in[0,T]$ and $T>0$ is a finite time horizon, $\eta$ is an $\mathbb{R}^d$-valued random variable, and $W$ a Brownian motion.

In \cite{cuchiero2025signaturesdesaffinepolynomial}, the authors consider SDEs whose coefficients depend on functionals of their own signature and, by exploiting the affine and polynomial structure of the solution together with its signature, derive explicit power-series representations for the Fourier--Laplace transform and for expectations of entire functionals of the signature process. In contrast, in \eqref{eq:sig_MKV_SDE_intro} the (time-augmented) signature enters the coefficients only through its expectation, and we do not work within the framework of affine and polynomial process theory; instead, our main methodological contribution lies in the analytical and probabilistic tools developed to establish well-posedness for \eqref{eq:sig_MKV_SDE_intro}.

Moreover, the class of signature McKean–Vlasov SDEs can be regarded as a generalisation of polynomial McKean–Vlasov SDEs \cite{cuchiero2025polynomialmckeanvlasovsdes}, which are recovered as a special case in the one-dimensional setting. Although our assumptions differ from those in \cite{cuchiero2025polynomialmckeanvlasovsdes}, by employing rough path techniques we establish strong existence and pathwise uniqueness for \eqref{eq:sig_MKV_SDE_intro}, as well as propagation of chaos for the associated particle system (see \eqref{eq:PS_intro} below).

In \eqref{eq:sig_MKV_SDE_intro}, $S(\widehat{\mathbb{X}})$ denotes the signature of the time-augmented process $\widehat{X}_t:=(t,X_t)$ and it is defined for every $0\leq s<t\leq T$, as the formal series of iterated Stratonovich integrals
\begin{equation*}
    S(\augpr{X})_{s,t} := \left(1,\, \int_{s<u_1<t} \circ d\widehat{X}_{u_1},\, \int_{s<u_1<u_2<t} \circ d\widehat{X}_{u_1}\otimes \circ d\widehat{X}_{u_2},\, \ldots \right)\in \prod_{n=0}^{\infty} (\mathbb{R}^{d+1})^{\otimes n},
\end{equation*}
whose $n$-th component is
\begin{equation*}
    S^n(\augpr{X})_{s,t} := \int_{s<u_1<\cdots<u_n<t} \circ d\widehat{X}_{u_1}\otimes \cdots \otimes \circ d\widehat{X}_{u_n}\in (\mathbb{R}^{d+1})^{\otimes n}.
\end{equation*}
If $s=0$ then we just write $S(\augpr{X})_t$.

Then, $\langle\ell,S(\widehat{\mathbb{X}})_t\rangle$ is defined for every $t\in[0,T]$ as
\begin{equation}\label{eq:linear_comb_in_sig_MKV_SDE_intro}
    \langle\ell,S(\widehat{\mathbb{X}})_t\rangle := \sum_{0\leq |I|\leq N} \alpha_I\langle \varepsilon_I,S^{|I|}(\widehat{\mathbb{X}})_t\rangle_{(\mathbb{R}^{d+1})^{\otimes |I|}},
\end{equation}
where $N\in\mathbb{N}$, $\alpha_I\in\mathbb{R}$, $I=\{i_1,\dots,i_m\}\in\{1,\dots,d+1\}^m$ is a multi-index, $|I|=m$, $\varepsilon_I:=\varepsilon_{i_1}\otimes\dots\otimes\varepsilon_{i_m}$ are the basis elements of $(\mathbb{R}^{d+1})^{\otimes m}$\footnote{For a basis $\{e_1,\dots,e_{d+1}\}$ of $\mathbb{R}^{d+1}$ and a multi-index $I=\{i_1,\dots,i_m\}\in\{1,\dots,d+1\}^m$, $\{e_I\}_{|I|=m}$ forms a basis of $\mathbb{R}^{(d+1)\otimes m}$}, and $\langle \cdot,\cdot\rangle_{(\mathbb{R}^{d+1})^{\otimes |I|}}$ is the usual inner product on $(\mathbb{R}^{d+1})^{\otimes |I|}$. $\langle\widetilde{\ell},S(\widehat{\mathbb{X}})_t\rangle$ is defined mutatis mutandis.

The notion of the signature of a path was firstly introduced by Chen \cite{Chen57,Chen77} and plays a central role in Lyons' rough path theory \cite{Lyons98}. In this setting, for a controlled differential equation driven by a rough path, the signature provides an enhancement of the driving signal that makes the solution map continuous with respect to the enhanced path. Interested readers may also refer to the monographs \cite{FrizHairer2020,Friz_Victoir_2010,Lyons_Qian}.

From a more general perspective, the signature can be regarded as a feature map for paths. Universal approximation theorems (UATs) show that continuous functionals of continuous or  even c\`{a}dl\`{a}g paths can be approximated, on compact sets, by linear functionals of the (time-extended) signature \cite{CuchieroChrista2025Uatf}. By working in appropriately weighted path spaces, the compactness assumption can be removed \cite{CSTGlobalUAT}.

Using precisely this last result, we establish a global universal approximation result which states that for every $\varepsilon>0$, under Lipschitz conditions on $b$, $\sigma$, and $\mathcal G$, there exists a linear map $\langle\ell,\cdot\rangle$ on the tensor space  such that
\begin{equation*}
    \mathbb{E}\left[\sup_{0\leq t\leq T} \lVert X_t - Y_t \rVert_{\mathbb{R}^d}\right] < \varepsilon,
\end{equation*}
where $X$ is the solution to the signature McKean--Vlasov SDE \eqref{eq:sig_MKV_SDE_intro} with $\langle\ell,\cdot\rangle$ and $Y$ is the solution to the path-dependent McKean--Vlasov SDE
\begin{equation}\label{eq:path_dep_MKV_intro}
    Y_t = \eta + \int_0^t b\left(s,Y_s,\mathbb{E}\left[\mathcal{G}(Y_{\cdot\wedge s})\right]\right)ds + \int_0^t\sigma \left(s,Y_s,\mathbb{E}\left[\widetilde{\mathcal{G}}(Y_{\cdot\wedge s})\right]\right)dW_s,
\end{equation}
with $t\in[0,T]$, where $\mathcal{G},\widetilde{\mathcal{G}}:C\left([0,T];\mathbb{R}^d\right)\rightarrow\mathbb{R}$ and $Y_{\cdot\wedge s}=\{Y_{t\wedge s}\}_{t\in[0,T]}$ is defined as
\begin{equation*}
    Y_{t\wedge s}=\begin{cases}
        Y_t\quad \textup{if $t\leq s$;}\\
        Y_s\quad \textup{if $s< t\leq T$.}
    \end{cases}
\end{equation*}

The question of extending the approximation results developed in the present work beyond the class of equations \eqref{eq:path_dep_MKV_intro} is left for future investigation. Nevertheless, there are strong indications that the scope of the approach may be substantially broader. In particular, \cite{BelomestnySchoenmakers} shows that a subclass of the path-dependent McKean--Vlasov SDEs considered here can be used to approximate the significantly more general class of McKean--Vlasov diffusions
\begin{equation*}
    \begin{aligned}
        &Z_t = \eta + \int_0^t\int_{\mathbb{R}} b(s,Z_s,y)\Phi(\mu_{\cdot\wedge s})ds + \int_0^t\int_{\mathbb{R}} \sigma(s,Z_s,y)\widetilde{\Phi}(\mu_{\cdot\wedge s})(dy)dW_s;\\
        &\mu_t:=\mathcal{L}(Z_t),
    \end{aligned}
\end{equation*}
with $t\in[0,T]$, where $\mathcal{L}(Z_t)$ denotes the law of $Z_t$, and $\Phi,\widetilde{\Phi}$ are maps from the speace of probability measures over path space to the space of measures over path space.

While this result does not directly imply an extension of the analysis carried out in the present paper, it provides compelling evidence that the signature-based approximation framework may be applicable to a much wider class of McKean--Vlasov dynamics.

Finally, we consider a particle approximation of the signature McKean--Vlasov SDE, where the interacting particle system is used to approximate the expectation term driving the dynamics. We refer to it as \emph{signature mean-field particle system} and the dynamics of the $i$-th particle satisfy the following SDE:
\begin{equation}\label{eq:PS_intro}
   \begin{aligned}
        X_t^{i,N_P} \,=\, \eta \,+ &\int_0^t b\left(s,X_s^{i,N_P},\frac{1}{N_P}\sum_{j=1}^{N_P}\Big\langle \ell,S(\augpr{X}^{j,N_P})_s\Big\rangle\right)ds\\
        &+ \int_0^t \sigma\left(s,X_s^{i,N_P},\frac{1}{N_P}\sum_{j=1}^{N_P}\Big\langle \widetilde{\ell},S(\augpr{X}^{j,N_P})_s\Big\rangle\right)dW^i_s,
   \end{aligned}
\end{equation}
with $t\in[0,T]$, $i=1,\ldots,N_P$, and $N_P\in\mathbb{N}$, where $\{W^i\}_{i=1}^{N_P}$ is a family of $N_P$ independent $m$-dimensional Brownian motions and $\Big\langle \ell,S(\augpr{X}^{j,N_P})_s\Big\rangle$ is defined as in \eqref{eq:linear_comb_in_sig_MKV_SDE_intro} for every $j\in\{1,\ldots,N_P\}$.

By establishing the propagation of chaos property \cite{ChaintronDiez,Sznitman} for the above particle system, we prove that the expectation term appearing in the coefficients of the signature McKean--Vlasov SDE \eqref{eq:sig_MKV_SDE_intro} admits a particle approximation given by the empirical average over the signatures of the geometric $p$-rough lifts of the time-augmented interacting particle system.

Equation \eqref{eq:sig_MKV_SDE_intro} is defined on a filtered probability space $(\Omega,\mathcal{F},\mathbb{F}=\{\mathcal{F}_t\}_{t\ge 0},\mathbb{P})$, which supports an $m$-dimensional $\mathbb{F}$-Brownian motion $\{W_t\}_{0\le t\le T}$. We work under the following natural conditions on the initial condition $\eta$, the drift and diffusion coefficients $b$ and $\sigma$, and $\mathcal{G}$.
\begin{hypothesis}\label{hypothesis_A}
   \begin{enumerate}[i)]
        \item $\eta\in L^2(\Omega)$ is an $\mathcal{F}_0$-random variable, independent of $W$.
        \item The coefficients
        \begin{equation*}
            b:[0,T]\times\mathbb{R}^d\times\mathbb{R}\rightarrow\mathbb{R}^d;\quad \sigma:[0,T]\times\mathbb{R}^d\times\mathbb{R}\rightarrow\mathbb{R}^{d\times m}
        \end{equation*}
        satisfy boundedness and Lipschitz continuity assumptions. Precisely, for every $s,t\in[0,T]$, $x,x^{\prime}\in\mathbb{R}^d$, and $y,y^{\prime}\in \mathbb{R}$, there exist the constants $M_b,M_\sigma,L>0$ such that
        \begin{equation*}
            \big|b(t,x,y)\big|\leq M_b\,;\qquad \big|\sigma(t,x,y)\big|\leq M_{\sigma},
        \end{equation*}
        and
        \begin{equation*}
            \begin{aligned}
                &\lVert b(t,x,y)-b(s,x^{\prime},y^{\prime})\rVert_{\mathbb{R}^d} + \lVert\sigma(t,x,y)-\sigma(s,x^{\prime},y^{\prime})\rVert_F\\
                &\leq L\left(|t-s| + \lVert x-x^{\prime}\rVert_{\mathbb{R}^d}+ |y-y^{\prime}|\right),
            \end{aligned}
        \end{equation*}
        where $\lVert\cdot\rVert_F$ is the Frobenius norm.
    \end{enumerate}
\end{hypothesis}

\begin{hypothesis}\label{hypothesis_B}
    For every $z^1,z^2\in C\left([0,T];\mathbb{R}^d\right)$, there exists two constants $L_{\mathcal{G}},L_{\widetilde{\mathcal{G}}}>0$ such that
    \begin{equation*}
        \big| \mathcal{G}(z^1_{\cdot\wedge t})-\mathcal{G}(z^2_{\cdot\wedge t}) \big| \leq L_{\mathcal{G}}\big\lVert z_1-z_2\big\rVert_{\infty,t};\qquad \big| \widetilde{\mathcal{G}}(z^1_{\cdot\wedge t})-\widetilde{\mathcal{G}}(z^2_{\cdot\wedge t}) \big| \leq L_{\widetilde{\mathcal{G}}}\big\lVert z_1-z_2\big\rVert_{\infty,t}
    \end{equation*}
    for every $t\in[0,T]$, where $\lVert\cdot\rVert_{\infty,t}$ is the supremum norm over $[0,t]$.
\end{hypothesis}

The paper is organised as follows. In Section \ref{sec:preliminaries}, we collect the main notions and results regarding rough path theory that will be used throughout the paper; we will mainly follow \cite[Chapters 7--9,14]{Friz_Victoir_2010}. Section \ref{sec:well_posed_path_dep_MKV} contains the proof of strong existence of a pathwise unique solution to the path-dependent McKean--Vlasov equation \eqref{eq:path_dep_MKV_intro}. Then, in Section \ref{sec:well_posed_path_sig_MKV}, it is established the well-posedness for the signature McKean--Vlasov equation \eqref{eq:sig_MKV_SDE_intro}. Section \ref{sec:approximation_result} is dedicated to the aforementioned approximation result. In Section \ref{sec:PS_prop_of_chaos}, we introduce the particle system \eqref{eq:PS_intro} associated with the signature McKean--Vlasov SDE \eqref{eq:sig_MKV_SDE_intro}; after proving well-posedness, propagation of chaos is established. Finally, in Appendix \ref{appendix:proofs}, we collect the proofs of the more technical results.

\paragraph{Notation.} All stochastic processes throughout the paper are assumed to be defined over the filtered probability space $(\Omega, \mathcal{F},\mathbb{F}=\{\mathcal{F}_t\}_{t\in [0,T]},\mathbb{P})$ and adapted to $\mathbb{F}$. Moreover, for any stochastic process $X$, $\mathcal{L}(X)$ denotes its law.

For any Banach space $V$, we denote with $\mathcal{B}V$ the Borel $\sigma$-algebra on $V$.

For every $s,t\in[0,T]$, $s<t$, $D_{[s,t]}$ indicates a partition of $[s,t]$; precisely,
\begin{equation*}
    D_{[s,t]}:=\big\{s=t_0<t_1<\dots<t_i<\dots<t_n=t\big\}.
\end{equation*}

\section{Preliminaries}\label{sec:preliminaries}
Let $k\in\mathbb{N}$. Recall that the \emph{extended tensor algebra} over $\mathbb{R}^k$ is defined as
\begin{equation*}
    (T(\mathbb{R}^k)) := \prod_{n=0}^{\infty} (\mathbb{R}^k)^{\otimes n},
\end{equation*}
where $(\mathbb{R}^k)^{\otimes n}$ is the $n$-fold tensor product of $\mathbb{R}^k$ with the convention $(\mathbb{R}^k)^{\otimes 0}:=\mathbb{R}$.

Also, for $N\in\mathbb{N}$, we define the \emph{truncated tensor algebra}
\begin{equation*}
    T^N(\mathbb{R}^k) := \bigoplus_{n=0}^N (\mathbb{R}^k)^{\otimes n}.
\end{equation*}

\subsection{The free step-$N$ nilpotent Lie group}
The elements $\mathbf{t}\in T^N(\mathbb{R}^k)$ will be denoted by
\begin{equation*}
    \mathbf{t} = (\mathbf{t}^m)_{m=0}^N \in T^N(\mathbb{R}^k),
\end{equation*}
where $\mathbf{t}^m \in (\mathbb{R}^d)^{\otimes m}$ for every $m=0,\dots,N$.

The space $T^N(\mathbb{R}^k)$ becomes a linear space by defining for every $\mathbf{s},\mathbf{t}\in T^N(\mathbb{R}^k)$ and $\lambda\in\mathbb{R}$,
\begin{equation*}
    \mathbf{s} + \mathbf{t} := \left(\mathbf{s}^0 + \mathbf{t}^0,\mathbf{s}^1 + \mathbf{t}^1,\dots,\mathbf{s}^N + \mathbf{t}^N\right),\quad \lambda\mathbf{s} := \left(\lambda\mathbf{s}^0,\lambda\mathbf{s}^1,\dots,\lambda\mathbf{s}^N\right).
\end{equation*}
Also, it is possible to give the following definition of a product between two elements of $T^N(\mathbb{R}^k)$, called \emph{tensor convolution product}: for every $\mathbf{s}=\{\mathbf{s}^i\}_i,\mathbf{t}=\{\mathbf{t}^j\}_j\in T^N(\mathbb{R}^k)$, $\mathbf{z}:=\mathbf{s}\bullet\mathbf{t}\in T^N(\mathbb{R}^k)$ is defined as
\begin{equation*}
    \mathbf{z}^k := \sum_{i+j=\ell} \mathbf{s}^i\otimes\mathbf{t}^j \in (\mathbb{R}^k )^{\otimes \ell},\quad \ell=0,\dots,N,
\end{equation*}
where $\otimes$ denotes the tensor product. Hence, $\left(T^N(\mathbb{R}^k),+,\cdot,\bullet\right)$ is an associative algebra with unit element $(1,0,\dots,0)\in T^N(\mathbb{R}^k)$.

Finally, $T^N(\mathbb{R}^k)$ can be endowed with the norm
\begin{equation*}
    \lVert \mathbf{t}\rVert_{T^N(\mathbb{R}^k)} := \max_{\ell=0,\dots,N} \lVert\mathbf{t}^k\rVert_{(\mathbb{R}^k)^{\otimes \ell}}.
\end{equation*}

The subsets
\begin{equation*}
    \mathfrak{t}^N(\mathbb{R}^k) := \{\mathbf{t}\in T^N(\mathbb{R}^k) \,:\, \mathbf{t}^0 = 0 \}
\end{equation*}
and
\begin{equation*}
    1+\mathfrak{t}^N(\mathbb{R}^k) := \{\mathbf{t}\in T^N(\mathbb{R}^k) \,:\, \mathbf{t}^0 = 1 \}
\end{equation*}
of $T^N(\mathbb{R}^k)$ are of particular relevance. Precisely, $1+\mathfrak{t}^N(\mathbb{R}^k)$ is a Lie group under the tensor convolution product $\bullet$, with unit element $(1,0,\dots,0)\in 1+\mathfrak{t}^N(\mathbb{R}^k)$; on the other hand, $(\mathfrak{t}^N(\mathbb{R}^k),+,\cdot,[\cdot,\cdot])$ is a Lie algebra, where
\begin{equation*}
    [\mathbf{s},\mathbf{t}] := \mathbf{s}\bullet\mathbf{t} - \mathbf{t}\bullet\mathbf{s} \in\mathfrak{t}^N(\mathbb{R}^k)\qquad \forall\,\mathbf{s},\mathbf{t}\in\mathfrak{t}^N(\mathbb{R}^k).
\end{equation*}
We also define the \emph{free step-$N$nilpotent Lie algebra}
\begin{equation*}
    \mathfrak{g}^N(\mathbb{R}^k) := \{0\}\oplus\mathbb{R}^k \oplus [\mathbb{R}^k,\mathbb{R}^k] \oplus\cdots\oplus \underbrace{[\mathbb{R}^k,[\mathbb{R}^k,\dots,[\mathbb{R}^k,\mathbb{R}^k]]]}_{\textit{$(N-1)$ brackets}},
\end{equation*}
where
\begin{equation*}
    [\underline{\mathbf{g}},\mathbf{h}] := \underline{\mathbf{g}}\otimes\mathbf{h} - \mathbf{h}\otimes\underline{\mathbf{g}}
\end{equation*}
for every $\underline{\mathbf{g}}\in\mathbb{R}^k$ and $\mathbf{h}\in \mathfrak{t}^M(\mathbb{R}^k)$ with $1\leq M\leq N-1$.

Now, let us introduce the \emph{exponential map}
\begin{equation*}
    \begin{aligned}
        \exp : \mathfrak{t}^N(\mathbb{R}^k) &\rightarrow 1 + \mathfrak{t}^N(\mathbb{R}^k)\\
        \mathbf{t} &\mapsto 1 + \sum_{k=1}^N \frac{\mathbf{t}^{\bullet k}}{k!},
    \end{aligned}
\end{equation*}
whose inverse is the \emph{logarithm map}
\begin{equation*}
    \begin{aligned}
        \log : 1+\mathfrak{t}^N(\mathbb{R}^k) &\rightarrow \mathfrak{t}^N(\mathbb{R}^k)\\
        1+\mathbf{t} &\mapsto \sum_{k=1}^N (-1)^{k+1}\frac{\mathbf{t}^{\bullet k}}{k!},
    \end{aligned}
\end{equation*}
and the so-called \emph{free step-$N$ nilpotent Lie group}
\begin{equation*}
    G^N(\mathbb{R}^k) := \exp(\mathfrak{g}^N(\mathbb{R}^k)) \subset 1+\mathfrak{t}^N(\mathbb{R}^k).
\end{equation*}

$G^N(\mathbb{R}^k)$ can be endowed with the \emph{Carnot-Carath\'{e}odory norm} $\lVert\cdot\rVert_{cc}$, which is defined as follows: for every $\mathbf{g}\in G^N(\mathbb{R}^k)$,
\begin{equation*}
    \lVert\mathbf{g}\rVert_{cc} := \inf\Bigg\{\int_0^T \lVert d\gamma_t\rVert \,:\, \gamma\in C^{1-var}([0,T],\mathbb{R}^k)\quad\textit{and}\quad S(\gamma)_T^{\leq N}=\mathbf{g}\Bigg\},
\end{equation*}
where $C^{1-var}([0,T],\mathbb{R}^k)$ denotes the space of continuous paths of finite variation from $[0,T]$ to $\mathbb{R}^k$, with $0<T<+\infty$, and 
\begin{equation}\label{eq:signature_C_1_path}
    S(\gamma)_T^{\leq N} := \left(1,\int_{0<u_1<T}d\gamma_u,\dots,\int_{0<u_1<\dots,u_N<T}d\gamma_{u_1}\otimes\dots\otimes d\gamma_{u_N}\right)\in T^N(\mathbb{R}^k)
\end{equation}
is the \emph{step-N signature} of $\gamma$.

Furthermore, the Carnot-Carath\'{e}odory norm induces the metric $d_{cc}$ defined as
\begin{equation*}
    d_{cc}(\mathbf{g},\mathbf{h}) := \lVert \mathbf{g}^{-1}\bullet\mathbf{h} \rVert_{cc}\qquad \forall\,\mathbf{g},\mathbf{h}\in G^N(\mathbb{R}^k)
\end{equation*}
and $(G^N(\mathbb{R}^k),d_{cc})$ is a metric space.

\subsection{(Weakly) geometric rough paths and their signature}
To start, we introduce two norms on the space of paths from $[0,T]$ to $G^N(\mathbb{R}^k)$ and their induced distances.
\begin{definition}[Homogeneous $p$-variation norm]
    Let $p\geq 1$. For any path $\mathbf{x}:[0,T]\rightarrow G^{\lfloor p\rfloor}(\mathbb{R}^k)$,
    \begin{equation*}
        [0,T]\ni t\mapsto \mathbf{x}_t := \left(1,\mathbf{x}^{(1)}_t,\dots,\mathbf{x}^{(\lfloor p\rfloor)}_t\right) \in G^{\lfloor p\rfloor}(\mathbb{R}^k),
    \end{equation*}
    the homogeneous $p$-variation norm of $\mathbf{x}$ is defined as
    \begin{equation*}
        \lVert \mathbf{x}\rVert_{\textup{hom $p$-var};[0,T]} := \sup_{D} \left(\sum_{i=0}^{n-1} d_{cc}(\mathbf{x}_{t_i},\mathbf{x}_{t_{i+1}})^p\right)^{1/p},
    \end{equation*}
    where $D=D_{[0,T]}$ is a partition of the time interval $[0,T]$. Then, for any $\mathbf{x},\mathbf{y}:[0,T]\rightarrow G^{\lfloor p\rfloor}(\mathbb{R}^k)$, the homogenous $p$-variation distance between $\mathbf{x}$ and $\mathbf{y}$ is defined as
    \begin{equation*}
        d_{\textup{hom $p$-var};[0,T]}(\mathbf{x},\mathbf{y}) := \sup_{D} \left(\sum_{i=0}^{n-1} d_{cc}(\mathbf{x}_{t_i,t_{i+1}},\mathbf{y}_{t_i,t_{i+1}})^p\right)^{1/p},
    \end{equation*}
    where $\mathbf{x}_{t_i,t_{i+1}}:=\mathbf{x}_{t_i}^{-1}\bullet \mathbf{x}_{t_{i+1}}$.
\end{definition}

Recalling that $G^N(\mathbb{R}^k)\subset T^N(\mathbb{R}^k)$, $G^N(\mathbb{R}^k)$ can also be endowed with the following norm.
\begin{definition}[Inhomogeneous $p$-variation norm]\label{def:inhom_p_var_norm}
    Let $p\geq 1$. For any path $\mathbf{x}:[0,T]\rightarrow G^{\lfloor p\rfloor}(\mathbb{R}^k)$,
    the inhomogeneous $p$-variation norm of $\mathbf{x}$ is defined as
    \begin{equation*}
        \lVert \mathbf{x}\rVert_{\textup{inhom $p$-var};[0,T]} := \max_{1\leq m\leq \floor{p}}\sup_D \left(\sum_{i=0}^{n-1}\lVert\mathbf{x}^{(m)}_{t_i,t_{i+1}}\rVert^{p/m}_{(\mathbb{R}^k)^{\otimes m}}\right)^{m/p},
    \end{equation*}
    where $D=D_{[0,T]}$ is a partition of the time interval $[0,T]$ and $\mathbf{x}_{t_i,t_{i+1}}:=\mathbf{x}_{t_i}^{-1}\bullet \mathbf{x}_{t_{i+1}}$. Then, for any $\mathbf{x},\mathbf{y}:[0,T]\rightarrow G^{\lfloor p\rfloor}(\mathbb{R}^k)$, the inhomogeneous $p$-variation distance between $\mathbf{x}$ and $\mathbf{y}$ is defined as
    \begin{equation*}
        d_{\textup{inhom $p$-var};[0,T]}(\mathbf{x},\mathbf{y}) := \max_{1\leq m\leq \floor{p}}\sup_D\left(\sum_{i=0}^{n-1}\lVert\mathbf{x}^{(m)}_{t_i,t_{i+1}}-\mathbf{y}^{(m)}_{t_i,t_{i+1}}\rVert^{p/m}_{(\mathbb{R}^k)^{\otimes m}}\right)^{m/p}.
    \end{equation*}
\end{definition}
\begin{remark}
    Let us recall that a norm $\lVert\cdot\rVert$ is said to be homogeneous when
    \begin{equation*}
        \lVert \delta_{\lambda}\cdot\rVert =  |\lambda| \lVert \cdot\rVert\qquad \forall \lambda\in\mathbb{R},
    \end{equation*}
    where $\delta_\lambda$ is the dilation map (see, e.g., \cite[section 7.5.4]{Friz_Victoir_2010}). From their definition, it is clear that, while $\lVert \cdot\rVert_{\textup{hom $p$-var};[0,T]}$ is homogeneous, $\lVert \cdot\rVert_{\textup{inhom $p$-var};[0,T]}$ is not. However, they induce the same topology and the same notion of Cauchy sequences.
\end{remark}
\begin{remark}
    Let us notice that the inhomogeneous $p$-variation norm is a pseudo-metric, in fact $d_{\textup{inhom $p$-var};[0,T]}$ satisfies all the properties of a metric except that $d_{\textup{inhom $p$-var};[0,T]}(\mathbf{x},\mathbf{y})=0$ does not in general imply that $\mathbf{x}=\mathbf{y}$. Therefore, the topology induced by $d_{\textup{inhom $p$-var};[0,T]}$ is a non-Hausdorff topology.
\end{remark}

We can finally define the set of \emph{(weakly) geometric $p$-rough paths}.
\begin{definition}[Weakly geometric $p$-rough paths]\label{def:weakly_geom_p_rp}
    Let $p\geq 1$. We define $C^{\textup{$p$-var}}_o([0,T],G^{\lfloor p\rfloor}(\mathbb{R}^k))$ as the set of continuous paths $\mathbf{x}:[0,T]\rightarrow G^{\lfloor p\rfloor}(\mathbb{R}^k)$ such that $\mathbf{x}_0=(1,0,\dots,0)$ and
    \begin{equation*}
        \lVert \mathbf{x} \rVert_{\textup{hom $p$-var};[0,T]} < \infty \qquad (\textit{equiv.}\; \lVert \mathbf{x} \rVert_{\textup{inhom $p$-var};[0,T]} < \infty).
    \end{equation*}
\end{definition}
\begin{remark}\label{remark:rough_path_starting_point}
    By removing the condition $\mathbf{x}_0=(1,0,\dots,0)$, it is defined the set $C^{\textup{$p$-var}}([0,T],G^{\lfloor p\rfloor}(\mathbb{R}^k))$ of paths $\mathbf{x}:[0,T]\rightarrow G^{\lfloor p\rfloor}(\mathbb{R}^k)$ such that 
    \begin{equation*}
        \{t\mapsto\mathbf{x}_{0,t} := \mathbf{x}^{-1}\bullet\mathbf{x}_t\}_{t\in[0,T]} \in C^{\textup{$p$-var}}_o([0,T],G^{\lfloor p\rfloor}(\mathbb{R}^k)).
    \end{equation*}
    The same goes for the set of geometric $p$-rough paths defined below.
\end{remark}

\begin{definition}[Geometric $p$-rough paths]\label{def:geom_p_rp}
    Let $p\geq 1$. We define $C^{0,\textup{$p$-var}}_o([0,T],G^{\lfloor p\rfloor}(\mathbb{R}^k))$ as the set of continuous paths $\mathbf{x}:[0,T]\rightarrow G^{\lfloor p\rfloor}(\mathbb{R}^k)$ such that $\mathbf{x}_0=(1,0,\dots,0)$ and for which there exists a sequence of bounded variation, $\mathbb{R}^k$-valued paths $\{x_n\}_{n\geq 0}$ such that
    \begin{equation*}
        d_{\textup{hom $p$-var};[0,T]}\left(\mathbf{x},S^{\leq \lfloor p\rfloor}(x_n)_\cdot\right)\xrightarrow{n\rightarrow\infty} 0\qquad \Big(\textit{equiv.}\;d_{\textup{inhom $p$-var};[0,T]}\left(\mathbf{x},S^{\leq \lfloor p\rfloor}(x_n)_\cdot\right)\xrightarrow{n\rightarrow\infty} 0\Big),
    \end{equation*}
    where $S^{\leq \lfloor p\rfloor}_\cdot$ denotes the \emph{step-$\lfloor p\rfloor$ signature} for continuous paths of bounded variation \eqref{eq:signature_C_1_path} as a function of time.
\end{definition}

From interpolations results (see \cite[Chapter 8]{Friz_Victoir_2010}), it holds that for any $p\geq 1$,
\begin{equation*}
    C^{0,\textup{$p$-var}}_o([0,T],G^{\lfloor p\rfloor}(\mathbb{R}^k)) \subset C^{\textup{$p$-var}}_o([0,T],G^{\lfloor p\rfloor}(\mathbb{R}^k)).
\end{equation*}

Now, we introduce the step-$N$ signature and the full signature of a weakly geometric rough path.
\begin{definition}[$p$-Lyons lift]
    Let $N\geq\lfloor p\rfloor \geq 1$ and $\mathbf{x}\in C_o^{\textup{$p$-var}}([0,T],G^{\lfloor p\rfloor}(\mathbb{R}^k))$. A path in $C_o^{\textup{$p$-var}}([0,T],G^N(\mathbb{R}^k))$ that projects down onto $\mathbf{x}$ is said to be a $p$-Lyons lift of $\mathbf{x}$ of order $N$. 
\end{definition}
It is a classical result (see, e.g., \cite[Theorem 9.5]{Friz_Victoir_2010}) that there exists a unique Lyons lift of order $N$, $N\geq\lfloor p\rfloor \geq 1$, for any path $\mathbf{x}\in C_o^{\textup{$p$-var}}([0,T],G^{\lfloor p\rfloor}(\mathbb{R}^k))$. With some abuse of notation, we denote such a path with $S^{\leq N}_T(\mathbf{x})$ and we call it the step-$N$ signature of $\mathbf{x}$.

Finally, the following result guarantees the existence of the full signature of any weakly geometric rough path.
\begin{theorem}\label{thm:signature_rough_case}
    Let $p\geq 1$. There exists a map
    \begin{equation*}
        S(\cdot)_T:C_o^{\textup{$p$-var}}([0,T],G^{\lfloor p\rfloor}(\mathbb{R}^k))\rightarrow (T(\mathbb{R}^k)):=\prod_{\ell=0}^\infty (\mathbb{R}^k)^{\otimes \ell};\quad S(\mathbf{x})_T=\left(1,S(\mathbf{x})^{(1)}_T,S(\mathbf{x})^{(2)}_T,\dots\right)
    \end{equation*}
    such that $S(\mathbf{x})_T^{(m)}=\mathbf{x}^{(m)}_T$ for all $m=0,\dots,\floor{p}$. Moreover, $S(\mathbf{x})_T=S(\mathbf{y})_T$ if and only if $\mathbf{x}\sim_t\mathbf{y}$, where $\sim_t$ denotes the tree-like equivalence (see Definition \ref{def:tree_like_eq} below).
\end{theorem}
\begin{proof}
    For the first part of the statement, see \cite{Lyons98}; for the second part, see \cite{Boedihardjo2016}.
\end{proof}

\begin{remark}\label{remark:invariant_time_par}
    It is possible to make the full signature map $S$ dependent onbb the parametrization of time by adding a univariate component with monotonically increasing (or decreasing) values that can be interpreted as auxiliary time (see, e.g., \cite{chevyrev2025primersignaturemethodmachine}). To this end, we define the subset    \begin{equation*}
        \begin{aligned}
            &\widehat{C}_o^{\textup{$p$-var}}([0,T],G^{\lfloor p\rfloor}(\mathbb{R}^{d+1})) :=\\
            &:= \Big\{\widehat{\mathbf{x}}\in C_o^{\textup{$p$-var}}([0,T],G^{\lfloor p\rfloor}(\mathbb{R}^{d+1})) \,:\, \widehat{\mathbf{x}}^{(1)}_t=\left(t,\mathbf{x}^{(1)}_t\right)\;\forall\,t\in[0,T],\,\mathbf{x}\in C_o^{\textup{$p$-var}}([0,T],G^{\lfloor p\rfloor}(\mathbb{R}^d))\Big\}.
        \end{aligned}
    \end{equation*}
    Moreover, the inhomogeneous $p$-variation norms induces on $\widehat{C}^{\textup{$p$-var}}_o([0,T],G^{\lfloor p\rfloor}(\mathbb{R}^{d+1}))$ a Hausdorff topology.
\end{remark}
\medskip

Also, from the local Lipschitz continuity of the signature map with respect to the inhomogenous path distances on the pathspace \cite[Theorem 9.10]{Friz_Victoir_2010}, the following corollary is immediate.
\begin{corollary}\label{corollary:local_Lip}
    Let $\mathbf{x},\mathbf{y}\in C^{0,\textup{$p$-var}}_o([0,T],G^{\lfloor p\rfloor}(\mathbb{R}^k))$. Then, for every $t\in[0,T]$,
    \begin{equation*}
        \sum_{m=0}^N \big\lVert S^m(\mathbf{x})_t - S^m(\mathbf{y})_t \big\rVert_{(\mathbb{R}^k)^{\otimes m}} \leq C(p,N,\Gamma)d_{\textup{inhom $p$-var};[0,t]}(\mathbf{x},\mathbf{y}),
    \end{equation*}
    where $\Gamma:=\max\big\{\lVert\mathbf{x}\rVert_{\textup{$p$-var};[0,T]},\lVert\mathbf{y}\rVert_{\textup{$p$-var};[0,T]}\big\}$.
\end{corollary}
\begin{proof}
    The above result follows from Proposition 8.7 and Theorem 9.10 in \cite{Friz_Victoir_2010}.
\end{proof}

Before concluding this subsection, we introduce one last norm.
\begin{definition}[Homogeneous $\alpha$-H\"{o}lder norm]\label{def:hom_alpha_Hold_norm}
    Let $\alpha\in(0,1]$. For any path $\mathbf{x}:[0,T]\rightarrow G^{\lfloor p\rfloor}(\mathbb{R}^k)$, the homogeneous $\alpha$-H\"{o}lder norm of $\mathbf{x}$ is defined as
    \begin{equation*}
        \lVert \mathbf{x}\rVert_{\textup{hom $\alpha$-H\"{o}l};[0,T]} := \sup_{0\leq s<t\leq T}\frac{d_{cc}(\mathbf{x}_s,\mathbf{x}_t)}{|t-s|^\alpha}.
    \end{equation*}
    Then, for any $\mathbf{x},\mathbf{y}:[0,T]\rightarrow G^{\lfloor p\rfloor}(\mathbb{R}^k)$, the homogenous $\alpha$-H\"{o}lder distance between $\mathbf{x}$ and $\mathbf{y}$ is defined as
    \begin{equation*}
        d_{\textup{hom $\alpha$-H\"{o}l};[0,T]}(\mathbf{x},\mathbf{y}) := \sup_{0\leq s<t\leq T}\frac{d_{cc}(\mathbf{x}_{s,t},\mathbf{y}_{s,t})}{|t-s|^\alpha},
    \end{equation*}
    where $\mathbf{x}_{s,t}:=\mathbf{x}_s^{-1}\bullet \mathbf{x}_t$.
\end{definition}

Similarly to Definition \ref{def:inhom_p_var_norm}, it is possible to define the inhomogeneous $\alpha$-H\"{o}lder norm. Also, following Definition \ref{def:weakly_geom_p_rp} and Definition \ref{def:geom_p_rp}, from Definition \ref{def:hom_alpha_Hold_norm} it is possible to define the set of $\alpha$-H\"{o}lder (weakly) geometric rough paths. However, since this will not play a role in the following, we omit the definition and discussion and interested readers can refer to \cite[Chapter 8]{Friz_Victoir_2010}.

\subsection{The signature map as a universal feature map}
In this subsection, we discuss an application developed in \cite[Section 5.3]{CSTGlobalUAT}, where the weighted Stone–Weierstrass theorem is employed to prove a global universal approximation theorem for functionals of weakly geometric rough paths using linear functions of their signature.

To start, we define the space of weakly geometric $p$-rough paths which are also $\alpha$-H\"{o}lder continuous.
\begin{definition}
    Let $(p,\alpha)\in(1,\infty)\times(0,1)$ with $p\alpha<1$. We define $C^{\alpha,p}_o([0,T],G^{\lfloor 1/\alpha\rfloor}(\mathbb{R}^k))$ as the set of continuous paths $\mathbf{x}:[0,T]\rightarrow G^{\lfloor 1/\alpha\rfloor}(\mathbb{R}^k)$ such that $\mathbf{x}_0=(1,0,\dots,0)$ and
    \begin{equation*}
        \begin{aligned}
            &\lVert \mathbf{x} \rVert_{\textup{hom $p$-var};[0,T]} + \lVert \mathbf{x} \rVert_{\textup{hom $\alpha$-H\"{o}l};[0,T]} < \infty\\(\textit{equiv.}\; &\lVert \mathbf{x} \rVert_{\textup{inhom $p$-var}} + \lVert \mathbf{x} \rVert_{\textup{inhom $\alpha$-H\"{o}l};[0,T]} < \infty).
        \end{aligned}
    \end{equation*}
    Moreover, we define
    \begin{equation*}
        \begin{aligned}
            &\widehat{C}^{\alpha,p}_o=\widehat{C}^{\alpha,p}_o([0,T],G^{\lfloor 1/\alpha\rfloor}(\mathbb{R}^{k+1})) :=\\
            &:= \Big\{\widehat{\mathbf{x}}\in C_o^{\alpha,p}([0,T],G^{\lfloor 1/\alpha\rfloor}(\mathbb{R}^{k+1})) \,:\, \widehat{\mathbf{x}}^{(1)}=\left(t,\mathbf{x}^{(1)}\right)\;\forall\,t\in[0,T],\,\mathbf{x}\in C^{\alpha,p}_o([0,T],G^{\lfloor 1/\alpha\rfloor}(\mathbb{R}^k))\Big\}.
        \end{aligned}
    \end{equation*}
\end{definition}

Then, the following result holds.
\begin{theorem}\label{thm:CuchieroTeichmannGUAT}
    For any $\widehat{\mathbf{x}}\in\widehat{C}^{\alpha,p}_o$ and some $\beta>0$ and $\gamma\geq\lfloor 1/\alpha\rfloor$, let
    \begin{equation*}
        \psi(\widehat{\mathbf{x}}) := \exp\left(\beta(\lVert \widehat{\mathbf{x}} \rVert_{\textup{hom $p$-var};[0,T]} + \lVert \widehat{\mathbf{x}} \rVert_{\textup{hom $\alpha$-H\"{o}l};[0,T]})^\gamma\right).
    \end{equation*}
    Also, let
    \begin{equation*}
        B_\psi(\widehat{C}^{\alpha,p}_o) := \Bigg\{f:\widehat{C}^{\alpha,p}_o\rightarrow\mathbb{R} \,:\, \lVert f\rVert
        _{\mathcal{B}_\psi(\widehat{C}^{\alpha,p}_o)}:=\sup_{\widehat{\mathbf{x}}\in\widehat{C}^{\alpha,p}_o}\frac{f(\widehat{\mathbf{x}})}{\psi(\widehat{\mathbf{x}})}<\infty\Bigg\}
    \end{equation*}
    and $\mathcal{B}_\psi(\widehat{C}^{\alpha,p}_o)$ be the closure with respect to $\lVert \cdot\rVert_{\mathcal{B}_\psi(\widehat{C}^{\alpha,p}_o)}$ of the space $C_b^0\left(\widehat{C}^{\alpha,p}_o;\mathbb{R}\right)$ of bounded and continuous real-valued functions over $\widehat{C}^{\alpha,p}_o$. Then, for every map $f\in \mathcal{B}_\psi(\widehat{C}^{\alpha,p}_o)$ and for every $\varepsilon>0$, there exists a linear function of the form 
    \begin{equation*}
        (T(\mathbb{R}^{k+1}))\ni S(\widehat{\mathbf{x}})_T \longmapsto \langle\ell,S(\widehat{\mathbf{x}})_T\rangle := \sum_{0\leq |I|\leq N} \alpha_I\big\langle \varepsilon_I, S^{|I|}(\widehat{\mathbf{x}})_T \big\rangle_{(\mathbb{R}^{k+1})^{\otimes |I|}},
    \end{equation*}
    where $N\in\mathbb{N}$, $\alpha_I\in\mathbb{R}$, $I=(i_1,\dots,i_m)\in\{1,\dots,k\}^m$ is a multi-index with $|I|=m$, $\varepsilon_I:=\varepsilon_{i_1}\otimes\dots\otimes\varepsilon_{i_m}$ are the basis elements of $(\mathbb{R}^{k+1})^{\otimes m}$, and $\langle \cdot,\cdot\rangle_{(\mathbb{R}^{k+1})^{\otimes |I|}}$ is the usual inner product on $(\mathbb{R}^{k+1})^{\otimes |I|}$, such that
    \begin{equation*}
        \sup_{t\in[0,T]}\sup_{\widehat{\mathbf{x}}\in \widehat{C}^{\alpha,p}_o} \frac{ \big|f(\widehat{\mathbf{x}}_{\cdot\wedge t}) -\langle\ell,S(\widehat{\mathbf{x}})_t\rangle\big| }{ \psi(\widehat{\mathbf{x}}_{\cdot\wedge t}) } < \varepsilon.
    \end{equation*}
\end{theorem}
\begin{proof}
    See \cite[Corollary 5.7]{CSTGlobalUAT} and \cite[Theorem 3.13]{CuchieroChrista2025Uatf}.
\end{proof}

\subsection{Geometric $p$-rough path lifts of semimartingales}
In this subsection, we specialise the above results to the case of semimartingales.
\medskip

\begin{theorem}\label{thm:geometric_p_rough_path_of_semimartingale}
    Let $p\in(2,3)$ and $X:[0,T]\rightarrow\mathbb{R}^d$ be a semimartingale defined on the probability space $(\Omega,\mathcal{F},\mathbb{F},\mathbb{P})$. Let us define
    \begin{equation*}
        \mathbb{X}_{0,t} = \mathbb{X}_t := \left(1,\int_0^t\circ dX_s,\int_0^t\circ dX^{\otimes 2}\right)=:\left(1,\mathbb{X}^1_t,\mathbb{X}^2_t\right),
    \end{equation*}
   for every $t\in[0,T]$. Then, for almost every $\omega\in \Omega$, $\mathbb{X}$ is a geometric $p$-rough path. Also, $\mathbb{X}$ is called the geometric $p$-rough path lift of $X$.
\end{theorem}
\begin{proof}
    See Chapter 14 in \cite{Friz_Victoir_2010}.
\end{proof}

\begin{proposition}\label{prop:signature_of_semimartingale}
    Let $p\in (2,3)$ and $X:[0,T]\rightarrow\mathbb{R}^d$ be a semimartingale defined on the probability space $(\Omega,\mathcal{F},\mathbb{F},\mathbb{P})$ with geometric $p$-rough path lift
    \begin{equation*}
        \mathbb{X}_t = \left(1,\mathbb{X}^1_t,\mathbb{X}^2_t\right) = \left(1,\int_0^t\circ dX,\int_0^t\circ dX^{\otimes 2}\right)\quad \forall t\in[0,T].
    \end{equation*}
    Then,
    \begin{equation}\label{eq:definition_signature_for_rough_paths}
        S(\mathbb{X})_T=\left(1,\int_0^T\circ dX,\int_0^T\circ dX^{\otimes 2},\int_0^T\circ dX^{\otimes 3},\dots\right) \in (T(\mathbb{R}^d)),
    \end{equation}
    where all the stochastic integrals are taken in the Stratonovich sense.
\end{proposition}
\begin{proof}
    The result follows directly from Theorems \ref{thm:signature_rough_case} and \ref{thm:geometric_p_rough_path_of_semimartingale}. Indeed, combining these two results shows that, for any semimartingale $X$, once the second-level iterated integral $\int \circ dX^{\otimes 2}$ is defined in the Stratonovich sense, all higher-order iterated integrals $\int \circ dX^{\otimes m}$ are canonically defined for every $m \geq \lfloor p \rfloor $.
\end{proof}
\begin{remark}
    Given a semimartingale $X$, let $\augpr{X}$ denote the \emph{time-augmented geometric p-rough path lift} of $X$, which is the canonical $p$-rough path lift of the process $t\mapsto \widehat{X}_t:=(t,X_t)$. The inclusion of the time component ensures that the signature map $S$ is dependent to the time parametrization. For every $0\leq s<t\leq T$, $\augpr{X}_{s,t}$ is defined as
    \begin{equation}\label{eq:time_augmente_sig}
        \begin{aligned}
            &\augpr{X}_{s,t}:=\left(1,\, \int_{s<u_1<t} \circ d\widehat{X}_{u_1},\, \int_{s<u_1<u_2<t} \circ d\widehat{X}_{u_1}\otimes \circ d\widehat{X}_{u_2}\right)\\
            &=\left(1,
            \begin{bmatrix}
                t-s\\
                \int_{s<u_1<t} \circ dX_{u_1}
            \end{bmatrix},
            \begin{bmatrix}
                \frac{1}{2}(t^2-s^2) &\int_s^t (u-s)\circ dX_u^\top\\
                \int_s^t (X_u-X_s)du &\int_s^t (X_u-X_s)\otimes \circ dX_u^\top
            \end{bmatrix}\right).
        \end{aligned}
    \end{equation}
    Moreover, the inhomogeneous $p$-variation norms are Hausdorff on the space of the $p$-rough lifts of time-augmented semimartingales (see, e.g., \cite{chevyrev2025primersignaturemethodmachine}).
\end{remark}

\section{The path-dependent McKean--Vlasov SDE}\label{sec:well_posed_path_dep_MKV}
In this section, we establish the well-posedness of the path-dependent McKean--Vlasov SDE
\begin{equation}\label{eq:path_dep_MKV}
    \begin{aligned}
        Y_t = \eta + \int_0^t b\left(s,Y_s,\mathbb{E}\left[\mathcal{G}(Y_{\cdot\wedge s})\right]\right)dt + \int_0^t \sigma\left(s,Y_s,\mathbb{E}\left[\mathcal{G}(Y_{\cdot\wedge s})\right]\right)dW_s,
    \end{aligned}
\end{equation}
with $t\in[0,T]$. To start, we recall some basic notions and results regarding the Wasserstein distance. For this, we follow Chapter 6 of \cite{Villani}.

Let $\left(\mathcal{X},\rho\right)$ be a separable metric space and consider two probability measures on $\mathcal{X}$, namely $\mu,\nu\in\mathcal{P}(\mathcal{X})$. The space of the so-called couplings of $\mu$ and $\nu$ is defined as
\begin{equation*}
    \Pi(\mu,\nu):=\big\{\pi\in\mathcal{P}\left(\mathcal{X}\times\mathcal{X}\right) \,:\, \pi(\cdot\times\mathcal{X})=\mu(\cdot),\quad \pi(\mathcal{X}\times\cdot)=\nu(\cdot)\big\}.
\end{equation*}
Moreover, for fixed $p\geq 1$, we define
\begin{equation*}
    \mathcal{P}^p(\mathcal{X}) := \Bigg\{ \mu\in\mathcal{P}(\mathcal{X}) \,:\, \int_{\mathcal{X}}\rho(x,x_0)^p\mu(dx) < \infty \Bigg\}.
\end{equation*}
\begin{remark}
    The definition of $\mathcal{P}^p(\mathcal{X})$ does not depend on the choice of $x_0\in\mathcal{X}$.
\end{remark}

The \emph{$p$-Wasserstein metric} is defined as follows.
\begin{definition}\label{def:Wasserstein_norm}
    Let $\left(\mathcal{X},\rho\right)$ be a separable metric space and $p\geq 1$. Given two probability measures $\mu,\nu\in\mathcal{P}^p(\mathcal{X})$, $\mathcal{W}^p(\mu,\nu)$ defined as
    \begin{equation*}
        \mathcal{W}^p(\mu,\nu) := \inf_{\pi\in\Pi(\mu,\nu)} \left[\int_{\mathcal{X}\times\mathcal{X}}\rho(x,y)^p\pi(dx,dy)\right]^{1/p}
    \end{equation*}
    is called the $p$-Wasserstein metric.
\end{definition}

We have the following result.
\begin{theorem}[Villani \cite{Villani}, Theorem 6.18]\label{thm:Wasserstein_Polish}
    If  $\left(\mathcal{X},\rho\right)$ is a Polish space (complete and separable), then
    \begin{enumerate}[i)]
        \item $\mathcal{W}^p$ is a metric.
        \item $\left(\mathcal{P}^p(\mathcal{X}),\mathcal{W}^p\right)$ is also a Polish space
    \end{enumerate}
\end{theorem}
\begin{proof}
    See proof of Theorem 6.18 in \cite{Villani}.
\end{proof}

We also recall that the Wasserstein distance metrizes the weak convergence in $\mathcal{P}^p(\mathcal{X})$.
\begin{theorem}\label{thm:equivalence_Wasserstein_weak_conv}
    Let $(\mathcal{X},\rho)$ be a Polish space and $p\geq 1$. If $\{\mu_k\}_{k\in\mathbb{N}}$ is a sequence of measures in $\mathcal{P}(\mathcal{X})$, then the following statements are equivalent:
    \begin{enumerate}[i)]
        \item $\{\mu_k\}_{k\in\mathbb{N}}$ converges weakly in $\mathcal{P}^p(\mathcal{X})$ to $\mu$.
        \item $\mathcal{W}^p(\mu_k,\mu)\xrightarrow{k\rightarrow\infty}0$.
    \end{enumerate}
\end{theorem}
\begin{proof}
    See proof of Theorem 6.7 in \cite[Chapter 6]{Villani}.
\end{proof}

Let $C^d:=C([0,T];\mathbb{R}^d)$ be the space of $\mathbb{R}^d$-valued continuous functions over $[0,T]$ equipped with the standard supremum norm, denoted $\Vert\cdot\Vert_{\infty,T}$.

It is a classical result that $\left(C^d,\lVert\cdot\rVert_{\infty,T}\right)$ is a Polish space (see, e.g., \cite[Chapter 18]{FristedtBert1997Amat}). Therefore, for every $p\geq 1$ the space of probability measures over $\mathcal{P}^p(C^d)$ endowed with the Wasserstein metric
\begin{equation*}
    \mathcal{W}^p_T(\mu,\nu):= \inf_{\pi\in\Pi(\mu,\nu)}\left[\int_{C^d\times C^d} \lVert x-y\rVert_{\infty,T}^p\pi(dx,dy)\right]^{1/p},\qquad \forall\mu,\nu\in\mathcal{P}^p(C^d)
\end{equation*}
is also a Polish space.

Moreover, for every $\mu,\nu\in\mathcal{P}^p(C^d)$ and every $t\in[0,T]$, we also define
\begin{equation*}
    \mathcal{W}^p_t(\mu,\nu):= \inf_{\pi\in\Pi(\mu,\nu)}\left[\int_{C^d\times C^d} \lVert x-y\rVert_{\infty,t}^p\pi(dx,dy)\right]^{1/p},\qquad \forall\mu,\nu\in\mathcal{P}^p(C^d).
\end{equation*}
Let us notice that for every $t\in[0,T]$, $\mathcal{W}^p_t$ is not a distance on $\mathcal{P}^p(C^d)=\mathcal{P}^p(C([0,T];\mathbb{R}^d))$, but rather on $\mathcal{P}^p(C([0,t];\mathbb{R}^d))$.

We are now ready to discuss the well-posedness of the path-dependent McKean-Vasov SDE \eqref{eq:path_dep_MKV}.
\smallskip

\begin{definition}\label{def:strong}
    We say that \eqref{eq:path_dep_MKV} admits a strong solution if for any filtered probability space $(\Omega,\mathcal{F},\mathbb{F}=\{\mathcal{F}_t\}_t,\mathbb{P})$ equipped with an $\mathbb{F}$-Brownian motion $W$, there exists an $\mathbb{F}$-adapted process $Y$ that satisfies \eqref{eq:path_dep_MKV}.
\end{definition}
\medskip

\begin{definition}\label{def:weak}
    We say that \eqref{eq:path_dep_MKV} admits a weak solution if there exist a filtered probability space $\left(\Omega,\mathcal{F},\mathbb{F}=\{\mathcal{F}_t\}_{t\geq 0},\mathbb{P}\right)$ equipped with an $\mathbb{F}$-Brownian motion W and an $\mathbb{F}$-adapted process $Y$ that satisfies \eqref{eq:path_dep_MKV}.
\end{definition}
\medskip

\begin{definition}\label{def:strong_unique}
     Let $Y$ and $Y^\prime$ be two strong solutions to SDE \eqref{eq:path_dep_MKV}. We say that strong uniqueness holds for equation \eqref{eq:path_dep_MKV} if $Y$ and $Y^\prime$ are indistinguishable.
\end{definition}
\medskip

\begin{definition}\label{def:law_unique}
    Let $Y$ and $Y^\prime$ be two weak solutions to SDE \eqref{eq:path_dep_MKV} on possibly  different filtered probability spaces $\left(\Omega,\mathcal{F},\mathbb{F}=\{\mathcal{F}_t\}_{t\geq 0},\mathbb{P}\right)$ and $(\Omega^\prime,\mathcal{F}^\prime,\mathbb{F}^\prime=\{\mathcal{F}^\prime_t\}_{t\geq 0},\mathbb{P}^\prime)$. We say that SDE \eqref{eq:path_dep_MKV} admits a weak solution unique in the sense of probability law if $\mathcal{L}(Y)=\mathcal{L}(Y^\prime)$.
\end{definition}
\medskip

The following well-posedness result holds.
\begin{theorem}\label{eq:well_posed_path_dep_MKV}
    Under the assumption of Hypothesis \ref{hypothesis_A} and Hypothesis \ref{hypothesis_B}, the path-dependent McKean--Vlasov SDE \eqref{eq:path_dep_MKV} admits a pathwise unique strong solution as in Definition \ref{def:strong} and Definition \ref{def:strong_unique}. Moreover, it admits a weak solution unique in the sense of probability law as in Definition \ref{def:weak} and Definition \ref{def:law_unique}.
\end{theorem}
\begin{proof}
    Following the approach of \cite{Sznitman}, we consider the map
    \begin{equation}\label{eq:def_map_Psi}
        \Psi:\mathcal{P}^2(C^d)\rightarrow\mathcal{P}^2(C^d),\qquad \Psi(\nu)=\mathcal{L}(Y^\nu),
    \end{equation}
    where
    \begin{equation}\label{eq:path_dep_MKV_fixed_measure}
        Y^{\nu}_t = \eta + \int_0^t b\left(s,Y^{\nu}_s,\mathbb{E}_{z\sim\nu}\left[\mathcal{G}(z_{\cdot\wedge s})\right]\right)ds + \int_0^t \sigma\left(s,Y^{\nu}_s,\mathbb{E}_{z\sim\nu}\left[\mathcal{G}(z_{\cdot\wedge s})\right]\right)dW_s,
    \end{equation}
    with $t\in[0,T]$.

    For a fixed $\nu\in\mathcal{P}^2(C^d)$, SDE \eqref{eq:path_dep_MKV_fixed_measure} admits a pathwise unique strong solution. In fact, by the Lipschitz continuity assumption on the drift and diffusion coefficients,
    \begin{equation*}
        \begin{aligned}
            &\big\lVert b\left(t,x,\mathbb{E}_{z\sim\nu}\left[\mathcal{G}(z_{\cdot\wedge t})\right]\right) - b\left(t,y,\mathbb{E}_{z\sim\nu}\left[\mathcal{G}(z_{\cdot\wedge t})\right]\right)\big\rVert_{\mathbb{R}^d}\\
            &+ \big\lVert \sigma\left(t,x,\mathbb{E}_{z\sim\nu}\left[\mathcal{G}(z_{\cdot\wedge t})\right]\right) - \sigma\left(t,y,\mathbb{E}_{z\sim\nu}\left[\mathcal{G}(z_{\cdot\wedge t})\right]\right)\big\rVert_F \leq L\lVert x-y\rVert_{\mathbb{R}^d}.
        \end{aligned}
    \end{equation*}

    Also, notice that the image of $\Psi$ belongs to $\mathcal{P}^2(C^d)$. In fact, for every $t\in[0,T]$
    \begin{equation*}
        \begin{aligned}
            \mathbb{E}\left[\sup_{0\leq t\leq T}\lVert Y^{\nu}_t\rVert_{\mathbb{R}^d}^2\right] \leq &3\mathbb{E}\left[\lVert \eta\rVert_{\mathbb{R}^d}^2\right] + 3\mathbb{E}\left[\sup_{0\leq t\leq T}\Bigg\lVert \int_0^t b\left(s,Y^{\nu}_s,\mathbb{E}_{z\sim\nu}\left[\mathcal{G}(z_{\cdot\wedge s})\right]\right)ds \Bigg\rVert_{\mathbb{R}^d}^2\right]\\
            &+3\mathbb{E}\left[\sup_{0\leq t\leq T}\Bigg\lVert \int_0^t \sigma\left(s,Y^{\nu}_s,\mathbb{E}_{z\sim\nu}\left[\mathcal{G}(z_{\cdot\wedge s})\right]\right)dW_s \Bigg\rVert_{\mathbb{R}^d}^2\right].
        \end{aligned}
    \end{equation*}
    
    By Hypothesis \ref{hypothesis_A}, $\mathbb{E}\left[\lVert \eta\rVert_{\mathbb{R}^d}^2\right] <\infty$.
    
    By Jensen's inequality and the boundedness assumption on the drift coefficient $b$,
    \begin{equation*}
        \begin{aligned}
            &\mathbb{E}\left[\sup_{0\leq t\leq T} \Bigg\lVert \int_0^t b\left(s,Y^{\nu}_s,\mathbb{E}_{z\sim\nu}\left[\mathcal{G}(z_{\cdot\wedge s})\right]\right)ds \Bigg\rVert_{\mathbb{R}^d}^2\right]\\
            &\leq \mathbb{E}\left[ \sup_{0\leq t\leq T}t\int_0^t \lVert b\left(s,Y^{\nu}_s,\mathbb{E}_{z\sim\nu}\left[\mathcal{G}(z_{\cdot\wedge s})\right]\right)\rVert_{\mathbb{R}^d}^2ds\right] \leq T^2M_b^2.
        \end{aligned}
    \end{equation*}
    By Burkholder-Davis-Gundy inequality (see, e.g., \cite[Theorem 3.28]{KaratzasShreve}) and the boundedness assumption on the diffusion coefficient $\sigma$,
    \begin{equation*}
        \begin{aligned}
            &\mathbb{E}\left[\sup_{0\leq t\leq T}\Bigg\lVert \int_0^t \sigma\left(s,Y^{\nu}_s,\mathbb{E}_{z\sim\nu}\left[\mathcal{G}(z_{\cdot\wedge s})\right]\right)dW_s \Bigg\rVert_{\mathbb{R}^d}^2\right]\\
            &\leq \mathbb{E}\left[\sup_{0\leq t\leq T}\int_0^t \lVert\sigma\left(s,Y^{\nu}_s,\mathbb{E}_{z\sim\nu}\left[\mathcal{G}(z_{\cdot\wedge s})\right]\right)\rVert_{\mathbb{R}^d}^2ds\right]\leq TM_\sigma^2.
        \end{aligned}
    \end{equation*}
    Therefore, $\mathbb{E}\left[\lVert Y^{\nu}_t\rVert_{\mathbb{R}^d}^2\right] <\infty$.
    
    By proving that $\Psi$ admits a unique fixed point, we obtain that the path-dependent McKean--Vlasov SDE \eqref{eq:path_dep_MKV} admits a weak solution. As a byproduct of the computations, we obtain that also uniqueness in the sense of probability law and pathwise uniqueness holds for equation \eqref{eq:path_dep_MKV}. Then, by Yamada and Watanabe's result we conclude that \eqref{eq:path_dep_MKV} also admits a strong solution.
    
    Let $\nu,\nu^\prime\in\mathcal{P}^2(C^d)$. By \eqref{eq:def_map_Psi}, $\Psi(\nu)$ denotes the law of the solution to the SDE
    \begin{equation*}
        Y^{\nu}_t = \eta + \int_0^t b\left(s,Y^{\nu}_s,\mathbb{E}_{z\sim\nu}\left[\mathcal{G}(z_{\cdot\wedge s})\right]\right)ds + \int_0^t \sigma\left(s,Y^{\nu}_s,\mathbb{E}_{z\sim\nu}\left[\mathcal{G}(z_{\cdot\wedge s})\right]\right)dW_s 
    \end{equation*}
    and $\Psi(\nu^\prime)$ denotes the law of the solution to the SDE
    \begin{equation*}
        Y^{\nu^\prime}_t = \eta + \int_0^t b\left(s,Y^{\nu^\prime}_s,\mathbb{E}_{z^\prime\sim\nu^\prime}\left[\mathcal{G}(z^\prime_{\cdot\wedge s})\right]\right)ds + \int_0^t \sigma\left(s,Y^{\nu^\prime}_s,\mathbb{E}_{z^\prime\sim\nu^\prime}\left[\mathcal{G}(z^\prime_{\cdot\wedge s})\right]\right)dW_s,
    \end{equation*}
    for every $t\in[0,T]$.
    
    Let $\pi\in\Pi(\nu,\nu^\prime)$ be an arbitrary coupling of $\nu$ and $\nu^\prime$. We have that
    \begin{align}
        &\mathbb{E}\left[ \big\lVert Y^\nu - Y^{\nu^\prime} \big\rVert_{\infty,t}^2 \right]\notag\\
        &\begin{aligned}
            \leq\, &2\mathbb{E}\Bigg[ \sup_{0\leq r \leq t} r \int_0^{r } \Big\lVert b\left(s,Y^{\nu}_s,\mathbb{E}_{z\sim\nu}\left[\mathcal{G}(z_{\cdot\wedge s})\right]\right) - b\left(s,Y^{\nu^\prime}_s,\mathbb{E}_{z^\prime\sim\nu^\prime}\left[\mathcal{G}(z^\prime_{\cdot\wedge s})\right]\right) \Big\rVert_{\mathbb{R}^d}^2ds \Bigg]\\
            &+2\mathbb{E}\Bigg[ \sup_{0\leq r\leq t}\Bigg\lVert \int_0^r \left[\sigma\left(s,Y^{\nu}_s,\mathbb{E}_{z\sim\nu}\left[\mathcal{G}(z_{\cdot\wedge s})\right]\right) - \sigma\left(s,Y^{\nu^\prime}_s,\mathbb{E}_{z^\prime\sim\nu^\prime}\left[\mathcal{G}(z^\prime_{\cdot\wedge s})\right]\right)\right]dW_s \Bigg\rVert_F^2\Bigg]
        \end{aligned}\label{eq:proof_ex_uniq_path_dep_MKV_1}\\
        &\begin{aligned}
            \leq\, &2T\mathbb{E}\Bigg[ \int_0^t \Big\lVert b\left(s,Y^{\nu}_s,\mathbb{E}_{z\sim\nu}\left[\mathcal{G}(z_{\cdot\wedge s})\right]\right) - b\left(s,Y^{\nu^\prime}_s,\mathbb{E}_{z^\prime\sim\nu^\prime}\left[\mathcal{G}(z^\prime_{\cdot\wedge s})\right]\right) \Big\rVert_{\mathbb{R}^d}^2ds \Bigg]\\
            &+2C_2\mathbb{E}\Bigg[ \int_0^t \Big\lVert \sigma\left(s,Y^{\nu}_s,\mathbb{E}_{z\sim\nu}\left[\mathcal{G}(z_{\cdot\wedge s})\right]\right) - \sigma\left(s,Y^{\nu^\prime}_s,\mathbb{E}_{z^\prime\sim\nu^\prime}\left[\mathcal{G}(z^\prime_{\cdot\wedge s})\right]\right)\Big\rVert^2_F ds\Bigg]
        \end{aligned}\label{eq:proof_ex_uniq_path_dep_MKV_2}\\
        &\leq 8(1+T)L^2\int_0^t\left(\mathbb{E}\left[\big\lVert Y_s^\nu-Y_s^{\nu^\prime}\big\rVert_{\mathbb{R}^d}^2\right] + C_2L_\mathcal{G}^2\int_{C^d\times C^d} \lVert z-z^\prime \rVert_{\infty,s}^2\pi(dz,dz^\prime)\right)ds\label{eq:proof_ex_uniq_path_dep_MKV_3}\\
        &\leq 8(1+T)L^2\int_0^t\left(\mathbb{E}\left[\big\lVert Y^\nu-Y^{\nu^\prime}\big\rVert_{\infty,s}^2\right] + C_2L_\mathcal{G}^2\int_{C^d\times C^d} \lVert z-z^\prime \rVert_{\infty,s}^2\pi(dz,dz^\prime)\right)ds,\notag
    \end{align}
     where \eqref{eq:proof_ex_uniq_path_dep_MKV_1} is due to the integral Cauchy-Schwartz inequality, \eqref{eq:proof_ex_uniq_path_dep_MKV_2} by the multidimensional Burkholder-Davis-Gundy inequality (see, e.g., \cite[Theorem 3.28]{KaratzasShreve}), and \eqref{eq:proof_ex_uniq_path_dep_MKV_3} by the assumptions in Hypothesis \ref{hypothesis_A} and \eqref{hypothesis_B}.
     
     By Gronwall's lemma,
     \begin{equation}\label{eq:proof_ex_uniq_path_dep_MKV_4}
         \mathbb{E}\left[ \big\lVert Y^\nu_\cdot - Y^{\nu^\prime}_\cdot \big\rVert_{\infty,t}^2\right]\leq e^{8(1+T)L^2T}C_2L_\mathcal{G}^2\int_0^t\int_{C^d\times C^d} \lVert z-z^\prime \rVert_{\infty,s}^2\pi(dz,dz^\prime)ds.
     \end{equation}

     By taking the infimum over the space $\Pi(\nu,\nu^\prime)$ of couplings of $\nu$ and $\nu^\prime$ over both sides of \eqref{eq:proof_ex_uniq_path_dep_MKV_4}, we obtain
    \begin{equation}\label{eq:proof_ex_uniq_path_dep_MKV_5}
        \mathbb{E}\left[ \big\lVert Y^\nu_\cdot - Y^{\nu^\prime}_\cdot \big\rVert_{\infty,t}^2 \right]\leq e^{8(1+T)L^2T}C_2L_\mathcal{G}^2\int_0^t [\mathcal{W}_s^2(\nu,\nu^\prime)]^2ds.
    \end{equation}

    Moreover,
    \begin{equation}\label{eq:contraction_ineq_path_MKV}
        [\mathcal{W}_t^2(\Psi(\nu),\Psi(\nu^\prime))]^2\leq\mathbb{E}\left[ \big\lVert Y^\nu_\cdot - Y^{\nu^\prime}_\cdot \big\rVert_{\infty,t}^2 \right]\leq e^{8(1+T)L^2T}C_2L_\mathcal{G}^2\int_0^t [\mathcal{W}_s^2(\nu,\nu^\prime)]^2ds,
    \end{equation}
    where the first inequality is due to the fact that we are picking a specific coupling of $\nu$ and $\nu^\prime$, while the second inequality is due to \eqref{eq:proof_ex_uniq_path_dep_MKV_5}.

    Assume there exist two weak solutions $Y,Y^\prime$ to the path-dependent McKean--Vlasov SDE \eqref{eq:path_dep_MKV}, as in Definition \ref{def:weak}. Let us denote with $\mu$ and $\mu^\prime$ the laws of $Y$ and $Y^\prime$, respectively. Then, $\Psi(\mu)=\mu$ and $\Psi(\mu^\prime)=\mu^\prime$. By \eqref{eq:contraction_ineq_path_MKV},
    \begin{equation*}
        [\mathcal{W}_T^2(\mu,\mu^\prime)]^2=[\mathcal{W}_T^2(\Psi(\mu),\Psi(\mu^\prime))]^2\leq e^{8(1+T)L^2T}C_2L_\mathcal{G}^2\int_0^T [\mathcal{W}_s^2(\mu,\mu^\prime)]^2ds.
    \end{equation*}
    By Gronwall's lemma,
    \begin{equation}\label{eq:weak_uniq_path_MKV_SDE}
        \mathcal{W}_T^2(\mu,\mu^\prime)=0,
    \end{equation}
    which implies uniqueness in probability law for equation \eqref{eq:path_dep_MKV} (see Definition \ref{def:law_unique}).

    On the other hand, assume there exist two strong solutions $Y,Y^\prime$ to the path-dependent McKean--Vlasov SDE \eqref{eq:path_dep_MKV}, as in Definition \ref{def:strong}. By substituting \eqref{eq:weak_uniq_path_MKV_SDE} in \eqref{eq:proof_ex_uniq_path_dep_MKV_5}, we obtain
    \begin{equation*}
        \mathbb{E}\left[ \big\lVert Y^\nu - Y^{\nu^\prime} \big\rVert^2_{\infty,T}  \right] = 0,
    \end{equation*}
    which implies strong uniqueness for equation \eqref{eq:path_dep_MKV} (see Definition \ref{def:strong_unique}).
   
    Now, fix a measure $m\in \mathcal{P}^2\left(C \right)$.   By iteratively applying   the map $\Psi$  to the measure $m$, we define a sequence of measures $  \{m_k\}_{k\in \mathbb{N}}$ in $\mathcal{P}^2\left(C^d \right)$  such that
    \begin{equation}\label{eq:def_m_k_ex_uni_SDE}
        m_k\,:=\,  \Psi^k(m) = \Psi\left(\Psi^{k-1}(m)  \right).
    \end{equation} 
    
    By iterating \eqref{eq:contraction_ineq_path_MKV}, 
    \begin{equation*}
         \left[\mathcal{W}^2_T\left( \Psi^{k+1}(m),\Psi^k(m) \right)\right]^2  \leq \frac{\left(e^{8(1+T)L^2T}C_2L_\mathcal{G}^2\right)^k}{k!} [\mathcal{W}_T^2\left( \Psi(m),m \right)]^2.
    \end{equation*}

    Therefore, for any $\ell,n\in\mathbb{N}$,
    \begin{equation*}
         \left[\mathcal{W}^2_T\left( \Psi^\ell(m),\Psi^n(m) \right)\right]^2  \leq [\mathcal{W}_T^2\left( \Psi(m),m \right)]^2\sum_{k=\ell}^n\frac{\left(e^{8(1+T)L^2T}C_2L_\mathcal{G}^2\right)^k}{k!}.
    \end{equation*}

    Since the tail
    \begin{equation*}
        \sum_{k=\ell}^\infty\frac{\left(e^{8(1+T)L^2T}C_2L_\mathcal{G}^2\right)^k}{k!} \xrightarrow{\ell\rightarrow\infty} 0,
    \end{equation*}
    $\{m_k\}_{k\in\mathbb{N}}$ is a Cauchy sequence in $\left(\mathcal{P}^2(C ), \mathcal W^2_T)\right)$. Then, by Theorem \ref{thm:equivalence_Wasserstein_weak_conv}, there exists a probability measure $\bar{\mu}\in\mathcal{P}^2\left(C \right)$ such that $m_k$ weakly converges to  $\bar{\mu}$ as $k$ increases to $\infty$. From \eqref{eq:def_m_k_ex_uni_SDE}, we get $\bar{\mu}=\Psi(\bar{\mu})=\mathcal{L}(Y^{\bar{\mu}})$. This concludes the proof of existence of a weak solution to the path-dependent McKean--Vlasov SDE \eqref{eq:path_dep_MKV}. By the results of Yamada and Watanabe \cite[p. 308]{KaratzasShreve}, there exists a strong unique solution to equation \eqref{eq:path_dep_MKV}.
\end{proof}

Finally, let us remark that by the same arguments as above, Theorem \ref{eq:well_posed_path_dep_MKV} can be easily extended to consider path dependent McKean--Vlasov SDEs \eqref{eq:path_dep_MKV} with two different path functionals in the drift and in the diffusion; that is,
\begin{equation*}
    Y_t = \eta + \int_0^t b\left(s,Y_s,\mathbb{E}\left[\mathcal{G}(Y_{\cdot\wedge s})\right]\right)ds + \int_0^t\sigma \left(s,Y_s,\mathbb{E}\left[\widetilde{\mathcal{G}}(Y_{\cdot\wedge s})\right]\right)dW_s;\qquad t\in[0,T],
\end{equation*}
where $\mathcal{G},\widetilde{\mathcal{G}}:C\left([0,T];\mathbb{R}^d\right)\rightarrow\mathbb{R}$ satisfy Hypothesis \ref{hypothesis_B}.

\section{The signature McKean--Vlasov SDE}\label{sec:well_posed_path_sig_MKV}
In this section, we establish strong existence and uniqueness for the signature McKean--Vlasov SDE
\begin{equation}\label{eq:sig_MKV_SDE}
    X_t = \eta + \int_0^t b\left(s,X_s,\Big\langle \ell,\mathbb{E}\left[S(\augpr{X})_s\right]\Big\rangle\right)ds + \int_0^t \sigma\left(s,X_s,\Big\langle \ell,\mathbb{E}\left[S(\augpr{X})_s\right]\Big\rangle\right)dW_s,
\end{equation}
with $t\in[0,T]$ and
\begin{equation}\label{eq:linear_comb_in_sig_MKV_SDE}
    \langle\ell,S(\widehat{\mathbb{X}})_t\rangle := \sum_{0\leq |I|\leq N} \alpha_I\langle \varepsilon_I,S(\widehat{\mathbb{X}})_t\rangle,
\end{equation}
where $N\in\mathbb{N}$, $\alpha_I\in\mathbb{R}$, $I=\{i_1,\dots,i_m\}\in\{1,\dots,d+1\}^m$ is a multi-index, and $\varepsilon_I:=\varepsilon_{i_1}\otimes\dots\otimes\varepsilon_{i_m}$ are the basis elements of $\mathbb{R}^{(d+1)\otimes m}$.

In order to do this, we first introduce a notion Wasserstein distance suitable for our framework. Specifically, we have the following result.
\begin{proposition}
    Let $k\in\mathbb{N}$ and $p\geq 1$. Let $\widehat{G\Omega}_p(\mathbb{R}^{k+1})$ be the space of time-augmented geometric $p$-rough paths $\widehat{C}^{0,p-\textup{var}}([0,T],G^{\lfloor p\rfloor}(\mathbb{R}^{k+1}))$ (Remark \ref{remark:invariant_time_par}) with the topology induced by the inhomogeneous $p$-variation norm $d_{\textup{inhom $p$-var;$[0,T]$}}$ (Definition \ref{def:inhom_p_var_norm}). Then $\left(\widehat{G\Omega}_p(\mathbb{R}^{k+1}),d_{\textup{inhom $p$-var;$[0,T]$}}\right)$ is a Polish space. Furthermore, $(\mathcal{P}^p\left(\widehat{G\Omega}_p(\mathbb{R}^{k+1})\right),\overline{\mathcal{W}}^p_T)$, where $\overline{\mathcal{W}}^p_T$ is the Wasserstein distance defined in Definition \ref{def:Wasserstein_norm} with $\rho=d_{\textup{inhom $p$-var;$[0,T]$}}$, is a Polish space.
\end{proposition}
\begin{proof}
    For a complete proof of the first part of the statement, see, e.g., Proposition 8.25 in \cite{Friz_Victoir_2010} or Proposition 39 in \cite{Chevyrev2022}. The second part is a direct consequence of Theorem \ref{thm:Wasserstein_Polish}.
\end{proof}

In what follows, to avoid cumbersome notation, the subscript $p$-var will always denote the inhomogeneous $p$-variation norm or distance.

To begin our analysis, we establish an a priori estimate for the solution of the signature McKean--Vlasov SDE \eqref{eq:sig_MKV_SDE}.
\begin{proposition}\label{prop:a_priori_estimate}
    Let $p\in[5/2,3)$. Under the assumption of Hypothesis \ref{hypothesis_A}, let $X$ be a solution to the signature McKean--Vlasov SDE \eqref{eq:sig_MKV_SDE}. It holds that
    \begin{equation*}
        \mathbb{E}\left[\lVert \augpr{X}\rVert_{\textup{$p$-var};[0,T]}\right] <\infty.
    \end{equation*}
\end{proposition}
\begin{proof}
    See Appendix \ref{appendix:prop:a_priori_estimate}.
\end{proof}
\medskip

In analogy with the proof of well-posedness for the path-dependent McKean--Vlasov SDE \eqref{eq:path_dep_MKV}, fixed a measure $\mu \in \mathcal{P}^p\left(\widehat{G\Omega}_p(\mathbb{R}^{d+1})\right)$, we consider the auxiliary equation
\begin{equation}\label{eq:sig_MKV_SDE_fixed_mu}
    X^\mu_t = \eta + \int_0^t b\left(s,X^{\mu}_s,\Big\langle \ell,\mathbb{E}_{z\sim\mu}\left[S(z)_s\right]\Big\rangle\right)ds + \int_0^t \sigma\left(s,X^{\mu}_s,\Big\langle \ell,\mathbb{E}_{z\sim\mu}\left[S(z)_s\right]\Big\rangle\right)dW_s,
\end{equation}
with $t\in[0,T]$. By the Lipschitz continuity of the coefficients, \eqref{eq:sig_MKV_SDE_fixed_mu} admits a pathwise unique strong solution.

Given $\mu,\nu \in \mathcal{P}(C^d)$, let $X^{\mu}$ and $Y^{\nu}$ denote the corresponding solutions to \eqref{eq:sig_MKV_SDE_fixed_mu}, and let
\begin{equation*}
    \augpr{X}^{\mu}=\left(1,\augpr{X}^{1,\mu},\augpr{X}^{2,\mu}\right),\qquad \augpr{Y}^{\nu}=\left(1,\augpr{Y}^{1,\nu},\augpr{Y}^{2,\nu}\right)
\end{equation*}
be their respective time-augmented geometric $p$-rough path lifts. The following estimate controls the $p$-variation of the first-level increments of the lifted paths.
\begin{proposition}\label{prop:p_var_norm_estimate_m=1}
    Let $p\in(2,3)$. For every $t\in[0,T]$,
    \begin{equation*}
        \begin{aligned}
            &\mathbb{E}\left[\sup_D\sum_{i=0}^{n-1} \Big\lVert \augpr{X}^{\mu,1}_{t_i,t_{i+1}}-\augpr{Y}^{\nu,1}_{t_i,t_{i+1}} \Big\rVert^p_{\mathbb{R}^{d+1}}\right]\\
            &=\mathbb{E}\left[\sup_D\sum_{i=0}^{n-1} \Big\lVert\left(X^{\mu}_{t_{i+1}}-X^{\mu}_{t_i}\right)-\left(Y^{\nu}_{t_{i+1}}-Y^{\nu}_{t_i}\right)\Big\rVert^p_{\mathbb{R}^d}\right]\\
            &\leq \widetilde{C}\int_0^t\left(\mathbb{E}\left[\big\lVert\augpr{X}^{\mu}-\augpr{Y}^{\nu}\big\rVert_{\textup{$p$-var};[0,s]}^p\right] + \int_{\widehat{G\Omega}_p(\mathbb{R}^{d+1})\times \widehat{G\Omega}_p(\mathbb{R}^{d+1})}C_s\lVert x-y \rVert_{\textup{$p$-var};[0,s]}^p\pi(dx,dy)\right)ds,
        \end{aligned}
    \end{equation*}
    where $D=D_{[0,t]}$ is a partition of $[0,t]$, $\widetilde{C}=\widetilde{C}(p,T,L)$, $C_s=C_s(d,p,N,T,\alpha_\infty,\Gamma_s)$, $\alpha_\infty:=\max_{0\leq |I|\leq N}|\alpha_I|$ where $\alpha_I$ are the coefficients in \eqref{eq:linear_comb_in_sig_MKV_SDE}, $\Gamma_s:=\max\{\lVert x\rVert_{\textup{$p$-var};[0,s]},\lVert y\rVert_{\textup{$p$-var};[0,s]}\}$, and $\pi\in\Pi(\mu,\nu)$ is an arbitrary coupling of $\mu$ and $\nu$.
\end{proposition}
\begin{proof}
    See Appendix \ref{appendix:prop:p_var_norm_estimate_m=1_PROOF}.
\end{proof}

Next, the $p$-variation of the second-level increments of the lifted paths is controlled by the following estimate.
\begin{proposition}\label{prop:p_var_norm_estimate_m=2}
    Let $p\in[5/2,3)$. For every $t\in[0,T]$,
    \begin{align*}
        &\mathbb{E}\left[\sup_D\left(\sum_{i=0}^{n-1} \Big\lVert\left(\augpr{X}^{\mu,2}_{t_{i+1}}-\augpr{X}^{\mu,2}_{t_i}\right)-\left(\augpr{Y}^{\nu,2}_{t_{i+1}}-\augpr{Y}^{\nu,2}_{t_i}\right)\Big\rVert^{p/2}_{(\mathbb{R}^{d+1})^{\otimes 2}}\right)^2\right]\\
        &\leq \widetilde{C}\int_0^t\left(\mathbb{E}\left[\big\lVert\augpr{X}^{\mu}-\augpr{Y}^{\nu}\big\rVert_{\textup{$p$-var};[0,s]}^p\right] + \int_{\widehat{G\Omega}_p(\mathbb{R}^{d+1})\times \widehat{G\Omega}_p(\mathbb{R}^{d+1})}C_s\lVert x-y \rVert_{\textup{$p$-var};[0,s]}^p\pi(dx,dy)\right)ds,,
    \end{align*}
    where $D=D_{[0,t]}$ is a partition of $[0,t]$, $\widetilde{C}=\widetilde{C}(p,T,L,M_b,M_\sigma)$, $C_s=C_s(d,p,N,T,\alpha_\infty,\Gamma_s)$, $\alpha_\infty:=\max_{0\leq |I|\leq N}|\alpha_I|$ where $\alpha_I$ are the coefficients in \eqref{eq:linear_comb_in_sig_MKV_SDE}, $\Gamma_s:=\max\{\lVert x\rVert_{\textup{$p$-var};[0,s]},\lVert y\rVert_{\textup{$p$-var};[0,s]}\}$, and $\pi\in\Pi(\mu,\nu)$ is an arbitrary coupling of $\mu$ and $\nu$.
\end{proposition}
\begin{proof}
    See Appendix \ref{appendix:prop:p_var_norm_estimate_m=2_PROOF}.
\end{proof}

We are now ready to proceed with the proof of strong existence and uniqueness for a solution of \eqref{eq:sig_MKV_SDE}.

\subsection{Existence and uniqueness result}
The notions of strong unique solution and weak solution unique in law, as in Definition \ref{def:strong}, Definition \ref{def:strong_unique}, Definition \ref{def:weak}, and Definition \ref{def:law_unique}, apply unchanged to the signature McKean--Vlasov SDE \eqref{eq:sig_MKV_SDE}. Consequently, we obtain the following well-posedness result for equation \eqref{eq:sig_MKV_SDE}.

\begin{theorem}\label{thm:well_posed_sig_MKV}
    Under the assumption of Hypothesis \ref{hypothesis_A}, the signature McKean--Vlasov SDE \eqref{eq:sig_MKV_SDE} admits a strong unique solution as in Definition \ref{def:strong} and Definition \ref{def:strong_unique}. Moreover, it admits a weak solution unique in the sense of probability law as in Definition \ref{def:weak} and Definition \ref{def:law_unique}.
\end{theorem}
\begin{proof}
    Throughout the proof, it is assumed $p\in[5/2,3)$.
    
    \textbf{Step 1.} Because of the local Lipschitz continuity of the signature map (see Corollary \ref{corollary:local_Lip}), we need to start proving existence and uniqueness in the case the time-augmented $p$-rough path lift of the process is projected onto a ball of radius $n\in\mathbb{N}$.
    
    Fixed $n\in\mathbb{N}$, let $B_n(0)$ be the centered ball of radius $n$ in $L^p\left(\Omega,\mathcal{F},\mathbb{P};\widehat{G\Omega}_p(\mathbb{R}^{d+1})\right)$  and $\pi_n$ be the projection onto $B_n(0)$. Specifically,
    \begin{equation*}
        \pi_n:L^p\left(\Omega,\mathcal{F},\mathbb{P};\widehat{G\Omega}_p(\mathbb{R}^{d+1})\right)\rightarrow B_n(0);\quad \pi_n(\augpr{X}_\cdot):=\left(1,\lambda_n \augpr{X}_\cdot^{(1)},\lambda_n^2 \augpr{X}_\cdot^{(2)}\right),
    \end{equation*}
    where
    \begin{equation*}
        \lambda_n := \frac{n}{\lVert \augpr{X}\rVert_{\textup{$p$-var};[0,\cdot]}}\wedge 1.
    \end{equation*}
    
    The map $\pi_n$ is 1-Lipschitz continuous with respect to the inhomogeneous $p$-variation distance and we define the push-forward measure of $\pi_n$ over $\mathcal{P}^p\left(\widehat{G\Omega}_p(\mathbb{R}^{d+1})\right)$ as the map
    \begin{equation*}
        \pi_{n\#}\cdot:\mathcal{P}^p\left(\widehat{G\Omega}_p(\mathbb{R}^{d+1})\right)\mapsto\mathcal{P}^p\left(\widehat{G\Omega}_p(\mathbb{R}^{d+1})\right),\quad \pi_{n\#}\mu:=\mathcal{L}(\pi_n(\augpr{X})),
    \end{equation*}
    where $\augpr{X}\in L^p\left(\Omega,\mathcal{F},\mathbb{P};G\Omega_p(\mathbb{R}^{d+1}\right)$ and $\mathcal{L}(\augpr{X})=\mu$.

    The map $\pi_{n\#}$ is well-defined. In fact, whenever $\augpr{X},\augpr{Y}\in L^p\left(\Omega,\mathcal{F},\mathbb{P};\widehat{G\Omega}_p(\mathbb{R}^{d+1})\right)$ have the same law, then $\mathcal{L}(\pi_n(\augpr{X}))=\mathcal{L}(\pi_n(\augpr{Y}))$.
    
    Moreover, $\pi_{n\#}$ inherits the 1-Lipschitz continuity of $\pi_n$. Indeed, let $\mu,\nu\in\mathcal{P}^p\left(\widehat{G\Omega}_p(\mathbb{R}^{d+1})\right)$ and choose $\augpr{X},\augpr{Y}\in L^p\left(\Omega,\mathcal{F},\mathbb{P};\widehat{G\Omega}_p(\mathbb{R}^{d+1})\right)$ with laws $\mu$ and $\nu$, respectively. Then, for every $t\in[0,T]$,
    \begin{equation*}
        \left[\overline{\mathcal{W}}^p_t(\pi_{n\#}\mu,\pi_{n\#}\nu)\right]^p \leq \mathbb{E}\left[\lVert \pi_n(\augpr{X})-\pi_n(\augpr{Y})\rVert_{\textup{$p$-var};[0,t]}^p\right] \leq \mathbb{E}\left[\lVert \augpr{X}-\augpr{Y}\rVert_{\textup{$p$-var};[0,t]}^p\right].
    \end{equation*}
    Taking the infimum over the space of couplings of $\mu$ and $\nu$, namely $\Pi(\mu,\nu)$, on both sides of the above inequality,
    \begin{equation}\label{eq:Lipschitz_cont_map_pi_n}
        \left[\overline{\mathcal{W}}^p_t(\pi_{n\#}\mu,\pi_{n\#}\nu)\right]^p \leq \left[\overline{\mathcal{W}}^p_t(\mu,\nu)\right]^p.
    \end{equation}

    As an auxiliary step to prove well-posedness for the signature McKean--Vlasov equation, we introduce the following McKean--Vlasov-type equation
    \begin{equation}\label{eq:aux_truncated_MKV}
        X^n_t = \eta + \int_0^t b^n\left(s,X^n_s,\Big\langle \ell,\mathbb{E}\left[S(\augpr{X}^n)_s\right]\Big\rangle\right)ds + \int_0^t \sigma^n\left(s,X^n_s,\Big\langle \ell,\mathbb{E}\left[S(\augpr{X}^n)_s\right]\Big\rangle\right)dW_s,
    \end{equation}
    with $t\in[0,T]$,
    \begin{equation*}
        b^n\left(t,X^n_t,\Big\langle \ell,\mathbb{E}\left[S(\augpr{X}^n)_t\right]\Big\rangle\right):=b\left(t,X_t^n,\Big\langle \ell,\mathbb{E}\left[S(\pi_n(\augpr{X}^n))_t\right]\Big\rangle\right),
    \end{equation*}
    and $\sigma^n$ defined similarly.

    Fixed $\mu\in\mathcal{P}^p\left(\widehat{G\Omega}_p(\mathbb{R}^{d+1})\right)$, there exists a pathwise unique strong solution to the following equation
    \begin{equation}\label{eq:truncated_sig_MKV_SDE_ball_fixed_mu}
        X_t^{n,\mu} = \eta + \int_0^t b^n\left(s,X_s^{n,\mu},\Big\langle\ell,\mathbb{E}_{z\sim\mu}\left[S(z)_s\right]\Big\rangle\right)ds + \int_0^t \sigma^n\left(s,X_s^{n,\mu},\Big\langle\ell,\mathbb{E}_{z\sim\mu}\left[S(z)_s\right]\Big\rangle\right)dW_s,
    \end{equation}
    with $t\in[0,T]$. In fact, the drift and diffusion coefficients $b^n$ and $\sigma^n$ are Lipschitz continuous:
    \begin{align*}
        &\begin{aligned}
            &\Big\lVert b^n\left(t,x,\Big\langle\ell,\mathbb{E}_{z\sim\mu}\left[S(z)_t\right]\Big\rangle\right)-b^n\left(t,y,\Big\langle\ell,\mathbb{E}_{z\sim\mu}\left[S(z)_t\right]\Big\rangle\right)\Big\rVert_{\mathbb{R}^d}\\
            &+\Big\lVert \sigma^n\left(t,x,\Big\langle\ell,\mathbb{E}_{z\sim\mu}\left[S(z)_t\right]\Big\rangle\right)-\sigma^n\left(t,y,\Big\langle\ell,\mathbb{E}_{z\sim\mu}\left[S(z)_t\right]\Big\rangle\right)\Big\rVert_{\mathbb{R}^d}
        \end{aligned}\\
        &\begin{aligned}
            =&\Big\lVert b\left(t,x,\Big\langle\ell,\mathbb{E}_{z\sim\pi_{n\#}\mu}\left[S(z)_t\right]\Big\rangle\right)-b\left(t,y,\Big\langle\ell,\mathbb{E}_{z\sim\pi_{n\#}\mu}\left[S(z)_t\right]\Big\rangle\right)\Big\rVert_{\mathbb{R}^d}\\
            &+\Big\lVert \sigma\left(t,x,\Big\langle\ell,\mathbb{E}_{z\sim\pi_{n\#}\mu}\left[S(z)_t\right]\Big\rangle\right)-\sigma\left(t,y,\Big\langle\ell,\mathbb{E}_{z\sim\pi_{n\#}\mu}\left[S(z)_t\right]\Big\rangle\right)\Big\rVert_{\mathbb{R}^d}
        \end{aligned}\\
        &\leq L\lVert x-y\rVert_{\mathbb{R}^d},
    \end{align*}
    where the first inequality is due to the Lipschitz continuity of the drift term (see Hypothesis \ref{hypothesis_A}), while the second one is obtained by the 1-Lipschitz continuity of the map $\pi_n$. The same arguments apply for $\sigma^n$.
    
    Let us consider the map
    \begin{equation}\label{eq:map_Psi_truncated_MKV}
        \Psi:\mathcal{P}^p(\widehat{G\Omega}_p(\mathbb{R}^{d+1}))\rightarrow\mathcal{P}^p(\widehat{G\Omega}_p(\mathbb{R}^{d+1})),\qquad \Psi(\mu)=\mathcal{L}(\augpr{X}^{n,\mu}).
    \end{equation}
    
    It is straightforward to see $\Psi$ is well-defined on $\mathcal{P}^p(\widehat{G\Omega}_p(\mathbb{R}^{d+1}))$. Indeed, by the same arguments as in Proposition \ref{prop:a_priori_estimate}, we obtain
    \begin{equation*}
        \mathbb{E}\left[\lVert \augpr{X}^{n,\mu}\rVert_{\textup{$p$-var};[0,T]}\right] <\infty,
    \end{equation*}
    for every $\mu\in\mathcal{P}^p(\widehat{G\Omega}_p(\mathbb{R}^{d+1}))$.
    
    Similarly to the case of the path-dependent McKean--Vlasov SDE \eqref{eq:path_dep_MKV}, by proving that $\Psi$ admits a unique fixed point, we obtain that the auxiliary McKean--Vlasov SDE \eqref{eq:aux_truncated_MKV} admits a weak solution. Again, as a byproduct of the computations, we obtain that also uniqueness in the sense of probability law and pathwise uniqueness holds for equation \eqref{eq:aux_truncated_MKV}. Then, by Yamada and Watanabe's result we conclude that the SDE \eqref{eq:aux_truncated_MKV} also admits a strong solution. This is done in Step 2. In Step 3, we remove the constraint of bounded $p$-variation.
    
    \textbf{Step 2.} 
    Let $\mu,\nu\in\mathcal{P}^p\left(\widehat{G\Omega}_p(\mathbb{R}^{d+1})\right)$. By \eqref{eq:map_Psi_truncated_MKV}, $\Psi(\mu)$ denotes the law of the $p$-rough path lift of the solution to the SDE
    \begin{equation*}
        X_t^{n,\mu} = \eta + \int_0^t b^n\left(s,X_s^{n,\mu},\Big\langle\ell,\mathbb{E}_{z\sim\mu}\left[S(z)_s\right]\Big\rangle\right)ds + \int_0^t \sigma^n\left(s,X_s^{n,\mu},\Big\langle\ell,\mathbb{E}_{z\sim\mu}\left[S(z)_s\right]\Big\rangle\right)dW_s
    \end{equation*}
    and $\Psi(\nu)$ denotes the law of the $p$-rough path lift of the solution to the SDE
    \begin{equation*}
        Y_t^{n,\nu} = \eta + \int_0^t b^n\left(s,Y_s^{n,\nu},\Big\langle\ell,\mathbb{E}_{z\sim\nu}\left[S(z)_s\right]\Big\rangle\right)ds + \int_0^t \sigma^n\left(s,Y_s^{n,\nu},\Big\langle\ell,\mathbb{E}_{z\sim\nu}\left[S(z)_s\right]\Big\rangle\right)dW_s
    \end{equation*}
    for every $t\in[0,T]$.
    
    By Proposition \ref{prop:p_var_norm_estimate_m=1} and Proposition \ref{prop:p_var_norm_estimate_m=2},
    \begin{align*}
        &\mathbb{E}\left[\Big\lVert \augpr{X}^{\mu,n}-\augpr{Y}^{\nu,n}\Big\rVert^p_{\textup{$p$-var};[0,t]}\right]&&\\
        &\leq C\left(\int_0^t\mathbb{E}\left[\lVert \augpr{X}^{\mu,n}-\augpr{Y}^{\nu,n}\rVert_{\textup{$p$-var};[0,s]}^p\right]ds + \int_0^t \int_{\widehat{G\Omega}_p(\mathbb{R}^{d+1})\times \widehat{G\Omega}_p(\mathbb{R}^{d+1})}C_s\lVert x-y \rVert_{\textup{$p$-var};[0,s]}^p\pi(dx,dy)ds\right),
    \end{align*}
    where $C=C(d,n,p,N,T,\alpha_\infty,L,M_b,M_\sigma)$ with $\alpha_\infty$ as in Proposition \ref{prop:p_var_norm_estimate_m=1}, and $\pi\in\Pi({\pi_n}_{\#}\mu,{\pi_n}_{\#}\nu)$ is an arbitrary coupling of ${\pi_n}_{\#}\mu$ and ${\pi_n}_{\#}\nu$.

    By taking the infimum over the space $\Pi(\mu,\nu)$ of couplings of $\mu$ and $\nu$ on both sides of the previous inequality,
    \begin{equation*}
        \begin{aligned}
            &\mathbb{E}\left[\Big\lVert \augpr{X}^{\mu,n}-\augpr{Y}^{\nu,n}\Big\rVert^p_{\textup{$p$-var};[0,t]}\right]\\
            &\leq C\left(\int_0^t\mathbb{E}\left[\lVert \augpr{X}^{\mu,n}-\augpr{Y}^{\nu,n}\rVert_{\textup{$p$-var};[0,s]}^p\right]ds + \int_0^t \left[\overline{\mathcal{W}}^p_s(\pi_{n\#}\mu,\pi_{n\#}\nu)\right]^pds\right).
        \end{aligned}
    \end{equation*}

    Therefore, by Gronwall's lemma,
    \begin{equation}\label{eq:well_posed_res1}
        \mathbb{E}\left[\Big\lVert \augpr{X}^{\mu,n}-\augpr{Y}^{\nu,n}\Big\rVert^p_{\textup{$p$-var};[0,s]}\right]\leq Ce^{CT}\int_0^t \left[\overline{\mathcal{W}}^p_s(\pi_{n\#}\mu,\pi_{n\#}\nu)\right]^pds\leq Ce^{CT}\int_0^t \left[\overline{\mathcal{W}}^p_s(\mu,\nu)\right]^pds,,
    \end{equation}
    where the last inequality is due to \eqref{eq:Lipschitz_cont_map_pi_n}.

    Finally, we obtain
    \begin{equation}\label{eq:well_posed_res2}
        \begin{aligned}
            \left[\overline{\mathcal{W}}^p_t\left(\Psi(\mu),\Psi(\nu)\right)\right]^p &\leq \mathbb{E}\left[\Big\lVert \augpr{X}^{\mu,n}-\augpr{Y}^{\nu,n}\Big\rVert^p_{\textup{$p$-var};[0,s]}\right]\\
            &\leq Ce^{CT}\int_0^t \left[\overline{\mathcal{W}}^p_s(\mu,\nu)\right]^pds,
        \end{aligned}
    \end{equation}
    where the first inequality is because we are selecting a specific coupling, the second one is due to \eqref{eq:well_posed_res1}.

    Assume we have two weak solutions $X^{\mu,n}$ and $Y^{\nu,n}$ such that $\mathcal{L}(X^{\mu,n})=\mu$ and $\mathcal{L}(Y^{\nu,n})=\nu$. Then, $\Psi(\mu)=\mu$ and $\Psi(\nu)=\nu$ and by \eqref{eq:well_posed_res2},
    \begin{equation*}
       \left[\overline{\mathcal{W}}_T^p(\mu,\nu)\right]^p=\left[\overline{\mathcal{W}}_T^p(\Psi(\mu),\Psi(\nu))\right]^p\leq Ce^{CT}\int_0^T \left[\overline{\mathcal{W}}^p_s(\mu,\nu)\right]^pds.
    \end{equation*}
    
    By Gronwall's lemma,
    \begin{equation}\label{eq:weak_uniq_truncated_MKV}
        \overline{\mathcal{W}}_T^p(\mu,\nu)=0.
    \end{equation}

    Assume there exist two strong solutions $X^{\mu,n}$ and $Y^{\nu,n}$. By substituting \eqref{eq:weak_uniq_truncated_MKV} in \eqref{eq:well_posed_res1}, we obtain that for every $t\in[0,T]$,
    \begin{align}
        \mathbb{E}\left[\big\lVert X_t^{n,\mu}-Y_t^{n,\nu}\big\rVert_{\mathbb{R}^d}^p\right]&=\mathbb{E}\left[\big\lVert \left(X_t^{n,\mu}-Y_t^{n,\nu}\right)-\left(X_0^{n,\mu}-Y_0^{n,\nu}\right)\big\rVert_{\mathbb{R}^d}^p\right]\label{eq:pathwise_unique_proof_1}\\
        &\leq \mathbb{E}\left[\sup_D\sum_{i=0}^{n-1}\Big\lVert\left(X_{t_{i+1}}^{n,\mu}-X_{t_i}^{n,\mu}\right)-\left(Y_{t_{i+1}}^{n,\nu}-Y_{t_i}^{n,\nu}\right)\Big\rVert_{\mathbb{R}^d}^p\right]\label{eq:pathwise_unique_proof_2}\\
        &\leq \mathbb{E}\left[\Big\lVert \augpr{X}^{\mu,n}-\augpr{Y}^{\nu,n}\Big\rVert^p_{\textup{$p$-var};[0,T]}\right] = 0\notag,
    \end{align}
    where \eqref{eq:pathwise_unique_proof_1} is due to the fact that $X_0^{n,\mu}=Y_0^{n,\nu}=\eta$ and \eqref{eq:pathwise_unique_proof_2} is because we have picked as partition $D=D_{[0,t]}$ of $[0,t]$ the trivial partition $D=\{0,t\}$. Therefore, for every $t\in[0,T]$, $X_t^{n,\mu}=Y_t^{n,\nu}$. By the continuity of the processes $X^{n,\mu}$ and $Y^{n,\nu}$, this implies pathwise uniqueness for SDE \eqref{eq:aux_truncated_MKV}.
   
    Now, fix a measure $m\in \mathcal{P}^p\left(\widehat{G\Omega}_p(\mathbb{R}^{d+1})\right)$. By iteratively applying the map $\Psi$  to the measure $m$, we define a sequence of measures $\{m_k\}_{k\in \mathbb{N}}$ in $\mathcal{P}^p\left(\widehat{G\Omega}_p(\mathbb{R}^{d+1})\right)$  such that
    \begin{equation}\label{eq:def_m_k_ex_uni_SDE_MKV_aux}
        m_k\,:=\,  \Psi^k(m) = \Psi\left(\Psi^{k-1}(m)  \right).
    \end{equation} 
    
    The sequence $\{m_k\}_{k\in\mathbb{N}}$ is a Cauchy sequence in $\left(\mathcal{P}^p\left(\widehat{G\Omega}_p(\mathbb{R}^{d+1})\right), \overline{\mathcal{W}_T^p})\right)$. Indeed, by iterating \eqref{eq:well_posed_res2}, 
    \begin{equation*}
         \left[\overline{\mathcal{W}}^p_T\left( \Psi^{k+1}(m),\Psi^k(m) \right)\right]^p  \leq \frac{C^ke^{kCT}}{k!} \left[\overline{\mathcal{W}}_T^p\left( \Psi(m),m \right)\right]^p\xrightarrow{k\rightarrow \infty}0.
    \end{equation*}
 
    By Theorem \ref{thm:equivalence_Wasserstein_weak_conv}, there exists a probability measure $\bar{\mu}\in\mathcal{P}^p\left(\widehat{G\Omega}_p(\mathbb{R}^{d+1})\right)$ such that $m_k$ weakly converges to  $\bar{\mu}$ as $k$ increases to $\infty$, and from \eqref{eq:def_m_k_ex_uni_SDE_MKV_aux} we get $\bar{\mu}=\Psi(\bar{\mu})=\mathcal{L}(\augpr{X}^{n,\bar{\mu}})$. This concludes the proof of existence of a weak solution to the auxiliary McKean--Vlasov SDE \eqref{eq:aux_truncated_MKV}. By Yamada and Watanabe's theorem \cite[p. 308]{KaratzasShreve}, equation \eqref{eq:aux_truncated_MKV} admits a strong unique solution and a weak solution unique in the sense of probability law.

    \textbf{Step 3.} Finally, we take the limit $n\rightarrow\infty$.
    
    Let $Y^n$ be the stochastic process defined as $Y^n:=X_{\cdot\wedge\rho_n}^n$, where
    \begin{equation*}
        \rho_n := \inf_{t\geq 0}\Big\{\lVert \augpr{X}\rVert_{\textup{$p$-var};[0,t]} > n\Big\}.
    \end{equation*}
    By the previous step, $Y^n$ is a strong solution to the signature McKean--Vlasov SDE \eqref{eq:sig_MKV_SDE} on the interval $[0,\rho_n]$. Moreover, since $Y^n$ is pathwise unique, we have that $Y^n=Y^{n^{\prime}}$ on $[0,\rho_{n^{\prime}}]$ and $\rho_{n^{\prime}}\leq\rho_n$ whenever $n\geq n^{\prime}$. This allows us to define the process $X$ by $X_t=Y^n_t$ for $t\in[0,\rho_n]$. Clearly, $X$ solves the signature McKean--Vlasov SDE \eqref{eq:sig_MKV_SDE} up to $\rho\wedge T$, where $\rho=\lim_{n\rightarrow\infty}\rho_n$. Following the argument in Proposition A.2 in \cite{hambly2023control}, we now show that $\rho=\infty$ almost surely.

    By the a priori estimate in Proposition \ref{prop:a_priori_estimate}, for any $t\in[0,T]$,
    \begin{equation*}
        \mathbb{E}\left[\lVert \augpr{X}\rVert_{\textup{$p$-var};[0,t\wedge\rho]}\right]\leq\mathbb{E}\left[\lVert \augpr{X}\rVert_{\textup{$p$-var};[0,T]}\right] <\infty,
    \end{equation*}

    Assume $\mathbb{P}\left(\rho<\infty\right)>0$. Then, we can find a large enough $t\geq 0$ such that $q = \mathbb{P}\left(\rho\leq t\right) >0$. Further, we can choose $n\geq 1$ such that $qn > \mathbb{E}\left[\lVert \augpr{X}\rVert_{\textup{$p$-var};[0,t\wedge\rho]}\right]$. This leads to a contradiction, in fact
    \begin{equation*}
        qn > \mathbb{E}\left[\lVert \augpr{X}\rVert_{\textup{$p$-var};[0,t\wedge\rho]}\right]\geq\mathbb{E}\left[\mathbbm{1}_{\rho_n\leq t} \lVert \augpr{X}\rVert_{\textup{$p$-var};[0,\rho_n]}\right]\geq \mathbb{P}\left(\rho_n\leq t\right)n \geq \mathbb{P}\left(\rho\leq t\right)n = qn.
    \end{equation*}
\end{proof}

Finally, let us remark that by the same arguments as above, Theorem \ref{thm:well_posed_sig_MKV} can be easily extended to consider signature McKean--Vlasov SDE \eqref{eq:sig_MKV_SDE} with two different linear combinations of signature terms in the drift and in the diffusion coefficient; that is,
\begin{equation*}
    X_t = \eta + \int_0^t b\left(s,X_s,\Big\langle \ell,\mathbb{E}\left[S(\augpr{X})_s\right]\Big\rangle\right)ds + \int_0^t\sigma\left(s,X_s,\Big\langle \widetilde{\ell},\mathbb{E}\left[S(\augpr{X})_s\right]\Big\rangle\right)dW_s,
\end{equation*}
with $t\in[0,T]$.

\section{Approximation result}\label{sec:approximation_result}
In this section, we prove how signature McKean--Vlasov SDEs \eqref{eq:sig_MKV_SDE} can approximate the solution of the path-dependent McKean--Vlasov SDEs introduced in Section \ref{sec:well_posed_path_dep_MKV}. We have the following global universal approximation result.
\begin{theorem}\label{thm:approx_MKV_SDE}
    Fixed $\varepsilon>0$, let $\mathcal{G}$ and $\widetilde{\mathcal{G}}$ be two functionals as in Hypothesis \ref{hypothesis_B} also continuous over $\widehat{C}^{\alpha,p}_o$, and $\langle\ell,\cdot\rangle$ and $\langle\widetilde{\ell},\cdot\rangle$ be as in Theorem \ref{thm:CuchieroTeichmannGUAT} such that
    \begin{equation*}
        \sup_{t\in[0,T]}\sup_{\widehat{\mathbf{x}}\in \widehat{C}^{\alpha,p}_o} \frac{ \big|\mathcal{G}(\mathbf{x}^{(1)}_{\cdot\wedge t}) -\langle\ell,S(\widehat{\mathbf{x}})_t\rangle\big| }{ \psi(\widehat{\mathbf{x}}_{\cdot\wedge t}) } < \varepsilon;\qquad \sup_{t\in[0,T]}\sup_{\widehat{\mathbf{x}}\in \widehat{C}^{\alpha,p}_o} \frac{ \big|\widetilde{\mathcal{G}}(\mathbf{x}^{(1)}_{\cdot\wedge t}) -\langle\widetilde{\ell},S(\widehat{\mathbf{x}})_t\rangle\big| }{ \psi(\widehat{\mathbf{x}}_{\cdot\wedge t}) } < \widetilde{\varepsilon}.
    \end{equation*}
    
    Under the assumption of Hypothesis \ref{hypothesis_A}, let us consider the processes $(X,Y)$ being solutions to the equations
    \begin{equation*}
        \begin{aligned}
            &X_t = \eta + \int_0^t b\left(s,X_s,\Big\langle \ell,\mathbb{E}\left[S(\augpr{X})_s\right]\Big\rangle\right)ds
            + \int_0^t\sigma\left(s,X_s,\Big\langle \widetilde{\ell},\mathbb{E}\left[S(\augpr{X})_s\right]\Big\rangle\right)dW_s;\\
            &Y_t = \eta + \int_0^t b\left(s,Y_s,\mathbb{E}\left[\mathcal{G}(Y_{\cdot\wedge s})\right]\right)ds 
            + \int_0^t\sigma\left(s,Y_s,\mathbb{E}\left[\widetilde{\mathcal{G}}(Y_{\cdot\wedge s})\right]\right)dW_s,
        \end{aligned}
    \end{equation*}
    with $t\in[0,T]$. Then there exists $\delta=\delta(\varepsilon)>0$ such that
    \begin{equation*}
        \mathbb{E}\left[\big\lVert Y - X\big\rVert^2_{\infty,T}\right]\leq \delta,
    \end{equation*}
    where
    \begin{equation*}
        \delta=\delta(\varepsilon):= 2TLe^{4TL(1+2(L_\mathcal{G}\vee L_{\widetilde{\mathcal{G}}}))T}e^{2\beta(2T)^\gamma}\left(\varepsilon\vee\widetilde{\varepsilon}\right)^2
    \end{equation*}
\end{theorem}
\begin{proof}
    By Cauchy-Schwartz's and H\"{o}lder's inequalities,
    \begin{align*}
        &\mathbb{E}\left[\big\lVert Y - X\big\rVert^2_{\infty,T}\right]\\
        &\begin{aligned}
            \leq &2T\mathbb{E}\left[\int_0^T\Big\lVert b\left(t,Y_t,\mathbb{E}\left[\mathcal{G}(Y_{\cdot\wedge t})\right]\right) - b\left(t,X_t,\Big\langle \ell,\mathbb{E}\left[S(\augpr{X})_t\right]\Big\rangle\right)\Big\rVert^2dt\right]\\
            &+2\mathbb{E}\left[\Bigg\lVert \int_0^\cdot\left[\sigma\left(t,Y_t,\mathbb{E}\left[\widetilde{\mathcal{G}}(Y_{\cdot\wedge t})\right]\right) - \sigma\left(t,X_t,\Big\langle \widetilde{\ell},\mathbb{E}\left[S(\augpr{X})_t\right]\Big\rangle\right)\right]dW_t\Bigg\rVert^2_{\infty,T}\right]:=A+B.
        \end{aligned}
    \end{align*}

    As the term $A$ concerns, by the Lipschitz continuity assumption on the drift coefficient $b$ (see Hypothesis \ref{hypothesis_A}),
    \begin{equation*}
        \begin{aligned}
            A &\leq 2TL\mathbb{E}\left[\int_0^T\left(\big\lVert Y_t - X_t \big\rVert_{\mathbb{R}^d}^2 + \Big|\mathbb{E}\left[\mathcal{G}(Y_{\cdot\wedge t}) - \Big\langle \ell,S(\augpr{X})_t\Big\rangle\right]\Big|^2\right)dt\right]\\
            &\leq 2TL\int_0^T\mathbb{E}\left[\big\lVert Y - X \big\rVert_{\infty,t}^2\right]dt + 2TL\int_0^T\mathbb{E}\left[\Big|\mathcal{G}(Y_{\cdot\wedge t}) - \Big\langle \ell,S(\augpr{X})_t\Big\rangle\Big|^2\right]dt,
        \end{aligned}
    \end{equation*}
    where the second inequality is due to Tonelli's theorem and Jensen's inequality.

    Regarding the term $B$, by the usual Burkholder-Davis-Gundy inequality (see, e.g., \cite[Theorem 3.28]{KaratzasShreve}) and the Lipschitz continuity assumption on the diffusion coefficient $\sigma$ (see Hypothesis \ref{hypothesis_A}), we obtain the same bound as for the term $A$:
    \begin{equation*}
        B \leq 2TL\int_0^T\mathbb{E}\left[\big\lVert Y - X \big\rVert_{\infty,t}^2\right]dt + 2TL\int_0^T\mathbb{E}\left[\Big|\widetilde{\mathcal{G}}(Y_{\cdot\wedge t}) - \Big\langle \widetilde{\ell},S(\augpr{X})_t\Big\rangle\Big|^2\right]dt.
    \end{equation*}

    Now, since $\sup_{\widehat{\mathbf{x}}\in \widehat{C}^{\alpha,p}_o}1/\psi(\widehat{\mathbf{x}}) =1$, for every $t\in[0,T]$,
    \begin{equation*}
        \begin{aligned}
            &\Big|\mathcal{G}(Y_{\cdot\wedge t}) - \Big\langle \ell,S(\augpr{X})_t\Big\rangle\Big|\\
            &\leq \Big|\mathcal{G}(Y_{\cdot\wedge t}) - \mathcal{G}(X_{\cdot\wedge t})\Big| + \Big|\mathcal{G}(X_{\cdot\wedge t}) - \langle\ell,S(\augpr{X})_t\rangle\Big|\\
            &= L_\mathcal{G}\big\lVert Y-X\big\rVert_{\infty,t} + \Big|\mathcal{G}(X_{\cdot\wedge t}) - \langle\ell,S(\augpr{X})_t\rangle\Big|\frac{1}{\psi(\widehat{\mathbf{x}}_{\cdot\wedge t})}\psi(\widehat{\mathbf{x}}_{\cdot\wedge t})\\
            &\leq L_\mathcal{G}\big\lVert Y-X\big\rVert_{\infty,t} +  \sup_{\widehat{\mathbf{x}}\in \widehat{C}^{\alpha,p}_o} \frac{\Big|\mathcal{G}\left(\mathbf{x}^{(1)}_{\cdot\wedge t}\right)-\langle\ell,S(\widehat{\mathbf{x}})_t\rangle\Big|}{\psi(\widehat{\mathbf{x}}_{\cdot\wedge t})}\sup_{\widehat{\mathbf{x}}\in \widehat{C}^{\alpha,p}_o}\psi(\widehat{\mathbf{x}}_{\cdot\wedge t})\\
            &= L_\mathcal{G}\big\lVert Y-X\big\rVert_{\infty,t} + e^{\beta(2T)^\gamma}\varepsilon.
        \end{aligned}
    \end{equation*}

    Similarly,
    \begin{equation*}
        \Big|\widetilde{\mathcal{G}}(Y_{\cdot\wedge t}) - \Big\langle \widetilde{\ell},S(\augpr{X})_t\Big\rangle\Big|\leq L_{\widetilde{\mathcal{G}}}\big\lVert Y-X\big\rVert_{\infty,t} + e^{\beta(2T)^\gamma}\widetilde{\varepsilon}.
    \end{equation*}

    Therefore, 
    \begin{equation*}
        \mathbb{E}\left[\big\lVert Y - X\big\rVert^2_{\infty,T}\right]\leq 4TL(1+2(L_\mathcal{G}\vee L_{\widetilde{\mathcal{G}}}))\int_0^T\mathbb{E}\left[\Big\lVert Y - X \Big\rVert_{\infty,t}^2\right]dt + 2TLe^{2\beta(2T)^\gamma}(\varepsilon\vee\widetilde{\varepsilon})^2.
    \end{equation*}
    Finally, by Gronwall's lemma,
    \begin{equation*}
        \mathbb{E}\left[\big\lVert Y - X\big\rVert^2_{\infty,T}\right] \leq 2TLe^{2\beta(2T)^\gamma}(\varepsilon\vee\widetilde{\varepsilon})^2e^{4TL(1+2(L_\mathcal{G}\vee L_{\widetilde{\mathcal{G}}}))T}=\delta.
    \end{equation*}
\end{proof}

The previous result shows that signature McKean–Vlasov SDEs can approximate path-dependent McKean–Vlasov SDEs with arbitrary accuracy. In the next section, we show how signature McKean–Vlasov SDEs can in turn be approximated by their associated interacting particle system, which we refer to as \emph{signature mean-field particle system}, by establishing propagation of chaos.

\section{Signature mean-field particle system}\label{sec:PS_prop_of_chaos}
Let $N_P\in\mathbb{N}$ and consider the product probability space $\left( \Omega^{N_P},\mathcal{F}^{\otimes N_P}, \mathbb{P}^{\otimes N_P} \right)$ filtered by a natural extension of the original filtration to the product space and a family of independent $m$-dimensional Brownian motion $\{W^i\}_{i=1}^{N_P}$ adapted to it. Then the dynamics of the $i$-th particle is an $\mathbb{R}^d$-valued process satisfying the equation
\begin{equation}\label{eq:PS_sig_MKV}
    \begin{aligned}
        X_t^{i,N_P} \,=\, \eta\, + &\int_0^t b\left(s,X_s^{i,N_P},\frac{1}{N_P}\sum_{j=1}^{N_P}\Big\langle \ell,S(\augpr{X}^{j,N_P})_s\Big\rangle\right)ds\\
        &+ \int_0^t \sigma\left(s,X_s^{i,N_P},\frac{1}{N_P}\sum_{j=1}^{N_P}\Big\langle \ell,S(\augpr{X}^{j,N_P})_s\Big\rangle\right)dW^i_s,
    \end{aligned}
\end{equation}
with $t\in[0,T]$ and $i=1,\ldots,N_P$, where $\Big\langle \ell,S(\augpr{X}^{j,N_P})_s\Big\rangle$ is defined as in \eqref{eq:linear_comb_in_sig_MKV_SDE} for every $j\in\{1,\ldots,N_P\}$.

We have the following well-posedness result regarding the particle system \eqref{eq:PS_sig_MKV}.

\begin{theorem}
    Under the assumption of Hypothesis \ref{hypothesis_A}, the signature mean-field particle system \eqref{eq:PS_sig_MKV} admits a strong unique solution.
\end{theorem}
\begin{proof}
    \textbf{Step 1.} To start, the particle system is reformulated as an $\mathbb{R}^{dN_P}$-valued SDE.

    For every $t\in[0,T]$, let
    \begin{equation*}
        \mathbf{X}_t := \begin{bmatrix}
            X_t^{1,N_P}\\
            \vdots\\
            X_t^{N_P,N_P}
        \end{bmatrix}\in\mathbb{R}^{dN_P}.
    \end{equation*}

    Then the time-augmented geometric $p$-rough lift of $\mathbf{X}$ is $\augpr{X}_{s,t}:=(1,A_{s,t},B_{s,t})$, where
    \begin{equation*}
        A_{s,t}:= \begin{bmatrix}
                t-s\\
                \int_{s<u_1<t} \circ dX^{1,N_P}_{u_1}\\
                \vdots\\
                \int_{s<u_1<t} \circ dX^{N_P,N_P}_{u_1}
            \end{bmatrix}
    \end{equation*}
    and
    \begin{equation*}
        B_{s,t}:= \begin{bmatrix}
                \frac{1}{2}(t^2-s^2) &\int_s^t\int_s^udr\circ dX_u^{1,N_P\top} &\cdots &\int_s^t\int_s^udr\circ dX_u^{N_P,N_P\top}\\
                \int_s^t\int_s^u\circ dX^{1,N_P}_rdu &\int_s^t\int_s^u\circ dX^{1,N_P}_r\otimes \circ dX_u^{1,N_P\top} &\cdots &\int_s^t\int_s^u\circ dX^{1,N_P}_r\otimes \circ dX_u^{N_P,N_P\top}\\
                \vdots &\vdots &\ddots &\vdots  \\ 
                \int_s^t \int_s^u\circ dX^{N_P,N_P}_rdu &\int_s^t\int_s^u\circ dX^{N_P,N_P}_r\otimes \circ dX_u^{1,N_P\top} &\cdots &\int_s^t\int_s^u\circ dX^{N_P,N_P}_r\otimes \circ dX_u^{N_P,N_P\top}
            \end{bmatrix},
    \end{equation*}
    for every $0\leq s < t\leq T$.

    Therefore, for every $i=1,\ldots,N_P$ and every $0\leq s < t\leq T$, it is possible to rewrite the time-augmented geometric $p$-rough lift of $X^{i,N_P}$ as
    \begin{equation}\label{eq:rough_lift_PS_as_rough_lift_system}
        \augpr{X}_{s,t}^{i,N_P} = \left(1,\begin{bmatrix}
            A_{s,t}^1\\
            A_{s,t}^{i+1}
        \end{bmatrix},
        \begin{bmatrix}
            B_{s,t}^{(1,1)} &B_{s,t}^{(1,i+1)}\\
            B_{s,t}^{(i+1,1)} &B_{s,t}^{(i+1,i+1)}
        \end{bmatrix}\right)  =: \augpr{X}_{s,t}^{i}
    \end{equation}

    Now, for every $t\in[0,T]$, we also define
    \begin{equation*}
        \begin{aligned}
            &\mathbf{H}:=\begin{bmatrix}
                \eta\\
                \vdots\\
                \eta
            \end{bmatrix}\in\mathbb{R}^{dN_P};\quad \mathbf{W}_t := \begin{bmatrix}
            W_t^1\\
            \vdots\\
            W_t^{N_P}
        \end{bmatrix}\in\mathbb{R}^{mN_P};\\
        &B\left(t,\mathbf{X}_t,\frac{1}{N_P}\sum_{j=1}^{N_P}\Big\langle \ell,S(\augpr{X}^j)_t\Big\rangle\right) := \begin{bmatrix}
            b_t^{1,N_P}\\
            \vdots\\
            b_t^{N_P,N_P}
        \end{bmatrix}\in\mathbb{R}^{dN_P};\\
        &\Sigma\left(t,\mathbf{X}_t,\frac{1}{N_P}\sum_{j=1}^{N_P}\Big\langle \ell,S(\augpr{X}^j)_t\Big\rangle\right) := \begin{bmatrix}
        \sigma_t^{1,N_P} & 0 & 0\\
        0 & \ddots & 0 \\
        0 & 0 & \sigma_t^{N_P,N_P}
      \end{bmatrix}\in\mathbb{R}^{dN_P}\times \mathbb{R}^{mN_P},
        \end{aligned}
    \end{equation*}
    where, for every $i=1,\ldots,N_P$,
    \begin{equation*}
        b_t^{i,N_P} := b\left(t,X_t^{i,N_P},\frac{1}{N_P}\sum_{j=1}^{N_P}\Big\langle \ell,S(\augpr{X}^{j,N_P})_t\Big\rangle\right)=b\left(t,X_t^{i,N_P},\frac{1}{N_P}\sum_{j=1}^{N_P}\Big\langle \ell,S(\augpr{X}^j)_t\Big\rangle\right),
    \end{equation*}
    where the equality is due to \eqref{eq:rough_lift_PS_as_rough_lift_system}, and $\sigma_t^{i,N_P}$ is defined similarly.

    Then, the particle system can be rewritten as
    \begin{equation}\label{eq:PS_rewritten}
        \begin{aligned}
            \mathbf{X}_t = \mathbf{H} + &\int_0^t B\left(s,\mathbf{X}_s,\frac{1}{N_P}\sum_{j=1}^{N_P}\Big\langle \ell,S(\augpr{X}^j)_s\Big\rangle\right)ds\\
            &+ \int_0^t \Sigma\left(s,\mathbf{X}_s,\frac{1}{N_P}\sum_{j=1}^{N_P}\Big\langle \ell,S(\augpr{X}^j)_s\Big\rangle\right)d\mathbf{W}_s.
        \end{aligned}
    \end{equation}

    In the next steps, the existence of a strong solution to the particle system is established via a Picard–iteration argument (see, e.g., \cite[Theorem 9.2]{Baldi}) for \eqref{eq:PS_rewritten}. Specifically, in Step 2 we define a recurrence sequence of processes and we show that it is convergent uniformly on the time interval $[0,T]$; then, in Step 3, we show that the limit process satisfies \eqref{eq:PS_rewritten}.

    \textbf{Step 2.}  For every $t\in[0,T]$, we define the recurrence sequence of processes $\{\mathbf{X}^{(m)}\}_m$ as
    \begin{align}
        &\mathbf{X}^{(0)}_t = \mathbf{H};\notag\\
        &\mathbf{X}^{(m+1)}_t = \mathbf{H} + \int_0^t B\left(s,\mathbf{X}^{(m)}_s,\frac{1}{N_P}\sum_{j=1}^{N_P}\Big\langle \ell,S(\augpr{X}^{(m),j})_s\Big\rangle\right)ds\notag\\
        &\qquad\qquad\qquad\quad+ \int_0^t \Sigma\left(s,\mathbf{X}_s^{(m)},\frac{1}{N_P}\sum_{j=1}^{N_P}\Big\langle \ell,S(\augpr{X}^{(m),j})_s\Big\rangle\right)d\mathbf{W}_s.\label{eq:recursion_PS}
    \end{align}

    To show that the sequence $\{\mathbf{X}^{(m)}\}_m$ converges uniformly on the time interval $[0,T]$, we prove by induction that for every $t\in[0,T]$,
    \begin{equation}\label{eq:induction_PS}
        \mathbb{E}\left[\big\lVert \mathbf{X}^{(m+1)}_\cdot - \mathbf{X}^{(m)}_\cdot \big\rVert_{\infty,t}^p + \big\lVert \augpr{X}^{(m+1)}_\cdot - \augpr{X}^{(m)}_\cdot \big\rVert_{\textup{$p$-var};[0,t]}\right] \leq \frac{(Ct)^{m+1}}{(m+1)!},
    \end{equation}
    where $C$ is a positive constant and $\lVert\cdot \rVert_{\textup{$p$-var};[0,t]}$ denotes the inhomogeneous $p$-variation norm.

    For $m=0$, for every $u\in[0,t]$,
    \begin{equation*}
        \begin{aligned}
            \mathbf{X}^{(1)}_u - \mathbf{X}^{(0)}_u &= \int_0^u B\left(s,\mathbf{X}_0,\frac{1}{N_P}\sum_{j=1}^{N_P}\Big\langle \ell,S(\augpr{H}^{j})_s\Big\rangle\right)ds + \int_0^u \Sigma\left(s,\mathbf{X}_0,\frac{1}{N_P}\sum_{j=1}^{N_P}\Big\langle \ell,S(\augpr{H}^{j})_s\Big\rangle\right)dW_s\\
            &= \int_0^u B\left(s,\mathbf{X}_0,\langle \ell,S(\augpr{H}^1)_s\rangle\right)ds + \int_0^u \Sigma\left(s,\mathbf{X}_0,\langle \ell,S(\augpr{H}^1)_s\rangle\right)dW_s,
        \end{aligned}
    \end{equation*}
    where $\augpr{H}^j$ is the geometric $p$-rough lift of the time-augmentation of the $j$-th component of the process $\mathbf{H}$, $j=1,\ldots,N_P$, and the second equality is due to the definition of $\mathbf{H}$.

    Therefore,
    \begin{equation*}
        \begin{aligned}
            &\mathbb{E}\left[\lVert \mathbf{X}^{(1)}_\cdot - \mathbf{X}^{(0)}_\cdot\rVert_{\infty,t}\right]\\
            &\leq 2^{p-1} \mathbb{E}\left[\Bigg\lVert \int_0^\cdot B\left(s,\mathbf{X}_0,\langle \ell,S(\augpr{H}^1)_s\rangle\right)ds \Bigg\rVert_{\infty,t}^p\right] + 2^{p-1} \mathbb{E}\left[\Bigg\lVert\int_0^\cdot \Sigma\left(s,\mathbf{X}_0,\langle \ell,S(\augpr{H}^1)_s\rangle\right)dW_s\Bigg\rVert_{\infty,t}^p\right]\\
            &\leq 2^{p-1}t \mathbb{E}\left[\int_0^t \Big\lVert B\left(s,\mathbf{X}_0,\langle \ell,S(\augpr{H}^1)_s\rangle\right)\Big\rVert_{\mathbb{R}^{dN_P}}^pds\right] + 2^{p-1}t\mathbb{E}\left[\int_0^t \Big\lVert\Sigma\left(s,\mathbf{X}_0,\langle \ell,S(\augpr{H}^1)_s\rangle\right)\Big\rVert_F^pds\right]\\
            &\leq 2^{p-1}t(M_b^2+M_\sigma^2),
        \end{aligned}
    \end{equation*}
    where the first inequality is due to H\"{o}lder's inequality; the first term in the second inequality is obtained by Jensen's inequality, while the second term is due to Burkholder-Davis-Gundy inequality (see, e.g., \cite[Theorem 3.28]{KaratzasShreve}) and then Jensen's inequality; the third inequality is because of the boundedness assumption on the coefficients (see Hypothesis \ref{hypothesis_A}).

    By similar computations to Proposition \ref{prop:a_priori_estimate},
    \begin{equation*}
        \mathbb{E}\left[\big\lVert \augpr{X}^{(1)}_\cdot - \augpr{X}^{(0)}_\cdot \big\rVert_{\textup{$p$-var};[0,t]}\right]\leq Ct,
    \end{equation*}
    where $C$ is a positive constant.
    
    Thus,
    \begin{equation*}
        \mathbb{E}\left[\big\lVert \mathbf{X}^{(1)}_\cdot - \mathbf{X}^{(0)}_\cdot \big\rVert_{\infty,t}^p + \big\lVert \augpr{X}^{(1)}_\cdot - \augpr{X}^{(0)}_\cdot \big\rVert_{\textup{$p$-var};[0,t]}\right] \leq Ct,
    \end{equation*}
    where $C$ is a positive constant possibly different from the one above.

    Now, let us assume \eqref{eq:induction_PS} to be true for $m-1$ and we prove it for $m$.

    To do this, let $\rho_n$ be a stopping time defined as follows
    \begin{equation*}
        \rho_n:= \inf_{t\geq 0}\Big\{\lVert \augpr{X}\rVert_{\textup{$p$-var};[0,t]} > n\Big\}
    \end{equation*}
    and let $t\in[0,\rho_n)$.

    For every $s\in[0,T]$, we also define
    \begin{equation*}
        B_s^{(m)}:=B\left(s,\mathbf{X}^{(m)}_s,\frac{1}{N_P}\sum_{j=1}^{N_P}\Big\langle \ell,S(\augpr{X}^{(m),j})_s\Big\rangle\right),
    \end{equation*}
    and $B_s^{(m-1)}$, $\Sigma_s^{(m)}$, and $\Sigma_s^{(m-1)}$ are defined similarly.

    We have that
    \begin{equation*}
        \begin{aligned}
            &\mathbb{E}\left[\lVert\mathbf{X}^{(m+1)}_\cdot - \mathbf{X}^{(m)}_\cdot\rVert_{\infty,t}^p\right]\\
            &\leq 2^{p-1}t \int_0^t\mathbb{E}\left[\Big\lVert B_s^{(m)} - B_s^{(m-1)}\Big\rVert_{\mathbb{R}^{dN_P}}^p\right]ds + 2^{p-1}t\int_0^t\mathbb{E}\left[\Big\lVert \Sigma_s^{(m)} - \Sigma_s^{(m-1)}\Big\rVert_F^p\right]ds\\
            &\leq 2^pL^pT \int_0^t\mathbb{E}\left[\Big\lVert \mathbf{X}_s^{(m)} - \mathbf{X}_s^{(m-1)} \Big\rVert_{\mathbb{R}^{dN_P}}^p\right]ds\\
            &\quad+ 2^pL^p\alpha_{\infty}^p\widetilde{C}(d,N)^pT\frac{1}{N_P}\sum_{j=1}^{N_P}\int_0^t\mathbb{E}\left[\Big\lVert S^{\leq N}(\augpr{X}^{(m),j})_\cdot - S^{\leq N}(\augpr{X}^{(m-1),j})_\cdot\Big\rVert_{\textup{$p$-var;[0,s]}}^p\right]ds\\
            &\leq 2^pL^pT \int_0^t\mathbb{E}\left[\Big\lVert \mathbf{X}_s^{(m)} - \mathbf{X}_s^{(m-1)} \Big\rVert_{\mathbb{R}^{dN_P}}^p\right]ds\\
            &\quad+ \frac{2(L\alpha_{\infty}\widetilde{C}(d,N)n)^pT}{N_P}\int_0^t\mathbb{E}\left[\Big\lVert \augpr{X}^{(m)}_\cdot - \augpr{X}^{(m-1)}_\cdot\Big\rVert_{\textup{$p$-var;[0,s]}}^p\right]ds
        \end{aligned}
    \end{equation*}
    where $\alpha_{\infty}:=\max_{0\leq |I|\leq N}|\alpha_I|$ and $\alpha_I$ as defined in \eqref{eq:linear_comb_in_sig_MKV_SDE}, and
    \begin{equation*}
        \widetilde{C}(d,N):= \left(\frac{(d+1)^{N+1}-1}{d}\right)^{1/2}.
    \end{equation*}
    The first inequality is due to the same computations as above and Fubini's theorem; the second one by the Lipschitz continuity of the coefficients (see Hypothesis \ref{hypothesis_A}) and similar computations to Proposition \ref{prop:p_var_norm_estimate_m=1}; the third one by the local Lipschitz continuity of the signature map with respect to the inhomogeneous $p$-variation norm and fact that
    \begin{equation*}
        \sum_{j=1}^{N_P}\Big\lVert \augpr{X}^{(m),j}_\cdot - \augpr{X}^{(m-1),j}_\cdot\Big\rVert_{\textup{$p$-var;[0,s]}}^p \leq \Big\lVert \augpr{X}^{(m)}_\cdot - \augpr{X}^{(m-1)}_\cdot\Big\rVert_{\textup{$p$-var;[0,s]}}^p.
    \end{equation*}

    Moreover, by similar computations as in Proposition \ref{prop:p_var_norm_estimate_m=1} and Proposition \ref{prop:p_var_norm_estimate_m=2}, we also obtain
    \begin{equation*}
        \mathbb{E}\left[\Big\lVert \augpr{X}^{(m)}_\cdot - \augpr{X}^{(m-1)}_\cdot\Big\rVert_{\textup{$p$-var;[0,t]}}^p\right] \leq C\int_0^t\mathbb{E}\left[\Big\lVert \augpr{X}^{(m)}_\cdot - \augpr{X}^{(m-1)}_\cdot\Big\rVert_{\textup{$p$-var;[0,s]}}^p\right]ds,
    \end{equation*}
    where $C=C(p,T,L,M_b,M_\sigma,n)$ is a positive constant.

    Therefore, for every $t\in[0,\rho_n)$,
    \begin{equation*}
        \begin{aligned}
            &\mathbb{E}\left[\big\lVert \mathbf{X}^{(m+1)}_\cdot - \mathbf{X}^{(m)}_\cdot \big\rVert_{\infty,t}^p + \big\lVert \augpr{X}^{(m+1)}_\cdot - \augpr{X}^{(m)}_\cdot \big\rVert_{\textup{$p$-var};[0,t]}\right]\\
            &\leq C\int_0^t\mathbb{E}\left[\Big\lVert \mathbf{X}_s^{(m)} - \mathbf{X}_s^{(m-1)} \Big\rVert_{\mathbb{R}^{dN_P}}^p + \Big\lVert \augpr{X}^{(m)}_\cdot - \augpr{X}^{(m-1)}_\cdot\Big\rVert_{\textup{$p$-var;[0,s]}}^p\right]ds\\
            &\leq C\int_0^t \frac{(Ct)^m}{m!}ds = \frac{(Ct)^{m+1}}{(m+1)!},
        \end{aligned}
    \end{equation*}
    where the last inequality is due to the inductive hypothesis.

    To show that \eqref{eq:induction_PS} actually hold for every $t\in[0,T]$, following the same argument as in Theorem \ref{thm:well_posed_sig_MKV}, we define $\rho:=\lim_{n\rightarrow\infty} \rho_n$.
    
    By the boundedness assumption on the coefficients (see Hypothesis \ref{hypothesis_A}), by similar arguments as in Proposition \ref{prop:a_priori_estimate}, we have that $\rho<\infty$ almost surely and we conclude.

    Let
    \begin{equation*}
        A_m:=\Bigg\{\omega\in\Omega \,:\, \lVert \mathbf{X}^{(m+1)}_\cdot(\omega) - \mathbf{X}^{(m)}_\cdot(\omega) \rVert_{\infty,T} > \frac{1}{2^m} \Bigg\}.
    \end{equation*}
    
    By Markov's inequality,
    \begin{equation*}
        \begin{aligned}
            \mathbb{P}\left(A_m\right) &\leq 2^{pm}\mathbb{E}\left[\lVert \mathbf{X}^{(m+1)}_\cdot - \mathbf{X}^{(m)}_\cdot \rVert_{\infty,T}^p\right]\\
            &\leq2^{pm}\mathbb{E}\left[\big\lVert \mathbf{X}^{(m+1)}_\cdot - \mathbf{X}^{(m)}_\cdot \big\rVert_{\infty,T}^p + \big\lVert \augpr{X}^{(m+1)}_\cdot - \augpr{X}^{(m)}_\cdot \big\rVert_{\textup{$p$-var};[0,T]}\right]\\
            &\leq2^{pm}\frac{(CT)^{m+1}}{(m+1)!}.
        \end{aligned}
    \end{equation*}

    Since $\sum_{m=0}^\infty \mathbb{P}\left(A_m\right) <\infty$, by the Borel-Cantelli lemma, $\mathbb{P}\left(\limsup_{m\rightarrow\infty} A_m\right)=0$. Hence, almost surely, there exists a finite random index $M_0$ such that for every $m\geq M_0$,
    \begin{equation*}
        \lVert \mathbf{X}^{(m+1)}_\cdot - \mathbf{X}^{(m)}_\cdot \rVert_{\infty,T} \leq \frac{1}{2^m}.
    \end{equation*}
    It follows that the series
    \begin{equation*}
        \sum_{k=0}^{\infty} \left(\mathbf{X}^{(k+1)}-\mathbf{X}^{(k)}\right)
    \end{equation*}
    converges uniformly on $[0,T]$ almost surely. Consequently,
    \begin{equation*}
        \mathbf{X}^{(m)} = \mathbf{H} + \sum_{k=0}^{m-1} \left(\mathbf{X}^{(k+1)}-\mathbf{X}^{(k)}\right)
    \end{equation*}
    converges uniformly on $[0,T]$, almost surely, to a continuous process $\mathbf{X}$.

    \textbf{Step 3.} We now show that the limit process $\mathbf{X}$ is a strong solution to SDE \eqref{eq:PS_rewritten}.

    To start, let us remark that by the same argument for the convergence of $\{\mathbf{X}^{(m)}\}_m$, it is possible to prove that 
    $\{\augpr{X}^{(m)}\}_m$ converges to a $G^2(\mathbb{R}^{d+1})$-valued process $\augpr{Z}$; then, by the continuity of the Lyons' lift, $\augpr{Z}=\augpr{X}$, where $\augpr{X}$ denotes the time-augmented geometric $p$-rough path lift of $\mathbf{X}$.

    By similar computations as in Step 2,
    \begin{equation*}
        \lVert B_\cdot^{(m)} - B_\cdot \rVert_{\infty,T} \leq C\left(\lVert \mathbf{X}^{(m)}_\cdot - \mathbf{X}_\cdot\rVert_{\infty,T} + \lVert \augpr{X}^{(m)}_\cdot - \augpr{X}_\cdot\rVert_{\textup{$p$-var};[0,T]}\right),
    \end{equation*}
    where for every $s\in[0,T]$,
    \begin{equation}\label{eq:def_B}
        B_s:=B\left(s,\mathbf{X}_s,\frac{1}{N_P}\sum_{j=1}^{N_P}\Big\langle \ell,S(\augpr{X}^j)_s\Big\rangle\right).
    \end{equation}
    Therefore, uniformly on $[0,T]$,
    \begin{equation*}
        B\left(s,\mathbf{X}^{(m)}_s,\frac{1}{N_P}\sum_{j=1}^{N_P}\Big\langle \ell,S(\augpr{X}^{(m),j})_s\Big\rangle\right)\xrightarrow{m\rightarrow\infty}B\left(s,\mathbf{X}_s,\frac{1}{N_P}\sum_{j=1}^{N_P}\Big\langle \ell,S(\augpr{X}^j)_s\Big\rangle\right)\quad\textup{a.s.},
    \end{equation*}
    which implies
    \begin{equation*}
        \int_0^t B\left(s,\mathbf{X}^{(m)}_s,\frac{1}{N_P}\sum_{j=1}^{N_P}\Big\langle \ell,S(\augpr{X}^{(m),j})_s\Big\rangle\right)ds\xrightarrow{m\rightarrow\infty}\int_0^tB\left(s,\mathbf{X}_s,\frac{1}{N_P}\sum_{j=1}^{N_P}\Big\langle \ell,S(\augpr{X}^j)_s\Big\rangle\right)ds\quad\textup{a.s.}
    \end{equation*}

    Analogously, we have that, uniformly on $[0,T]$,
    \begin{equation*}
        \Sigma\left(s,\mathbf{X}^{(m)}_s,\frac{1}{N_P}\sum_{j=1}^{N_P}\Big\langle \ell,S(\augpr{X}^{(m),j})_s\Big\rangle\right)\xrightarrow{m\rightarrow\infty}\Sigma\left(s,\mathbf{X}_s,\frac{1}{N_P}\sum_{j=1}^{N_P}\Big\langle \ell,S(\augpr{X}^j)_s\Big\rangle\right)\quad\textup{a.s.},
    \end{equation*}
    which implies
    \begin{equation*}
        \int_0^T\Bigg\lVert \Sigma\left(s,\mathbf{X}^{(m)}_s,\frac{1}{N_P}\sum_{j=1}^{N_P}\Big\langle \ell,S(\augpr{X}^{(m),j})_s\Big\rangle\right) - \Sigma\left(s,\mathbf{X}_s,\frac{1}{N_P}\sum_{j=1}^{N_P}\Big\langle \ell,S(\augpr{X}^j)_s\Big\rangle\right) \Bigg\rVert_F^2ds\xrightarrow{m\rightarrow\infty}0\quad\textup{a.s.}
    \end{equation*}
    Then, \cite[Proposition B.41]{BainAlan2009FoSF}, it holds the convergence in probability of the stocahstic integral
    \begin{equation*}
        \int_0^t \Sigma\left(s,\mathbf{X}^{(m)}_s,\frac{1}{N_P}\sum_{j=1}^{N_P}\Big\langle \ell,S(\augpr{X}^{(m),j})_s\Big\rangle\right)d\mathbf{W}_s\xrightarrow{m\rightarrow\infty}\int_0^t\Sigma\left(s,\mathbf{X}_s,\frac{1}{N_P}\sum_{j=1}^{N_P}\Big\langle \ell,S(\augpr{X}^j)_s\Big\rangle\right)d\mathbf{W}_s.
    \end{equation*}

    Finally, taking the limit in probability on both side of \eqref{eq:recursion_PS}, we obtain that $\mathbf{X}$ is a strong solution to \eqref{eq:PS_rewritten} (equivalently, \eqref{eq:PS_sig_MKV}).

    \textbf{Step 4.} In this final step, it is established pathwise uniqueness for a solution to equation \eqref{eq:PS_rewritten}.

    Let $\mathbf{X},\mathbf{Y}$ be two strong solutions to \eqref{eq:PS_rewritten}. Also, let us define the following stopping times:
    \begin{equation*}
        \rho_n^X:=\inf_{t\geq 0}\Big\{\lVert \augpr{X}\rVert_{\textup{$p$-var};[0,t]} > n\Big\};\qquad \rho_n^Y:=\inf_{t\geq 0}\Big\{\lVert \augpr{Y}\rVert_{\textup{$p$-var};[0,t]} > n\Big\}.
    \end{equation*}

    In the following, $C$ denotes a positive constant that might change from line to line.
    
    For every $t\in[0,\rho_n^X\wedge\rho_n^Y)$, by the same computations as in Step 2,
    \begin{equation*}
        \mathbb{E}\left[\lVert \mathbf{X}_\cdot - \mathbf{Y}_\cdot\rVert_{\infty,t}\right]\leq C\int_0^t\mathbb{E}\left[\lVert\mathbf{X}_\cdot - \mathbf{Y}_\cdot\rVert_{\infty,s} + \lVert \augpr{X}_\cdot - \augpr{Y}_\cdot\rVert_{\textup{$p$-var};[0,s]}\right]ds.
    \end{equation*}
    
    Also, by the same arguments as in Proposition \ref{prop:p_var_norm_estimate_m=1} and Proposition \ref{prop:p_var_norm_estimate_m=2},
    \begin{equation*}
        \mathbb{E}\left[\lVert \augpr{X}_\cdot - \augpr{Y}_\cdot\rVert_{\textup{$p$-var};[0,t]}\right]\leq C\int_0^t\mathbb{E}\left[ \lVert \augpr{X}_\cdot - \augpr{Y}_\cdot\rVert_{\textup{$p$-var};[0,s]}\right]ds.
    \end{equation*}

    Therefore,
    \begin{equation*}
        \mathbb{E}\left[\lVert \mathbf{X}_\cdot - \mathbf{Y}_\cdot\rVert_{\infty,t} + \lVert \augpr{X}_\cdot - \augpr{Y}_\cdot\rVert_{\textup{$p$-var};[0,t]}\right]\leq C\int_0^t\mathbb{E}\left[\lVert\mathbf{X}_\cdot - \mathbf{Y}_\cdot\rVert_{\infty,s} + \lVert \augpr{X}_\cdot - \augpr{Y}_\cdot\rVert_{\textup{$p$-var};[0,s]}\right]ds.
    \end{equation*}

    By Gronwall's lemma,
    \begin{equation*}
        \mathbb{E}\left[\lVert \mathbf{X}_\cdot - \mathbf{Y}_\cdot\rVert_{\infty,t} + \lVert \augpr{X}_\cdot - \augpr{Y}_\cdot\rVert_{\textup{$p$-var};[0,t]}\right]=0,
    \end{equation*}
    which implies $\mathbb{E}\left[\lVert \mathbf{X}_\cdot - \mathbf{Y}_\cdot\rVert_{\infty,t}\right]=0$, i.e., pathwise uniqueness on the interval $[0,\rho_n^X\wedge\rho_n^Y)$.

    By the same arguments as in Proposition \ref{prop:a_priori_estimate},
    \begin{equation*}
        \rho^X:=\lim_{n\rightarrow\infty}\rho_n^X <\infty\quad\textup{a.s.};\qquad\rho^Y:=\lim_{n\rightarrow\infty}\rho_n^Y <\infty\quad\textup{a.s.}
    \end{equation*}
    Therefore, pathwise uniquness is established on the whole interval $[0,T]$.
\end{proof}

Now that we have established well-posedness for the \emph{signature mean-field particle system}, i.e., the particle system \eqref{eq:PS_sig_MKV} associated with the signature McKean--Vlasov SDE \eqref{eq:sig_MKV_SDE}, we establish the propagation of chaos property.

\subsection{Propagation of chaos}
Let us consider $N_P$ independent copies of the solution to the signature McKean--Vlasov SDE \eqref{eq:sig_MKV_SDE}. Specifically, for every $i\in\{1,\ldots,N_P\}$, $X^i$ is a solution to the signature McKean--Vlasov SDE
\begin{equation}\label{eq:copy_sig_MKV}
    X^i_t = \eta + \int_0^t b\left(s,X^i_s,\Big\langle \ell,\mathbb{E}\left[S(\augpr{X}^i)_s\right]\Big\rangle\right)ds + \int_0^t \sigma\left(s,X^i_s,\Big\langle \ell,\mathbb{E}\left[S(\augpr{X}^i)_s\right]\Big\rangle\right)dW^i_s,
\end{equation}
with $t\in[0,T]$.

Recall the definition of propagation of chaos \cite{ChaintronDiez}.
\begin{definition}\label{def:chaos_propagation}
    For any given  $N_P\in \mathbb N$, let $(X^{1,N_P},\ldots,X^{N_P,N_P})$ and $(X^1,\ldots,X^{N_P})$ be the solutions of the systems \eqref{eq:PS_sig_MKV} and \eqref{eq:copy_sig_MKV}, respectively. Propagation of chaos occurs whenever for any $k\ge 1$,
    \begin{equation*}
        \left( X^{1,N_P},\ldots,X^{k,N_P} \right)\xrightarrow[\mathcal{L}]{N_P\to \infty} \left( X^1,\ldots,X^k \right).
    \end{equation*}
\end{definition}

The following result establish propagation of chaos for the particle system \eqref{eq:PS_sig_MKV} associated with the signature McKean--Vlasov SDE \eqref{eq:sig_MKV_SDE}.
\begin{theorem}
   Let $N_P\in \mathbb N$, $(X^{1,N_P},\ldots,X^{N_P,N_P})$ be the solution to the system \eqref{eq:PS_sig_MKV}, and $(X^1,\ldots,X^{N_P})$ be the solution to the particle system \eqref{eq:copy_sig_MKV}. Then propagation of chaos occurs, according with Definition \ref{def:chaos_propagation}.
\end{theorem}
\begin{proof}
    To start, let us define
    \begin{equation*}
        \mu^N_t := \frac{1}{N_P}\sum_{i=1}^{N_P} \delta_{\augpr{X}^{i,N_P}_t};\qquad \nu^N_t := \frac{1}{N_P}\sum_{i=1}^{N_P} \delta_{\augpr{X}^{i}_t},
    \end{equation*}
    where $\delta_\cdot$ denotes the Dirac delta, and we denote with $\mu$ the law of the time-augmented geometric $p$-rough lift of the solution to the signature McKean--Vlasov SDE \eqref{eq:sig_MKV_SDE}.
    
    Following the proof of well-posedness for the signature McKean--Vlasov SDE \eqref{eq:sig_MKV_SDE}, we define the following stopping times:
    \begin{equation*}
        \rho_n^{\textup{PS}}:= \inf_{t\geq 0}\Big\{\lVert \augpr{X}^{i,N_P}\rVert_{\textup{$p$-var};[0,t]} > n\,\textup{for some }\,i\in{1,\ldots,N_P}\Big\}
    \end{equation*}
    and
    \begin{equation*}
        \rho_n^{\textup{IID}}:= \inf_{t\geq 0}\Big\{\lVert \augpr{X}^i\rVert_{\textup{$p$-var};[0,t]} > n\,\textup{for some }\,i\in{1,\ldots,N_P}\Big\},
    \end{equation*}
    where $\lVert\cdot\rVert_{\textup{$p$-var};[0,t]}$ denotes the inhomogeneous $p$-variation norm (see Definition \ref{def:inhom_p_var_norm}).
    
    Let $t\in[0,\rho_n^{\textup{PS}}\wedge\rho_n^{\textup{IID}})$ and $\overline{\mathcal{W}}^p_t$ be the Wasserstein distance defined in Definition \ref{def:Wasserstein_norm} with $\rho=d_{\textup{inhom $p$-var;$[0,t]$}}$. We have that
    \begin{equation}\label{eq:computation_prop_of_chaos}
        \begin{aligned}
            \mathbb{E}\left[\left(\overline{\mathcal{W}}^p_t(\mu^{N_P},\nu^{N_P})\right)^p\right]&\leq \mathbb{E}\left[\frac{1}{N_P}\sum_{i=1}^{N_P} \Big\lVert \augpr{X}^{i,N_P}_t - \augpr{X}^i_t \Big\rVert^p_{\textup{$p$-var};[0,t]}\right]\\
            &=\mathbb{E}\left[ \Big\lVert \augpr{X}^{1,N_P}_t - \augpr{X}^1_t \Big\rVert^p_{\textup{$p$-var};[0,t]}\right]\\
            &\leq  Ce^{CT}\int_0^t \mathbb{E}\left[\left(\overline{\mathcal{W}}^p_s(\mu^{N_P},\mu)\right)^p\right]ds,
        \end{aligned}
    \end{equation}
    where $C=C(d,n,p,N,T,\alpha_\infty,L,M_b,M_\sigma)$ with $\alpha_\infty$ as in Proposition \ref{prop:p_var_norm_estimate_m=1}. The first inequality is due to the fact that 
    \begin{equation*}
        \frac{1}{N_P}\sum_{i=1}^{N_P} \delta_{(\augpr{X}^{i,N_P}_t,\augpr{X}^{i}_t)}
    \end{equation*}
    is a coupling of $\mu^N_t$ and $\nu^N_t$; the second inequality is obtained by the same computations as in \eqref{eq:well_posed_res1}.

    Therefore,
    \begin{equation}\label{eq:computation_prop_of_chaos_2}
        \begin{aligned}
            \mathbb{E}\left[\left(\overline{\mathcal{W}}^p_t(\mu^{N_P},\mu)\right)^p\right] &\leq 2^{p-1}\mathbb{E}\left[\left(\overline{\mathcal{W}}^p_t(\mu^{N_P},\nu^{N_P})\right)^p\right] + 2^{p-1}\mathbb{E}\left[\left(\overline{\mathcal{W}}^p_t(\nu^{N_P},\mu)\right)^p\right]\\
            &\leq 2^{p-1}Ce^{CT}\int_0^t \mathbb{E}\left[\left(\overline{\mathcal{W}}^p_s(\mu^{N_P},\mu)\right)^p\right]ds + 2^{p-1}\mathbb{E}\left[\left(\overline{\mathcal{W}}^p_t(\nu^{N_P},\mu)\right)^p\right],
        \end{aligned}
    \end{equation}
    where the first inequality is due to Cauchy-Schwartz's inequality and the second one follows from \eqref{eq:computation_prop_of_chaos}.

    By applying Gronwall's inequality to \eqref{eq:computation_prop_of_chaos_2},
    \begin{equation*}
        \mathbb{E}\left[\left(\overline{\mathcal{W}}^p_t(\mu^{N_P},\mu)\right)^p\right] \leq 2^{p-1}\exp(2^{p-1}Ce^{CT}T)\mathbb{E}\left[\left(\overline{\mathcal{W}}^p_t(\nu^{N_P},\mu)\right)^p\right].
    \end{equation*}

     By the law of large numbers, since $\augpr{X}^i\sim\mu$ for every $i\in\{1,\ldots,N_P\}$,
    \begin{equation*}
        \mathbb{E}\left[\left(\overline{\mathcal{W}}^p_t(\nu^{N_P},\mu)\right)^p\right] \xrightarrow{N_P\rightarrow\infty} 0.
    \end{equation*}

    Now, fixed $k\geq 1$,
    \begin{equation*}
        \begin{aligned}
            \mathbb{E}\left[\max_{i=1,\ldots,k}\lVert X^{i,N_P}-X^i\rVert_{\infty,t}^p\right] &\leq \sum_{i=1}^k \mathbb{E}\left[\lVert X^{i,N_P}-X^i\rVert_{\infty,t}^p\right]\\
            &\leq \sum_{i=1}^k\mathbb{E}\left[ \Big\lVert \augpr{X}^{i,N_P}_t - \augpr{X}^i_t \Big\rVert^p_{\textup{$p$-var};[0,t]}\right]\\
            &\leq kCe^{CT}\int_0^t \mathbb{E}\left[\left(\overline{\mathcal{W}}^p_s(\mu^{N_P},\mu)\right)^p\right]ds \xrightarrow{N_P\rightarrow\infty} 0.
        \end{aligned}
    \end{equation*}

    Finally, we have established propagation of chaos for the particle system \eqref{eq:PS_sig_MKV} associated with the signature McKean--Vlasov SDE \eqref{eq:sig_MKV_SDE} in the time interval $[0,\rho_n^{\textup{PS}}\wedge\rho_n^{\textup{IID}})$. Now, let us define
    \begin{equation*}
        \rho^{\textup{PS}}:=\lim_{n\rightarrow \infty}\rho_n^{\textup{PS}};\qquad \rho^{\textup{IID}}:=\lim_{n\rightarrow \infty}\rho_n^{\textup{IID}}.
    \end{equation*}
    By similar computations as in Proposition \ref{prop:a_priori_estimate}, it is possible to show that $\rho^{\textup{PS}},\,\rho^{\textup{IID}} = \infty$ almost surely. Thus, propagation of chaos for \eqref{eq:PS_sig_MKV} is established on the whole interval $[0,T]$.
    
\end{proof}
Hence, by propagation of chaos, the expectation of the linear combination of the signature of the geometric $p$-rough lift of the the time-augmented solution in the signature McKean--Vlasov SDE \eqref{eq:sig_MKV_SDE} can be approximated by the empirical average of the linear combination of the signatures of the geometric $p$-rough path lifts of the time-augmented dynamics of the particles in the associated $N_P$-particle system \eqref{eq:PS_sig_MKV}.

\vskip 0.5cm
\noindent
\paragraph{Acknowledgements.} This work was supported by the SURE-AI Centre grant 357482, Research Council of Norway (Fred Espen Benth and Salvador Ortiz-Latorre). The authors would like to thank Francesco Carlo De Vecchi (University of Pavia) for the helpful discussions.

\appendix
\section{Appendix}\label{appendix:proofs}
This appendix collects technical material omitted from the main text for readability. We start by recalling the definition of tree-like equivalence used in Theorem \ref{thm:signature_rough_case} and several auxiliary results needed in the sequel, and then give detailed proofs of Proposition \ref{prop:a_priori_estimate}, Proposition \ref{prop:p_var_norm_estimate_m=1}, and Proposition \ref{prop:p_var_norm_estimate_m=2}.
\medskip

\begin{definition}[\cite{Gromov1987}, Appendix B.1 \cite{Chevyrev2022}]\label{def:tree_like_eq}
    \begin{enumerate}[i)]
        \item An $\mathbb{R}$-tree is a metric space $\mathfrak{T}$ such that for every three points $x,y,\rho\in \mathfrak{T}$ there exists a point $c=x\wedge y$ such that the geodesic segments $[\rho ,x],[\rho ,y]$ intersect in the segment $[\rho ,c]$ and also $c\in [x,y]$.
        \item For a topological space $\mathcal{X}$, a function $x:[0,T_x]\rightarrow\mathcal{X}$ is called tree-like if $x$ is continuous and there exists an $\mathbb{R}$-tree $\mathfrak{T}$, a continuous function $\varphi:[0,T_x]\rightarrow\mathfrak{T}$, and a map $\psi:\mathfrak{T}\rightarrow\mathcal{X}$ such that $\varphi(0)=\varphi(T_x)$ and $x=\psi\circ\varphi$.
        \item Let
        \begin{equation*}
            \overleftarrow{x}:[0,T_x]\rightarrow\mathcal{X};\qquad \overleftarrow{x}(t):=x(T_x-t)
        \end{equation*}
        denote the time-reversal of $x$. For another function $y:[0,T_y]\rightarrow\mathcal{X}$, we denote by $x*y$ the concatenation of $x$ and $y$. That is,
        \begin{equation*}
            x*y:[0,T_x+T_y]\rightarrow\mathcal{X};\qquad x*y(t)=\begin{cases}
                x(t)\quad &t\in[0,T_x]\\
                y(t)\quad &t\in(T_x,T_x+T_y].
            \end{cases}
        \end{equation*}
        We say that $x$ and $y$ are tree-like-equivalent, and write $x\sim_t y$, if $x*\overleftarrow{y}$ is tree-like.
    \end{enumerate}
\end{definition}
\medskip

\begin{theorem}[Besov-variation embedding]\label{thm:Besov_variation_embedding}
    Let $q>1$, $\alpha\in(1/q,1)$, and let $x\in C([0,T],E)$, where $(E,d)$ is a metric space. For $0\leq s<t\leq T$, define
    \begin{equation*}
        |x|^q_{W^{\alpha,q};[s,t]} :=\int\int_{[s,t]^2}\frac{d(x_v,x_u)^q}{|v-u|^{1+q\alpha}}dudv.
    \end{equation*}
    Then there exists a constant $C=C(\alpha,q)$ such that, for all $0\leq s<t\leq T$,
    \begin{equation*}
        \sup_D\left(\sum_{i=0}^{n-1} d(x_{t_i},x_{t_{i+1}})^{1/\alpha}\right)^\alpha \leq C|t-s|^{\alpha-1/q}|x|_{_{W^{\alpha,q};[s,t]}},
    \end{equation*}
    where $D=D_{[s,t]}$ is a partition of $[s,t]$.
\end{theorem}
\begin{proof}
    For a complete proof of the theorem, see \cite[Corollary A.3]{Friz_Victoir_2010}.
\end{proof}
\medskip

\begin{theorem}[L\'{e}pingle's inequality]\label{thm:Lepingle}
    For every $q\in[1,\infty)$, there exists a constant $C_q<\infty$ such that, for every $p>2$, and every martingale $M:[0,\infty)\rightarrow\mathbb{R}^d$, we have
    \begin{equation*}
        \mathbb{E}\left[\sup_D\left(\sum_{i=0}^{n-1} \Big\lVert M_{t_{i+1}}-M_{t_i}\Big\rVert^p_{\mathbb{R}^d}\right)^{q/p}\right] \leq C_q\frac{p}{p-2}\mathbb{E}\left[\lVert M\rVert^q_{\infty}\right]
    \end{equation*}
    where $D$ is a partition of $[0,\infty)$.
\end{theorem}
\begin{proof}
    For a complete proof of the theorem, see \cite{Lepingle}.
\end{proof}
\medskip

\begin{theorem}[Young--L\'{o}eve's inequality]\label{thm:Young_Loeve}
    Given $x\in C^{\textup{$p$-var}}([0,T],\mathbb{R}^d)$ and $y\in C^{\textup{$q$-var}}([0,T];L(\mathbb{R}^d,\mathbb{R}^e))$, where $L(\mathbb{R}^d,\mathbb{R}^e)$ denotes the space of linear functional from $\mathbb{R}^d$ to $\mathbb{R}^e$, there exists a unique (indefinite) Young integral of $y$ against $x$ denoted by $\int_0^\cdot ydx$. Moreover, the indefinite Young integral has finite $p$-variation and
    \begin{equation*}
        \Bigg|\int_0^\cdot ydx\Bigg|_{p-\textup{var};[s,t]} \leq C|x|_{p-\textup{var};[s,t]}\left(|y|_{p-\textup{var};[s,t]} + |y_0|\right),
    \end{equation*}
    where $C=C(p,q)$.
\end{theorem}
\begin{proof}
    For a complete proof of the theorem, see  \cite[Proposition 6.8]{Friz_Victoir_2010}
\end{proof}

We now move the the proof of the aforementioned results.

\subsection{Proof of Proposition \ref{prop:a_priori_estimate}}\label{appendix:prop:a_priori_estimate}
Let us recall that
\begin{equation*}
    \augpr{X}_{s,t}:=\left(1,\underbrace{\begin{bmatrix}
        t-s\\X_t-X_s
    \end{bmatrix}}_{\augpr{X}^{(1)}_{s,t}},\underbrace{\begin{bmatrix}
        \frac{1}{2}(t-s)^2 & \int_s^t(u-s)dX^\top_u\\
        \int_s^t(X_u-X_s)du & \int_s^t (X_u-X_s)\circ dX^\top_u
    \end{bmatrix}}_{\augpr{X}^{(2)}_{s,t}}\right)
\end{equation*}
and
\begin{equation*}
    \begin{aligned}
        \mathbb{E}\left[\lVert \augpr{X}\rVert_{\textup{$p$-var};[0,T]}\right] &= \mathbb{E}\left[\max_{m\in\{1,2\}}\sup_D\left(\sum_{i=0}^{n-1}\lVert\augpr{X}^{(m)}_{t_i,t_{i+1}}\rVert_{\mathbb{R}^{(d+1)\otimes m}}^{p/m}\right)^{m/p}\right]\\
        &\leq \sum_{m=1,2} \mathbb{E}\left[\sup_D\left(\sum_{i=0}^{n-1}\lVert\augpr{X}^{(m)}_{t_i,t_{i+1}}\rVert_{\mathbb{R}^{(d+1)\otimes m}}^{p/m}\right)^{m/p}\right],
    \end{aligned}
\end{equation*}
where $D=D_{[0,T]}$ is a partition of $[0,T]$.

As the first level $m=1$ regards,
\begin{align}
        &\mathbb{E}\left[\sup_{D}\sum_{i=0}^{n-1}\big\lVert\augpr{X}^{(1)}_{t_i,t_{i+1}}\big\rVert_{\mathbb{R}^{d+1}}^p\right] =\mathbb{E}\left[\sup_{D}\sum_{i=0}^{n-1}\Big(|t_{i+1}-t_i| + \big\lVert X_{t_{i+1}} - X_{t_i}\big\rVert_{\mathbb{R}^d}\Big)^p\right]\notag\\[0.5em]
        &\le 
        \begin{aligned}[t]
            &(3T)^{p-1} + 3^{p-1}\mathbb{E}\Bigg[\sup_{D}\sum_{i=0}^{n-1} 
                \Bigg\lVert \int_{t_i}^{t_{i+1}} b\left(s,X_s,\Big\langle\ell,\mathbb{E}\left[S(\augpr{X})_s\right]\Big\rangle\right)ds\Bigg\rVert_{\mathbb{R}^d}^p\Bigg] \\
            &+ 3^{p-1}\mathbb{E}\Bigg[\sup_{D}\sum_{i=0}^{n-1} 
                \Bigg\lVert \int_{t_i}^{t_{i+1}}\sigma\left(s,X_s,\Big\langle\ell,\mathbb{E}\left[S(\augpr{X})_s\right]\Big\rangle\right)dW_s\Bigg\rVert_{\mathbb{R}^d}^p\Bigg]
        \end{aligned} \notag\\[0.5em]
        &\le 
        \begin{aligned}[t]
            &(3T)^{p-1} + (3T)^{p-1}\mathbb{E}\Bigg[\int_0^T 
                \Big\lVert b\left(s,X_s,\Big\langle\ell,\mathbb{E}\left[S(\augpr{X})_s\right]\Big\rangle\right)\Big\rVert_{\mathbb{R}^d}^p ds\Bigg] \\
            &+ 3^{p-1} \frac{p}{p-2}C_p\mathbb{E}\Bigg[\sup_{0\leq t\leq T}\Bigg\lVert \int_0^t 
                \sigma\left(s,X_s,\Big\langle\ell,\mathbb{E}\left[S(\augpr{X})_s\right]\Big\rangle\right) dW_s\Bigg\rVert_{\mathbb{R}^d}^p\Bigg],
        \end{aligned}\label{eq:proof_bound_p_var_eq1}
\end{align}
where the first inequality is due to Jensen's inequality, while the second one is because of is because of Jensen's inequality in the first term and L\'{e}pingle's inequality (Theorem \ref{thm:Lepingle}) in the second term.

Then, for the second term in \eqref{eq:proof_bound_p_var_eq1}, by the multidimensional Burkholder-Davis-Gundy inequality (see, e.g., \cite[Theorem 3.28]{KaratzasShreve}) and Jensen's inequality, we get
\begin{equation}\label{eq:proof_bound_p_var_eq2}
    \begin{aligned}
        &\mathbb{E}\Bigg[\sup_{0\leq t\leq T}\Bigg\lVert \int_0^t \sigma\left(s,X_s,\Big\langle\ell,\mathbb{E}\left[S(\augpr{X})_s\right]\Big\rangle\right) dW_s\Bigg\rVert_{\mathbb{R}^d}^p\Bigg]\\
        &\leq \widetilde{C}_p\mathbb{E}\Bigg[\left(\int_0^T \Big\lVert\sigma\left(s,X_s,\Big\langle\ell,\mathbb{E}\left[S(\augpr{X})_s\right]\Big\rangle\right)\Big\rVert_F^{p/2} ds\right)^2\Bigg]\\
        &\leq\widetilde{C}_pT^{p-1}\mathbb{E}\Bigg[\int_0^T \Big\lVert\sigma\left(s,X_s,\Big\langle\ell,\mathbb{E}\left[S(\augpr{X})_s\right]\Big\rangle\right)\Big\rVert_F^p ds\Bigg]
    \end{aligned}
\end{equation}

By substituting \eqref{eq:proof_bound_p_var_eq2} in \eqref{eq:proof_bound_p_var_eq1} and the boundedness assumptions on the coefficients $b$ and $\sigma$ (see Hypothesis \ref{hypothesis_A}), we finally obtain
\begin{equation*}
    \mathbb{E}\left[\sup_{D}\sum_{i=0}^{n-1}\big\lVert\augpr{X}^{(1)}_{t_i,t_{i+1}}\big\rVert_{\mathbb{R}^{d+1}}^p\right]\leq (3T)^{p-1}\left(1+M_b^p+\frac{p}{p-2}C_p\widetilde{C}_pT^{\frac{2-p}{2}}M_\sigma^p\right).
\end{equation*}

As the second level $m=2$ regards,
\begin{align*}
        &\mathbb{E}\left[\sup_{D}\left(\sum_{i=0}^{n-1}\big\lVert\augpr{X}^{(2)}_{t_i,t_{i+1}}\big\rVert_{\mathbb{R}^{(d+1)\otimes 2}}^{p/2}\right)^2\right]\\
        &\begin{aligned}
            \,\leq\,&4^{p/2-2}\sup_D \left(\sum_{i=0}^{n-1}(t_{i+1}-t_i)^p\right)^2&\quad:=A\\
            &+4^{p-2}\mathbb{E}\left[\sup_D \left(\sum_{i=0}^{n-1}\Bigg\lVert\int_{t_i}^{t_{i+1}}(u-t_i)dX^\top_u\Bigg\rVert_{\mathbb{R}^d}^{p/2}\right)^2\right]&\quad:=B\\
            &+4^{p-2}\mathbb{E}\left[\sup_D \left(\sum_{i=0}^{n-1}\Bigg\lVert\int_{t_i}^{t_{i+1}}(X_u-X_{t_i})du\Bigg\rVert_{\mathbb{R}^d}^{p/2}\right)^2\right]&\quad:=C\\
            &+4^{p-2}\mathbb{E}\left[\sup_D \left(\sum_{i=0}^{n-1}\Bigg\lVert\int_{t_i}^{t_{i+1}}(X_u-X_{t_i})\circ dX^\top_u\Bigg\rVert_F^{p/2}\right)^2\right]&\quad:=D.
        \end{aligned}
\end{align*}

The term $A$ can be bounded by $4^{p/2-2}T^{2p-2}$.

For every $u\in[0,T]$, let us define
\begin{equation*}
    b_u := b\left(u,X_u,\Big\langle\ell,\mathbb{E}[S(\augpr{X})_u]\Big\rangle\right);\qquad \sigma_u := \sigma\left(u,X_u,\Big\langle\ell,\mathbb{E}[S(\augpr{X})_u]\Big\rangle\right).
\end{equation*}

As for the term $B$, since the norm is unchanged by transposition, we have that
\begin{equation*}
    \Bigg\lVert\int_{t_i}^{t_{i+1}}(u-t_i)dX^\top_u\Bigg\rVert_{\mathbb{R}^d} = \Bigg\lVert\int_{t_i}^{t_{i+1}}(u-t_i)b_udu + \int_{t_i}^{t_{i+1}}(u-t_i)\sigma_udW_u\Bigg\rVert_{\mathbb{R}^d}.
\end{equation*}
Therefore,
\begin{equation*}
4^{p-2}\mathbb{E}\left[\sup_D \left(\sum_{i=0}^{n-1}\Bigg\lVert\int_{t_i}^{t_{i+1}}(u-t_i)dX^\top_u\Bigg\rVert_{\mathbb{R}^d}^{p/2}\right)^2\right]\leq 4^{p-1/2}\left(B_1 + B_2\right),
\end{equation*}
where
\begin{equation*}
    B_1 := \mathbb{E}\left[\sup_D \left(\sum_{i=0}^{n-1}\Bigg\lVert\int_{t_i}^{t_{i+1}}(u-t_i)b_udu\Bigg\rVert_{\mathbb{R}^d}^{p/2} \right)^2 \right]
\end{equation*}
and
\begin{equation*}
    B_2 := \mathbb{E}\left[\sup_D \left(\sum_{i=0}^{n-1}\Bigg\lVert\int_{t_i}^{t_{i+1}}(u-t_i)\sigma_udW_u\Bigg\rVert_{\mathbb{R}^d}^{p/2}\right)^2\right].
\end{equation*}
Regarding the first term $B_1$,
\begin{equation*}
    B_1 \leq \mathbb{E}\left[\sup_D \left(\sum_{i=0}^{n-1}\int_{t_i}^{t_{i+1}}|u-t_i|^{p/2}\big\lVert b_u\big\rVert_{\mathbb{R}^d}^{p/2}du\right)^2 \right]\leq T^{p+1}M_b^p,
\end{equation*}
where the first inequality is due to Jensen's inequality while the second one is obtained by the boundedness of the drift coefficient $b$ (see Hypothesis \ref{hypothesis_A}).

To estimate the term $B_2$, we first apply the Besov-variation embedding (Theorem \ref{thm:Besov_variation_embedding}) with $q=3/2$ and $1/\alpha=p/2$ with $p\in[5/2,3)$. For some constant $C=C(p,q)>0$,
    \begin{align*}
        B_2 &= \mathbb{E}\left[\sup_D \left(\sum_{i=0}^{n-1}\Bigg\lVert\int_{t_i}^{t_{i+1}}(u-t_i)\sigma_udW_u\Bigg\rVert_{\mathbb{R}^d}^{p/2}\right)^2\right]\\
        &\leq C^pt^{\frac{p(p/2-1)}{q}}\mathbb{E}\left[\left(\int\int_{[0,t]^2}\frac{\Big\lVert\int_u^v (s-u)\sigma_s dW_s\Big\rVert_{\mathbb{R}^d}^q}{|u-v|^{1+2q/p}}dudv\right)^{p-q}\right]\\
        &\leq C^pt^{\frac{p(p/2-1)}{q}}t^{2(p-q-1)}\mathbb{E}\left[\int\int_{[0,t]^2}\frac{\Big\lVert\int_u^v (s-u)\sigma_s dW_s\Big\rVert_{\mathbb{R}^d}^p}{|u-v|^{1+2q/p}}dudv\right],
    \end{align*}
    where the last inequality is due to Jensen's theorem since $p-q \geq 1$.

    Now, let $C(t):=t^{\frac{p(p/2-1)}{q}}t^{2(p-q-1)}$. By Tonelli's theorem,
    \begin{align*}
        B_2 &\leq C^pC(t)\int\int_{[0,t]^2}\frac{\mathbb{E}\left[\Big\lVert\int_u^v (s-u)\sigma_s dW_s\Big\rVert_{\mathbb{R}^d}^p\right]}{|u-v|^{1+2q/p}}dudv\\
        &\leq C^pC(t)\int\int_{[0,t]^2}\frac{\mathbb{E}\left[\left(\int_u^v(s-u)^2\big\lVert \sigma_s\big\rVert_F^2ds\right)^{p/2}\right]}{|u-v|^{1+2q/p}}dudv\\
        &\leq C^pC(t)\int\int_{[0,t]^2}\frac{\mathbb{E}\left[\int_u^v(s-u)^p\big\lVert \sigma_s\big\rVert_F^pds\right]}{|u-v|^{1+2q/p-(p/2-1)}}dudv\\
        &\leq C^pC(t)M_\sigma^p\frac{1}{p+1}\int\int_{[0,t]^2}\frac{|u-v|^{p+1}}{|u-v|^{1+2q/p-(p/2-1)}}dudv<\infty,
    \end{align*}
    where we have applied the usual Burkholder-Davis-Gundy inequality (see, e.g., \cite[Theorem 3.28]{KaratzasShreve}), Jensen's inequality since $p/2\geq 1$, and the boundedness assumption on the diffusion coefficient $\sigma$ (see Hypothesis \ref{hypothesis_A}). The last integral is finite since $1+2q/p-(p/2-1)-(p+1)<0$  for $p\in[5/2,3)$.

    As the term $C$ concerns,
    \begin{equation*}
        C = \mathbb{E}\left[\sup_D \left(\sum_{i=0}^{n-1}\Bigg\lVert\int_{t_i}^{t_{i+1}}\left(\int_{t_i}^ub_rdr + \int_{t_i}^u\sigma_rdW_r\right)du\Bigg\rVert_{\mathbb{R}^d}^{p/2}\right)^2\right] \leq 2^{p-1}\left(C_1 + C_2\right),
    \end{equation*}
    where
    \begin{equation*}
        C_1:=\mathbb{E}\left[\sup_D \left(\sum_{i=0}^{n-1}\Bigg\lVert\int_{t_i}^{t_{i+1}}\int_{t_i}^ub_rdrdu\Bigg\rVert_{\mathbb{R}^d}^{p/2}\right)^2\right]
    \end{equation*}
    and
    \begin{equation*}
        C_2:=\mathbb{E}\left[\sup_D \left(\sum_{i=0}^{n-1}\Bigg\lVert\int_{t_i}^{t_{i+1}}\int_{t_i}^u\sigma_rdW_rdu\Bigg\rVert_{\mathbb{R}^d}^{p/2}\right)^2\right].
    \end{equation*}
    The first term $C_1$ can easily be bounded with a similar argument to the one for $B_1$. That is,
    \begin{equation*}
        \begin{aligned}
            C_1 &=\mathbb{E}\left[\sup_D \left(\sum_{i=0}^{n-1}\Bigg\lVert\int_{t_i}^{t_{i+1}}\int_{t_i}^ub_rdrdu\Bigg\rVert_{\mathbb{R}^d}^{p/2}\right)^2\right]\\
            &\leq T^{2p}\mathbb{E}\left[\int_0^T\int_0^T\big\lVert b_r\big\rVert_{\mathbb{R}^d}^pdrdu\right]\leq T^{2p+2}M_b^p.
        \end{aligned}
    \end{equation*}

    As for the second term $C_2$, we have that
    \begin{align}
        &\mathbb{E}\left[\sup_D \left(\sum_{i=0}^{n-1}\Bigg\lVert\int_{t_i}^{t_{i+1}}\int_{t_i}^u\sigma_rdW_rdu\Bigg\rVert_{\mathbb{R}^d}^{p/2}\right)^2\right]\notag\\
        &\leq T^p\mathbb{E}\left[\sup_D \left(\sum_{i=0}^{n-1}\int_{t_i}^{t_{i+1}}\Bigg\lVert\int_{t_i}^u\sigma_rdW_r\Bigg\rVert_{\mathbb{R}^d}^{p/2}du\right)^2\right]\label{eq1:dim_bound_p_var}\\
        &\leq (2T)^p\mathbb{E}\left[\sup_D \left(\sum_{i=0}^{n-1}\int_{t_i}^{t_{i+1}}\Bigg\lVert\int_0^\cdot \sigma_rdW_r\Bigg\rVert_{\infty,u}^{p/2}du\right)^2\right]\notag\\
        &= (2T)^p\mathbb{E}\left[\left(\int_0^t\Bigg\lVert\int_0^\cdot \sigma_rdW_r\Bigg\rVert_{\infty,u}^{p/2}du\right)^2\right]\notag\\
        &\leq 2^pT^{p+1}\int_0^t\mathbb{E}\left[\Bigg\lVert\int_0^\cdot \sigma_rdW_r\Bigg\rVert_{\infty,u}^p\right]du\label{eq2:dim_bound_p_var}\\
        &\leq 2^pT^{p+2}M_\sigma^p\label{eq3:dim_bound_p_var},
    \end{align}
    where \eqref{eq1:dim_bound_p_var} is due  to Jensen's inequality; \eqref{eq2:dim_bound_p_var} is obtained by Jensen's inequality and Fubini's theorem; finally, we get \eqref{eq3:dim_bound_p_var} by the usual Burkholder-Davis-Gundy inequality and the boundedness of the diffusion coefficient $\sigma$ (see Hypothesis \ref{hypothesis_A}).

    To bound the last term $D$, we first need to rewrite the Stratonovich integral as an It\^{o} integral:
    \begin{align*}
        &\int_{t_i}^{t_{i+1}} (X_u^\ell - X_{t_i}^\ell)\circ dX_u^j = \frac{1}{2} \int_{t_i}^{t_{i+1}} \sum_{k} \sigma_u^{\ell,k}\sigma_u^{j,k} du + \int_{t_i}^{t_{i+1}} (X_u^\ell - X_{t_i}^\ell) dX_u^j\\
        &=\frac{1}{2} \int_{t_i}^{t_{i+1}} \sum_{k} \sigma_u^{\ell,k}\sigma_u^{j,k} du + \int_{t_i}^{t_{i+1}} \left(\int_{t_i}^u b_r^\ell\, dr\right) b_u^j du + \int_{t_i}^{t_{i+1}} \left(\sum_k \int_{t_i}^u \sigma_r^{\ell,k} dW_r^k \right) b_u^j du \\
        &\quad + \sum_{k} \int_{t_i}^{t_{i+1}} \left(\int_{t_i}^u b_r^\ell dr\right) \sigma_u^{j,k} dW_u^k \\
        &\quad + \sum_{k} \int_{t_i}^{t_{i+1}} \left(\sum_h \int_{t_i}^u \sigma_r^{\ell,h} dW_r^h \right)\sigma_u^{j,k} dW_u^k ,
    \end{align*}
    for every $\ell,j\in\{1,\ldots,d\}$.
    
    As for the first three terms, their boundedness can be proved using the same techniques used to bound term $C$. As for the fourth one, due to the boundedness of the drift coefficient $b$, one can use the same arguments used for term $B_2$. We only need to show that also the fifth term is bounded.
    
    By the Besov-variation embedding (Theorem \ref{thm:Besov_variation_embedding}) with $q=3/2$ and $1/\alpha=p/2$ with $p\in[5/2,3)$, for some constant $C=C(p,q)>0$,
    \begin{align*}
        &\mathbb{E}\left[\sup_D \left(\sum_{i=0}^{n-1}\Bigg\lVert\int_{t_i}^{t_{i+1}}\left(\int_{t_i}^u \sigma_rdW_r\right)^\top \sigma_u dW_u\Bigg\rVert_{\mathbb{R}^d}^{p/2}\right)^2\right]\\
        &\leq C^pt^{\frac{p(p/2-1)}{q}}\mathbb{E}\left[\left(\int\int_{[0,t]^2}\frac{\Big\lVert\int_u^v\left(\int_u^s\sigma_rdW_r\right)^\top  \sigma_s dW_s\Big\rVert_{\mathbb{R}^d}^q}{|u-v|^{1+2q/p}}dudv\right)^{p-q}\right]\\
        &\leq C^pC(t)M_\sigma^p\int\int_{[0,t]^2}\frac{\mathbb{E}\left[\int_u^v\Big\lVert\int_u^s\sigma_rdW_r\Big\rVert_{\mathbb{R}^d}^pds\right]}{|u-v|^{1+2q/p-(p/2-1)}}dudv,
    \end{align*}
    where the last inequality is due to the same computations as in term $B_2$ and $C(t):=t^{\frac{p(p/2-1)}{q}}t^{2(p-q-1)}$

    By Fubini's theorem, Burkholder-Davis-Gundy's inequality, Jensen's inequality, and the boundedness assumption on the diffusion coefficient $\sigma$,
    \begin{align*}
        &\mathbb{E}\left[\sup_D \left(\sum_{i=0}^{n-1}\Bigg\lVert\int_{t_i}^{t_{i+1}}\left(\int_{t_i}^u \sigma_rdW_r\right)^\top \sigma_u dW_u\Bigg\rVert_{\mathbb{R}^d}^{p/2}\right)^2\right]\\
        &\leq C^pC(t)M_\sigma^p\int\int_{[0,t]^2}\frac{\int_u^v\mathbb{E}\left[\Big\lVert\int_u^s\sigma_rdW_r\Big\rVert_{\mathbb{R}^d}^p\right]ds}{|u-v|^{1+2q/p-(p/2-1)}}dudv\\
        &\leq C^pC(t)M_\sigma^p\int\int_{[0,t]^2}\frac{\int_u^v\mathbb{E}\left[\left(\int_u^s\lVert\sigma_r\rVert_F^2dr\right)^{p/2}\right]ds}{|u-v|^{1+2q/p-(p/2-1)}}dudv\\
        &\leq C^pC(t)M_\sigma^p\int\int_{[0,t]^2}\frac{\int_u^v\mathbb{E}\left[\int_u^s\lVert\sigma_r\rVert_F^pdr\right]ds}{|u-v|^{1+2q/p-2(p/2-1)}}dudv\\
        &\leq C^pC(t)M_\sigma^{2p}\frac{1}{2}\int\int_{[0,t]^2}\frac{|u-v|^2}{|u-v|^{1+2q/p-2(p/2-1)}}dudv<\infty
    \end{align*}
    where the last integral is finite since $1+2q/p-p<0$ for $p\in[5/2,3)$.

\subsection{Proof of Proposition \ref{prop:p_var_norm_estimate_m=1}}\label{appendix:prop:p_var_norm_estimate_m=1_PROOF}
For any partition $D=D_{[0,t]}$ of $[0,t]$,
    \begin{align*}
        &\left(X^{\mu}_{t_{i+1}}-X^{\mu}_{t_i}\right)-\left(Y^{\nu}_{t_{i+1}}-Y^{\nu}_{t_i}\right)&&\\
        &\begin{aligned}
            = &\int_{t_i}^{t_{i+1}} \left[b\left(s,X^{\mu}_s,\Big\langle\ell,\mathbb{E}_{x\sim\mu}\left[S(x)_s\right]\Big\rangle\right)-b\left(s,Y^{\nu}_s,\Big\langle\ell,\mathbb{E}_{x\sim\nu}\left[S(x)_s\right]\Big\rangle\right)\right]ds\\
            &+ \int_{t_i}^{t_{i+1}} \left[\sigma\left(s,X^{\mu}_s,\Big\langle\ell,\mathbb{E}_{x\sim\mu}\left[S(x)_s\right]\Big\rangle\right)-\sigma\left(s,Y^{\nu}_s,\Big\langle\ell,\mathbb{E}_{x\sim\mu}\left[S(x)_s\right]\Big\rangle\right)\right]dW_s.
        \end{aligned}&&
    \end{align*}
    
    Therefore,
    \begin{align}
        &\mathbb{E}\left[\sup_{D}\sum_{i=0}^{n-1}\big\lVert \augpr{X}^{\mu,(1)}_{t_{i+1},t_i}-\augpr{Y}^{\nu,(1)}_{t_{i+1},t_i}\big\rVert_{\mathbb{R}^d}^p\right] \notag\\[0.5em]
        &\le 
        \begin{aligned}[t]
            &2^{p-1}\; \mathbb{E}\Bigg[\sup_{D}\sum_{i=0}^{n-1} 
                \Bigg\lVert \int_{t_i}^{t_{i+1}} 
                \Big[b\left(s,X^{\mu}_s,\Big\langle\ell,\mathbb{E}_{x\sim\mu}\left[S(x)_s\right]\Big\rangle\right)\\
            &\hspace{5cm} - b\left(s,Y^{\nu}_s,\Big\langle\ell,\mathbb{E}_{x\sim\nu}\left[S(x)_s\right]\Big\rangle\right)\Big] ds\Bigg\rVert_{\mathbb{R}^d}^p\Bigg] \\
            &+ 2^{p-1}\; \mathbb{E}\Bigg[\sup_{D}\sum_{i=0}^{n-1} 
                \Bigg\lVert \int_{t_i}^{t_{i+1}} 
                \Big[\sigma\left(s,X^{\mu}_s,\Big\langle\ell,\mathbb{E}_{x\sim\mu}\left[S(x)_s\right]\Big\rangle\right) \\
            &\hspace{5cm} - \sigma\left(s,Y^{\nu}_s,\Big\langle\ell,\mathbb{E}_{x\sim\nu}\left[S(x)_s\right]\Big\rangle\right)\Big] dW_s\Bigg\rVert_{\mathbb{R}^d}^p\Bigg]
        \end{aligned} \notag\\[0.5em]
        &\le 
        \begin{aligned}[t]
            &2^{p-1}T^{p-1}\; \mathbb{E}\Bigg[\int_0^t 
                \Big\lVert b\left(s,X^{\mu}_s,\Big\langle\ell,\mathbb{E}_{x\sim\mu}\left[S(x)_s\right]\Big\rangle\right) \\
            &\hspace{5cm} - b\left(s,Y^{\nu}_s,\Big\langle\ell,\mathbb{E}_{x\sim\nu}\left[S(x)_s\right]\Big\rangle\right)\Big\rVert_{\mathbb{R}^d}^p ds\Bigg] \\
            &+ 2^{p-1} \frac{p}{p-2}C_p\; \mathbb{E}\Bigg[\sup_{0\leq s\leq t}\Bigg\lVert\int_0^t 
                \Big(\sigma\left(s,X^{\mu}_s,\Big\langle\ell,\mathbb{E}_{x\sim\mu}\left[S(x)_s\right]\Big\rangle\right) \\
            &\hspace{5cm} - \sigma\left(s,Y^{\nu}_s,\Big\langle\ell,\mathbb{E}_{x\sim\nu}\left[S(x)_s\right]\Big\rangle\right)\Big)dW_s \Bigg\rVert_{\mathbb{R}^d}^p\Bigg],
        \end{aligned}\label{eq:terms_to_be_bounded}
    \end{align}
    where the first term in \eqref{eq:terms_to_be_bounded} follows from Jensen’s inequality, while the second term in \eqref{eq:terms_to_be_bounded} is due to L\'{e}pingle's inequality (Theorem \ref{thm:Lepingle}).
    
    As for the first term in \eqref{eq:terms_to_be_bounded}, by the Lipschitz assumption on the coefficient $b$ and the linearity of the expectation,
    \begin{align}
        &\Big\lVert b\left(s,X^{\mu}_s,\Big\langle\ell,\mathbb{E}_{x\sim\mu}\left[S(x)_s\right]\Big\rangle\right)
        -b\left(s,Y^{\nu}_s,\Big\langle\ell,\mathbb{E}_{x\sim\nu}\left[S(x)_s\right]\Big\rangle\right)\Big\rVert^p_{\mathbb{R}^d}\notag\\
        &\leq 2^{p-1}L^p\Big\lVert X_s^{\mu}-Y_s^{\nu}\Big\rVert^p_{\mathbb{R}^d}\notag\\
        &\quad +2^{p-1}L^p\Bigg|\sum_{0\leq |I|\leq N}\alpha_I\Big\langle\varepsilon_I,
        \int_{\widehat{G\Omega}_p(\mathbb{R}^{d+1})\times \widehat{G\Omega}_p(\mathbb{R}^{d+1})}(S(x)_s-S(y)_s)\pi(dx,dy)
        \Big\rangle\Bigg|^p \notag\\
        &\leq 2^{p-1}L^p\Big\lVert \augpr{X}^{\mu}-\augpr{Y}^{\nu}\Big\rVert^p_{\textup{$p$-var};[0,s]}\notag\\
        &\quad +2^{p-1}L^p\alpha_{\infty}^p\Bigg|
        \int_{\widehat{G\Omega}_p(\mathbb{R}^{d+1})\times \widehat{G\Omega}_p(\mathbb{R}^{d+1})}
        \sum_{0\leq |I|\leq N}\Big\langle\varepsilon_I,S(x)_s-S(y)_s\Big\rangle
        \pi(dx,dy)\Bigg|^p,
        \label{eq:first_term_1}
\end{align}
    where $\alpha_{\infty}:=\max_{0\leq |I|\leq N}|\alpha_I|$ and $\pi\in\Pi(\mu,\nu)$ is an arbitrary coupling of $\mu$ and $\nu$.
    
    By the Cauchy-Schwartz inequality,
    \begin{equation}\label{eq:consequence_C_S_ineq}
         |x_1|+|x_2|+\dots+|x_n| \leq \sqrt{n}\left(\sum_{i=1}^n x_i^2\right)^{1/2}
     \end{equation}
     for every $(x_1,\dots,x_n)\in\mathbb{R}^n$ and $n\geq 0$. Let
    \begin{equation*}
        C_1(d,N):= \left(\frac{(d+1)^{N+1}-1}{d}\right)^{1/2}.
    \end{equation*} Then, by \eqref{eq:consequence_C_S_ineq}, \eqref{eq:first_term_1} becomes
    \begin{align*}
        &\Big\lVert b\left(s,X^{\mu}_s,\Big\langle\ell,\mathbb{E}_{x\sim\mu}\left[S(x)_s\right]\Big\rangle\right)-b\left(s,Y^{\nu}_s,\Big\langle\ell,\mathbb{E}_{x\sim\nu}\left[S(y)_s\right]\Big\rangle\right)\Big\rVert^p_{\mathbb{R}^d}&&\\
        &\leq 2^{p-1}L^p\Big\lVert \augpr{X}^{\mu}-\augpr{Y}^{\nu}\Big\rVert^p_{\textup{$p$-var};[0,s]}\\
        &\quad+ 2^{p-1}L^p\alpha_{\infty}^pC_1(d,N)^p\int_{\widehat{G\Omega}_p(\mathbb{R}^{d+1})\times \widehat{G\Omega}_p(\mathbb{R}^{d+1})}\Big\lVert S(x)^{\leq N}_s - S^{\leq N}(y)_s\Big\rVert_{T^N(\mathbb{R}^{d+1})}^p\pi(dx,dy),&&
    \end{align*}
    where
    \begin{equation*}
        \Big\lVert S(x)^{\leq N}_s - S(y)^{\leq N}_s\Big\rVert_{T^N(\mathbb{R}^{d+1})} := \sum_{k=0}^N\big\lVert S(x)^k_s - S(y)^k_s \big\rVert_{\mathbb{R}^{(d+1)\otimes k}}.
    \end{equation*}
    
    By Corollary \ref{corollary:local_Lip},
    \begin{equation*}
        \begin{aligned}
            &\int_{\widehat{G\Omega}_p(\mathbb{R}^{d+1})\times \widehat{G\Omega}_p(\mathbb{R}^{d+1})}\Big\lVert S(x)^{\leq N}_s - S^{\leq N}(y)_s\Big\rVert_{T^N(\mathbb{R}^{d+1})}^p\pi(dx,dy)\\
            &\leq\int_{\widehat{G\Omega}_p(\mathbb{R}^{d+1})\times \widehat{G\Omega}_p(\mathbb{R}^{d+1})}C_2(p,N,\Gamma_s)\lVert x-y \rVert_{\textup{$p$-var};[0,s]}^p\pi(dx,dy),
        \end{aligned}
    \end{equation*}
    where $\Gamma_s:=\max\{\lVert x\rVert_{\textup{$p$-var};[0,s]},\lVert y\rVert_{\textup{$p$-var};[0,s]}\}$.
    
    Let
    \begin{equation*}
        C:=\alpha_{\infty}^pC_1(d,N)^pC_2(p,N,\Gamma_s).
    \end{equation*}
    We have obtained
    \begin{equation}\label{eq:first_continuity_result_eq1}
        \begin{aligned}
            &\Big\lVert b\left(s,X^{\mu}_s,\Big\langle\ell,\mathbb{E}\left[S^{\leq M}(\augpr{X}^{\mu})_s\right]\Big\rangle\right)-b\left(s,Y^{\nu}_s,\Big\langle\ell,\mathbb{E}\left[S^{\leq M}(\augpr{Y}^{\nu})_s\right]\Big\rangle\right)\Big\rVert^p_{\mathbb{R}^d}&&\\
            &\leq2^{p-1}L^p\Big\lVert \augpr{X}^{\mu}-\augpr{Y}^{\nu}\Big\rVert^p_{\textup{$p$-var};[0,s]} + 2^{p-1}L^p\int_{\widehat{G\Omega}_p(\mathbb{R}^{d+1})\times \widehat{G\Omega}_p(\mathbb{R}^{d+1})}C\lVert x-y \rVert_{\textup{$p$-var};[0,s]}^p\pi(dx,dy).
        \end{aligned}
    \end{equation}

    As for the second term in \eqref{eq:terms_to_be_bounded}, 
    \begin{align}
        &\mathbb{E}\Bigg[\sup_{0\leq s\leq t}\Bigg\lVert\int_0^t \Big(\sigma\left(s,X^{\mu}_s,\Big\langle\ell,\mathbb{E}_{x\sim\mu}\left[S(x)_s\right]\Big\rangle\right) - \sigma\left(s,Y^{\nu}_s,\Big\langle\ell,\mathbb{E}_{x\sim\nu}\left[S(x)_s\right]\Big\rangle\right)\Big)dW_s \Bigg\rVert_{\mathbb{R}^d}^p\Bigg]\notag\\
        &\leq\widetilde{C}_p\mathbb{E}\Bigg[\Bigg(\int_0^t \Big\lVert \sigma\left(s,X^{\mu}_s,\Big\langle\ell,\mathbb{E}_{x\sim\mu}\left[S(x)_s\right]\Big\rangle\right) - \sigma\left(s,Y^{\nu}_s,\Big\langle\ell,\mathbb{E}_{x\sim\nu}\left[S(x)_s\right]\Big\rangle\right)\Big\rVert_F^2 ds \Bigg)^{p/2}\Bigg]\label{eq:ineq_well_pos_SDE_2}\\
        &\leq \widetilde{C}_pT^{p-1}\mathbb{E}\Bigg[\int_0^t \Big\lVert \sigma\left(s,X^{\mu}_s,\mathbb{E}_{x\sim\mu}\left[S(x)_s\right]\right)-\sigma\left(s,Y^{\nu}_s,\mathbb{E}_{x\sim\nu}\left[S(x)_s\right]\right)\Big\rVert_F^pds\Bigg],\label{eq:ineq_well_pos_SDE_3}
    \end{align}
    where \eqref{eq:ineq_well_pos_SDE_2} is obtained by the multidimensional Burkholder–Davis–Gundy inequality (see, e.g., \cite[Theorem 3.28]{KaratzasShreve}), while the bound in \eqref{eq:ineq_well_pos_SDE_3} follows from Jensen’s inequality. Then, by doing the same computations as for the first term,
    \begin{equation}\label{eq:first_continuity_result_eq2}
        \begin{aligned}
            &\Big\lVert\sigma\left(s,X^{\mu}_s,\Big\langle\ell,\mathbb{E}\left[S(\augpr{X}^{\mu})_s\right]\Big\rangle\right)-\sigma\left(s,Y^{\nu}_s,\Big\langle\ell,\mathbb{E}\left[S(\augpr{Y}^{\nu})_s\right]\Big\rangle\right)\Big\rVert_{\mathbb{R}^d}^p&&\\
            &\leq2^{p-1}L^p\Big\lVert \augpr{X}^{\mu}-\augpr{Y}^{\nu}\Big\rVert^p_{\textup{$p$-var};[0,s]} + 2^{p-1}L^p\int_{\widehat{G\Omega}_p(\mathbb{R}^{d+1})\times \widehat{G\Omega}_p(\mathbb{R}^{d+1})}C\lVert x-y \rVert_{\textup{$p$-var};[0,s]}^p\pi(dx,dy).
        \end{aligned}
    \end{equation}

    By substituting \eqref{eq:first_continuity_result_eq1} and \eqref{eq:first_continuity_result_eq2} in \eqref{eq:terms_to_be_bounded}, we conclude.

\subsection{Proof of Proposition \ref{prop:p_var_norm_estimate_m=2}}\label{appendix:prop:p_var_norm_estimate_m=2_PROOF}
    Let $Z:=X^{\mu}-Y^{\nu}$. Then, for any partition $D=D_{[0,t]}$ of $[0,t]$,
    \begin{equation*}
        \augpr{X}^{\mu,(2)}_{t_{i+1},t_i}-\augpr{Y}^{\nu,(2)}_{t_{i+1},t_i} 
    \end{equation*}
    is the $\mathbb{R}^d$-valued $(d+1)\times (d+1)$ matrix
    \begin{equation*}
        \begin{bmatrix}
        0 & \int_{t_i}^{t_{i+1}} (u-t_i)\circ dZ_u^\top\\
        \int_{t_i}^{t_{i+1}} (Z_u-Z_{t_i})du & \int_{t_i}^{t_{i+1}} (Z_u-Z_{t_i})\circ dX^{\mu\top}_u + \int_{t_i}^{t_{i+1}} (Y^{\nu}_u-Y^{\nu}_{t_i})\circ dZ_u^\top
        \end{bmatrix}.
    \end{equation*}
    In fact, by adding and subtracting
    \begin{equation*}
        \int_{t_i}^{t_{i+1}}(Y^{\nu}_u-Y^{\nu}_{t_i})\circ dX^{\mu\top}_u,
    \end{equation*}
    one obtains that
    \begin{equation*}
        \begin{aligned}
            &\int_{t_i}^{t_{i+1}}(X^{\mu}_u-X^{\mu}_{t_i})\circ dX^{\mu\top}_u - \int_{t_i}^{t_{i+1}}(Y^{\nu}_u-Y^{\nu}_{t_i})\circ dY^{\nu\top}_u\\
            &=\int_{t_i}^{t_{i+1}} (Z_u-Z_{t_i})\circ dX^{\mu\top}_u + \int_{t_i}^{t_{i+1}} (Y^{\nu}_u-Y^{\nu}_{t_i})\circ dZ_u^\top.
        \end{aligned}
    \end{equation*}
    
    Therefore,
    \begin{align*}
        &\mathbb{E}\left[\sup_D\left(\sum_{i=0}^{n-1} \Big\lVert\augpr{X}^{\mu,(2)}_{t_{i+1},t_i}-\augpr{Y}^{\nu,(2)}_{t_{i+1},t_i}\Big\rVert^{p/2}\right)^2\right]\\
        &\begin{aligned}
            \,\leq\,&4^{p-2}\mathbb{E}\left[\sup_D \left(\sum_{i=0}^{n-1}\Bigg\lVert\int_{t_i}^{t_{i+1}}(u-t_i)\circ dZ_u^\top\Bigg\rVert_{\mathbb{R}^d}^{p/2}\right)^2\right]&\quad:=A\\
            &+4^{p-2}\mathbb{E}\left[\sup_D \left(\sum_{i=0}^{n-1}\Bigg\lVert\int_{t_i}^{t_{i+1}}(Z_u-Z_{t_i})du\Bigg\rVert_{\mathbb{R}^d}^{p/2}\right)^2\right]&\quad:=B\\
            &+4^{p-2}\mathbb{E}\left[\sup_D \left(\sum_{i=0}^{n-1}\Bigg\lVert\int_{t_i}^{t_{i+1}}(Z_u-Z_{t_i})\circ dX^{\mu\top}_u\Bigg\rVert_{\mathbb{R}^d}^{p/2}\right)^2\right]&\quad:=C\\
            &+4^{p-2}\mathbb{E}\left[\sup_D \left(\sum_{i=0}^{n-1}\Bigg\lVert\int_{t_i}^{t_{i+1}}(Y^{\nu}_u-Y^{\nu}_{t_i})\circ dZ_u^\top\Bigg\rVert_{\mathbb{R}^d}^{p/2}\right)^2\right]&\quad:=D.
        \end{aligned}
    \end{align*}

    As the first term $A$ concerns, for almost every $\omega\in\Omega$, by Young--L\'{o}eve's inequality (Theorem \ref{thm:Young_Loeve}),
    \begin{equation*}
        \begin{aligned}
            &\sup_D \left(\sum_{i=0}^{n-1}\Bigg\lVert\int_{t_i}^{t_{i+1}}(u-t_i)\circ dZ_u(\omega)^\top\Bigg\rVert_{\mathbb{R}^d}^{p/2}\right)^2\\
            &\leq\frac{2^{1/p}}{2^{1/p}-1}\sup_D \left(\sum_{i=0}^{n-1} \lVert \textup{Id}-t_i\rVert_{1-var;[t_i,t_{i+1}]}^{p/2}\lVert Z(\omega)\rVert_{\textup{$p$-var};[t_i,t_{i+1}]}^{p/2}\right)^2\\
            &\leq\frac{2^{1/p}}{2^{1/p}-1}\lVert Z(\omega)\rVert_{\textup{$p$-var};[0,t]}^p\sup_D \left(\sum_{i=0}^{n-1} \lVert \textup{Id}-t_i\rVert_{1-var;[t_i,t_{i+1}]}^{p/2}\right)^2\\
            &\leq\frac{2^{1/p}}{2^{1/p}-1}T^p\lVert Z(\omega)\rVert_{\textup{$p$-var};[0,t]}^p,
        \end{aligned}
    \end{equation*}
    where $\textup{Id}$ is the identity function.

    By Proposition \ref{prop:p_var_norm_estimate_m=1},
    \begin{equation*}
        \begin{aligned}
            A&\leq\frac{2^{1/p}}{2^{1/p}-1}T^p\mathbb{E}\left[\Big\lVert X^\mu-Y^\nu \Big\rVert_{\textup{$p$-var};[0,t]}^p\right]\\
            &\leq \widetilde{C}\int_0^t\left(\mathbb{E}\left[\big\lVert\augpr{X}^{\mu}-\augpr{Y}^{\nu}\big\rVert_{\textup{$p$-var};[0,s]}^p\right] + \int_{\widehat{G\Omega}_p(\mathbb{R}^{d+1})\times \widehat{G\Omega}_p(\mathbb{R}^{d+1})}C_s\lVert x-y \rVert_{\textup{$p$-var};[0,s]}^p\pi(dx,dy)\right)ds,
        \end{aligned}
    \end{equation*}
    where $\widetilde{C}=\widetilde{C}(p,T,L)$ and $C_s$ as in Proposition \ref{prop:p_var_norm_estimate_m=1}, and $\pi\in\Pi(\mu,\nu)$ is a coupling of $\mu$ and $\nu$. Here, the $p$-variation norm in the first line is the one over $C([0,T];\mathbb{R}^d)$, while the one in the second line is the inhomogeneous $p$-variation norm over the space of geometric $p$-rough paths.
    
    Regarding the second term $B$, by Jensen's inequality and Fubini's theorem,
    \begin{equation*}
        \begin{aligned}
            B&\leq \mathbb{E}\left[\left(\sup_D \sum_{i=0}^{n-1} \int_{t_i}^{t_{i+1}} \big\lVert Z_u-Z_{t_i}\big\rVert_{\mathbb{R}^d}^{p/2}du\right)^2\right]\\
            &\leq \mathbb{E}\left[\left(\sup_D \sum_{i=0}^{n-1} \int_{t_i}^{t_{i+1}} \left(2^{p/2-1}\big\lVert X_u^\mu-Y_u^\nu\big\rVert_{\mathbb{R}^d}^{p/2}+2^{p/2-1}\big\lVert X_{t_i}^\mu-Y_{t_i}^\nu\big\rVert_{\mathbb{R}^d}^{p/2}\right)du\right)^2\right]\\
            &\leq 2^{p-2}\mathbb{E}\left[\left(\int_0^t \big\lVert\augpr{X}^{\mu} - \augpr{Y}^{\nu}\big\rVert_{\textup{$p$-var};[0,s]}^{p/2}ds\right)^2\right] \leq 2^{p-2}\int_0^t \mathbb{E}\left[\big\lVert\augpr{X}^{\mu} - \augpr{Y}^{\nu}\big\rVert_{\textup{$p$-var};[0,s]}^p\right]ds.
        \end{aligned}
    \end{equation*}

    As for the third term $C$, we start by rewriting the Stratonovich integral as an It\^{o} integral:
    \begin{align}
        &\mathbb{E}\left[\sup_D \left(\sum_{i=0}^{n-1}\Bigg\lVert\int_{t_i}^{t_{i+1}}(Z_u-Z_{t_i})\circ dX^{\mu\top}_u\Bigg\rVert_{\mathbb{R}^d}^{p/2}\right)^2\right]\\\notag
        &\begin{aligned}
            =&2^{p-1}\mathbb{E}\left[\sup_D \left(\sum_{i=0}^{n-1}\Bigg\lVert\int_{t_i}^{t_{i+1}}(Z_u-Z_{t_i})dX^{\mu\top}_u\Bigg\rVert_{\mathbb{R}^d}^{p/2}\right)^2\right]\\
            &+\frac{2^{p-1}}{2}\mathbb{E}\left[\sup_D \left(\sum_{i=0}^{n-1}\Bigg\lVert\int_{t_i}^{t_{i+1}}\left(\sigma_u^\mu-\sigma_u^\nu\right)\sigma_u^{\mu\top} du\Bigg\rVert_{\mathbb{R}^d}^{p/2}\right)^2\right]=:C_1 + C_2,
        \end{aligned}\label{eq:terms_2_bound_term_c}
    \end{align}
    where
    \begin{equation*}
        \sigma_u^\mu := \sigma\left(u,X_u^\mu,\Big\langle\ell,\mathbb{E}_{x\sim\mu}[S(x)_u]\Big\rangle\right);\qquad \sigma_u^\nu := \sigma\left(u,Y_u^\nu,\Big\langle\ell,\mathbb{E}_{y\sim\nu}[S(y)_u]\Big\rangle\right),
    \end{equation*}
    for every $u\in[0,T]$; similarly, we also define $b_u^\mu$ and $b_u^\nu$.
    
     Then, since the norm is unchanged by transposition,
     \begin{align*}
         C_1 \leq &2^{2p-2}\mathbb{E}\left[\sup_D \left(\sum_{i=0}^{n-1}\Bigg\lVert\int_{t_i}^{t_{i+1}}(Z_u-Z_{t_i})^\top b_u^\mu du\Bigg\rVert_{\mathbb{R}^d}^{p/2}\right)^2\right]\\
         &+ 2^{2p-2}\mathbb{E}\left[\sup_D \left(\sum_{i=0}^{n-1}\Bigg\lVert\int_{t_i}^{t_{i+1}}(Z_u-Z_{t_i})^\top \sigma_u^\mu dW_u\Bigg\rVert_{\mathbb{R}^d}^{p/2}\right)^2\right]=:C_{1,1} + C_{1,2}
     \end{align*}
     and by Jensen's inequality, the boundedness assumption on the coefficient $b$ (see Hypothesis \ref{hypothesis_A}), and Tonelli's theorem,
    \begin{align*}
        C_{1,1} \leq &2^{2p-2}\mathbb{E}\left[\sup_D \left(\sum_{i=0}^{n-1}\int_{t_i}^{t_{i+1}}\big\lVert Z_u-Z_{t_i}\big\rVert_{\mathbb{R}^d}^{p/2}\lVert b_u^\mu\rVert_{\mathbb{R}^d}^{p/2}du\right)^2\right]\\
        &\leq 2^{2p-2}M_b^p\mathbb{E}\left[\sup_D \left(\sum_{i=0}^{n-1}\int_{t_i}^{t_{i+1}}\big\lVert Z_u-Z_{t_i}\big\rVert_{\mathbb{R}^d}^{p/2}du\right)^2\right]\\
        &\leq 2^{p-2}M_b^p\int_0^t \mathbb{E}\left[\big\lVert\augpr{X}^{\mu} - \augpr{Y}^{\nu}\big\rVert_{\textup{$p$-var};[0,s]}^p\right]ds,
    \end{align*}
    where the last inequality is obtained by the same computations done for the term $B$.

    Similarly, one obtains,
    \begin{align*}
        C_2 \leq &\frac{2^{2p-2}}{2}M_\sigma^p\int_0^t \mathbb{E}\left[\Big\lVert \sigma_u^\mu-\sigma_u^\nu\Big\rVert_F^p\right]du\\
        &\leq \widetilde{C}\int_0^t\left(\mathbb{E}\left[\big\lVert\augpr{X}^{\mu}-\augpr{Y}^{\nu}\big\rVert_{\textup{$p$-var};[0,s]}^p\right] + \mathbb{E}_{(x,y)\sim(\mu,\nu)}\left[C_s\lVert x-y \rVert_{\textup{$p$-var};[0,s]}^p\right]\right)ds,
    \end{align*}
    where the last inequality is due to Proposition \ref{prop:p_var_norm_estimate_m=1}, with $\widetilde{C}=\widetilde{C}(p,T,L,M_\sigma)$ and $C_s$ as in the proposition.

    Lastly,
    \begin{align*}
        C_{1,2} \leq &2^{p-1}\mathbb{E}\left[\sup_D \left(\sum_{i=0}^{n-1}\Bigg\lVert\int_{t_i}^{t_{i+1}}\left(\int_{t_i}^u \left(b_r^\mu-b_r^\nu\right)dr\right)^\top \sigma_u^\mu dW_u\Bigg\rVert_{\mathbb{R}^d}^{p/2}\right)^2\right]\\
        &+2^{p-1}\mathbb{E}\left[\sup_D \left(\sum_{i=0}^{n-1}\Bigg\lVert\int_{t_i}^{t_{i+1}}\left(\int_{t_i}^u \left(\sigma_r^\mu-\sigma_r^\nu\right)dW_r\right)^\top \sigma_u^\mu dW_u\Bigg\rVert_{\mathbb{R}^d}^{p/2}\right)^2\right] =: C_{1,2,1} + C_{1,2,2}.
    \end{align*}

    We carry out the explicit computations only for the term $C_{1,2,2}$, since those for $C_{1,2,1}$ are analogous. Moreover, the same arguments used for $C_{1,2,2}$ also yield the desired estimate for the term $D$.

    Let $q=3/2$. By the Besov-variation embedding (Theorem \ref{thm:Besov_variation_embedding}), for some constant $C=C(p,q)>0$
    \begin{align*}
        C_{1,2,2} &= 2^{p-1}\mathbb{E}\left[\sup_D \left(\sum_{i=0}^{n-1}\Bigg\lVert\int_{t_i}^{t_{i+1}}\left(\int_{t_i}^u \left(\sigma_r^\mu-\sigma_r^\nu\right)dW_r\right)^\top \sigma_u^\mu dW_u\Bigg\rVert_{\mathbb{R}^d}^{p/2}\right)^2\right]\\
        &\leq C^pt^{\frac{p(p/2-1)}{q}}\mathbb{E}\left[\left(\int\int_{[0,t]^2}\frac{\Big\lVert\int_u^v\left(\int_u^s\left(\sigma_r^\mu-\sigma_r^\nu\right)dW_r\right)^\top  \sigma_s^\mu dW_s\Big\rVert_{\mathbb{R}^d}^q}{|u-v|^{1+2q/p}}dudv\right)^{p-q}\right]\\
        &\leq C^pt^{\frac{p(p/2-1)}{q}}t^{2(p-q-1)}\mathbb{E}\left[\int\int_{[0,t]^2}\frac{\Big\lVert\int_u^v\left(\int_u^s\left(\sigma_r^\mu-\sigma_r^\nu\right)dW_r\right)^\top \sigma_s^\mu dW_s\Big\rVert_{\mathbb{R}^d}^p}{|u-v|^{1+2q/p}}dudv\right],
    \end{align*}
    where the last inequality is due to Jensen's theorem since $p-q \geq 1$.

    Now, let $C(t):=t^{\frac{p(p/2-1)}{q}}t^{2(p-q-1)}$. By Tonelli's theorem,
    \begin{align*}
        C_{1,2,2} &\leq C^pC(t)\int\int_{[0,t]^2}\frac{\mathbb{E}\left[\Big\lVert\int_u^v\left(\int_u^s\left(\sigma_r^\mu-\sigma_r^\nu\right)dW_r\right)^\top \sigma_s^\mu dW_s\Big\rVert_{\mathbb{R}^d}^p\right]}{|u-v|^{1+2q/p}}dudv\\
        &\leq C^pC(t)\int\int_{[0,t]^2}\frac{\mathbb{E}\left[\left(\int_u^v\Big\lVert\int_u^s\left(\sigma_r^\mu-\sigma_r^\nu\right)dW_r\Big\rVert_{\mathbb{R}^d}^2\big\lVert \sigma_s^\mu\big\rVert_F^2ds\right)^{p/2}\right]}{|u-v|^{1+2q/p}}dudv\\
        &\leq C^pC(t)\int\int_{[0,t]^2}\frac{\mathbb{E}\left[\int_u^v\Big\lVert\int_u^s\left(\sigma_r^\mu-\sigma_r^\nu\right)dW_r\Big\rVert_{\mathbb{R}^d}^p\big\lVert \sigma_s^\mu\big\rVert_F^pds\right]}{|u-v|^{1+2q/p-(p/2-1)}}dudv\\
        &\leq C^pC(t)M_\sigma^p\int\int_{[0,t]^2}\frac{\mathbb{E}\left[\int_u^v\Big\lVert\int_u^s\left(\sigma_r^\mu-\sigma_r^\nu\right)dW_r\Big\rVert_{\mathbb{R}^d}^pds\right]}{|u-v|^{1+2q/p-(p/2-1)}}dudv,
    \end{align*}
    where we have applied the usual Burkholder-Davis-Gundy inequality (see, e.g., \cite[Theorem 3.28]{KaratzasShreve}), Jensen's inequality since $p/2\geq 1$, and the boundedness assumption on the diffusion coefficient $\sigma$ (see Hypothesis \ref{hypothesis_A}).

    Finally, by applying Fubini's theorem and the Burkholder-Davis-Gundy and Jensen inequalities,
    \begin{align*}
        C_{1,2,2} &\leq C^pC(t)M_\sigma^p\int\int_{[0,t]^2}\frac{\int_u^v\mathbb{E}\left[\Big\lVert\int_u^s\left(\sigma_r^\mu-\sigma_r^\nu\right)dW_r\Big\rVert_{\mathbb{R}^d}^p\right]ds}{|u-v|^{1+2q/p-(p/2-1)}}dudv\\
        &\leq C^pC(t)M_\sigma^p\int\int_{[0,t]^2}\frac{\int_u^v\mathbb{E}\left[\left(\int_u^s\lVert\sigma_r^\mu-\sigma_r^\nu\rVert_F^2dr\right)^{p/2}\right]ds}{|u-v|^{1+2q/p-(p/2-1)}}dudv\\
        &\leq C^pC(t)M_\sigma^p\int\int_{[0,t]^2}\frac{\int_u^v\mathbb{E}\left[\int_u^s\lVert\sigma_r^\mu-\sigma_r^\nu\rVert_F^pdr\right]ds}{|u-v|^{1+2q/p-2(p/2-1)}}dudv\\
        &\leq C^pC(t)M_\sigma^p\left(\int\int_{[0,t]^2}\frac{1}{|u-v|^{1+2q/p-2(p/2-1)-1}}dudv\right)\int_0^t\mathbb{E}\left[\lVert\sigma_r^\mu-\sigma_r^\nu\rVert_F^pdr\right]ds\\
        &\leq C^pC(t)M_\sigma^p\int_0^t\mathbb{E}\left[\lVert\sigma_r^\mu-\sigma_r^\nu\rVert_F^pdr\right]ds,
    \end{align*}
    where the last inequality is due to the fact that the integral is convergent since $2q/p-p+2\in(0,1)$  for $p\in[5/2,3)$.

    Thus, by Proposition \ref{prop:p_var_norm_estimate_m=1},
    \begin{align*}
        C_{1,2,2}\leq \widetilde{C}\int_0^t\left(\mathbb{E}\left[\big\lVert\augpr{X}^{\mu}-\augpr{Y}^{\nu}\big\rVert_{\textup{$p$-var};[0,s]}^p\right] + \int_{\widehat{G\Omega}_p(\mathbb{R}^{d+1})\times \widehat{G\Omega}_p(\mathbb{R}^{d+1})}C_s\lVert x-y \rVert_{\textup{$p$-var};[0,s]}^p\pi(dx,dy)\right)ds,
    \end{align*}
    where $\widetilde{C}=\widetilde{C}(p,T,L,M_\sigma)$ and $C_s$ as in the proposition.
    
\printbibliography

\end{document}